\documentclass[a4paper,12pt]{article}

\usepackage[usenames]{xcolor} 
\usepackage{amssymb} 
\usepackage{mathtools} 
\usepackage{amsfonts}
\usepackage{amsthm}
\usepackage[utf8]{inputenc} 
\usepackage[british]{babel}
\usepackage{array} 
\usepackage{multirow} 
\usepackage{float} 
\usepackage[margin=1.3in]{geometry}
\usepackage{multirow}
\usepackage{enumerate}
\usepackage{parskip}
\usepackage[all,cmtip]{xy}
\usepackage{tikz}
\usetikzlibrary{matrix,arrows.meta,positioning}
\usepackage{arydshln}
\usepackage{comment}
\usepackage{setspace}

\DeclareMathOperator{\Sel}{Sel}
\DeclareMathOperator{\Gal}{Gal}

\DeclareMathOperator{\res}{res}
\DeclareMathOperator{\ind}{ind}

\renewcommand{\Im}{\textrm{Im}}
\newcommand{\Frob}{\textrm{Frob}}
\newcommand{\Fitt}{\textrm{Fitt}}

\DeclareMathOperator{\Hom}{Hom}
\DeclareMathOperator{\Ext}{Ext}

\DeclareMathOperator{\Aut}{Aut}
\DeclareMathOperator{\KS}{KS}

\newcommand{\loc}{\textrm{loc}}
\newcommand{\m}{\mathfrak{m}}
\newcommand{\p}{\mathfrak{p}}

\newcommand{\coker}{\textrm{coker}}

\newcommand{\rank}{\textrm{rank}}

\newcommand{\Reg}{\textrm{Reg}}

\newcommand{\Q}{\mathbb Q}
\newcommand{\R}{\mathbb R}

\newcommand{\Z}{\mathbb Z}
\newcommand{\F}{\mathbb F}
\newcommand{\N}{\mathbb N}
\newcommand{\FF}{\mathcal F}

\newcommand{\FBK}{{\mathcal{F}_{\textrm{BK}}}}
\newcommand{\FGR}{{\mathcal{F}_{\textrm{Gr}}}}
\newcommand{\PP}{\mathcal P}
\newcommand{\NN}{\mathcal N}

\newcommand{\GG}{\mathcal G}

\newcommand{\II}{\mathcal I}
\newcommand{\UU}{\mathcal U}

\newcommand{\BB}{\mathcal B}
\newcommand{\HH}{\mathcal H}

\newcommand{\MM}{\mathcal M}

\renewcommand{\AA}{\mathcal A}

\newcommand{\tr}{\textrm{tr}}
\newcommand{\ur}{\textrm{ur}}
\newcommand{\f}{\textrm{f}}
\newcommand{\s}{\textrm{s}}

\newcommand{\fs}{\textrm{fs}}
\newcommand{\cl}{\textrm{cl}}
\renewcommand{\f}{\textrm{f}}

\newcommand{\GL}{\textrm{GL}}

\newcommand{\Kato}{\textrm{Kato}}
\newcommand{\LLZ}{\textrm{LLZ}}
\newcommand{\tors}{\textrm{tors}}

\newcommand{\bH}{\textbf{H}}
\newcommand{\bSS}{\textbf{SS}}
\newcommand{\bKS}{\textbf{KS}}

\newcommand{\bT}{\textbf{T}}

\newcommand{\bz}{\textbf{z}}

\newcommand{\bTheta}{\mathbf{\Theta}}

\newcommand{\ku}{\mathfrak u}
\newcommand{\kn}{\mathfrak n}

\newcommand{\ka}{\mathfrak a}
\newcommand{\kb}{\mathfrak b}
\newcommand{\kc}{\mathfrak c}
\newcommand{\kq}{\mathfrak q}
\newcommand{\kr}{\mathfrak r}
\newcommand{\ks}{\mathfrak s}

\renewcommand{\Im}{\textrm{Im}}
\newcommand{\sign}{\textrm{sign}}

    \DeclareFontFamily{U}{wncy}{}
    \DeclareFontShape{U}{wncy}{m}{n}{<->wncyr10}{}
    \DeclareSymbolFont{mcy}{U}{wncy}{m}{n}
\DeclareMathSymbol{\Sha}{\mathord}{mcy}{"58}

\theoremstyle{definition}
\newtheorem{theorem}{Theorem}[section]
\newtheorem{theorem*}{Theorem}[section]
\newtheorem{definition}[theorem]{Definition}

\newtheorem{remark}[theorem]{Remark}
\newtheorem{remark*}[theorem*]{Remark}
\newtheorem{proposition}[theorem]{Proposition}
\newtheorem{corollary}[theorem]{Corollary}
\newtheorem{corollary*}[theorem*]{Corollary}
\newtheorem{lemma}[theorem]{Lemma}

\newtheorem{notation}[theorem]{Notation}

\newtheorem{innercustomass}{Assumption}

\newenvironment{namedass}[1]
  {\renewcommand\theinnercustomass{#1}\innercustomass}
  {\endinnercustomass}

\makeatletter
\def\namedlabel#1#2{\begingroup
    #2%
    \def\@currentlabel{#2}%
    \phantomsection\label{#1}\endgroup
}
\makeatother

\usepackage[backend=biber, style=alphabetic]{biblatex}

\usepackage{csquotes}

\usepackage[pdfencoding=auto,hidelinks]{hyperref} 

\begin{document}

\title{Ultra Kolyvagin systems and higher Fitting ideals of Iwasawa Selmer groups}
\author{Alberto Angurel}
\maketitle

\begin{abstract}
We develop the theory of equivariant, ultra Kolyvagin systems to bypass structural limitations of the Euler system machinery over infinite rings. By utilizing collections of classes living in the exterior powers of patched Selmer groups---constructed from ultraproducts of classical Selmer groups---we compute the structure of an Iwasawa Selmer group up to pseudo-isomorphism of Iwasawa modules and prove the absence of finite submodules. We apply this theoretical framework to the fine Selmer group of an elliptic curve and the Bloch-Kato Selmer group of the Rankin-Selberg convolution of modular forms.
\end{abstract}

\tableofcontents

\section{Introduction}

The method of Euler systems is a cornerstone of Iwasawa theory that has led to considerable progress towards the proof of important conjectures in arithmetic geometry. Arguably, its most important consequence is the proof of the Birch and Swinnerton-Dyer conjecture for elliptic curves whose analytic rank is at most one. Furthermore, this method has also been applied in greater generality to prove the Bloch-Kato conjecture and one divisibility of the Iwasawa main conjecture in other settings.

Historically, this argument was introduced in \cite{Thaine}, \cite{Kolyvagin_first}, \cite{Kolyvagin}, \cite{PerrinRiou98}, \cite{Kato99} and \cite{Rubin}, where Selmer groups were studied using collections of cohomology classes, known as Euler systems, satisfying certain norm-compatibility relations with a similar flavour to the Euler factors used to define the $L$-function. After that, Mazur and Rubin axiomatized in \cite{MazurRubin} the ideas at the core of this argument by introducing the notion of Kolyvagin systems.

In this article, we generalise the theory of Kolyvagin systems to prove refinements of the Iwasawa main conjectures. In addition to the equality between the $p$-adic $L$-function and the characteristic ideal of the dual Selmer group predicted by this conjecture, we are able to describe its full structure up to pseudo-isomorphism of Iwasawa modules in terms of higher Stickelberger ideals, similar to those appearing in \cite{Kur2014}. Although our methods can be applied in wider generality, we give explicit computations for the Selmer groups associated with elliptic curves and Rankin-Selberg convolutions of modular forms.

Our work builds on further generalisations of the theory of Kolyvagin systems in \cite{MazurRubin16}, \cite{BurnsSano}, \cite{Sakamoto18}, \cite{BurnsSakamotoSano2} and \cite{BurnsSakamotoSano3}. They extended the results in \cite{MazurRubin} in two different directions: Selmer groups with higher core rank and those with non-principal coefficient rings, allowing an equivariant generalisation of the theory.

Here we develop a further generalisation of the theory of Kolyvagin systems that gives a full description of the Selmer group at the top of a $\Z_p$-extension. In order to do this, we develop an equivariant theory of Kolyvagin systems of higher rank that combines previous methods in \cite{BurnsSakamotoSano2} and \cite{Sweeting}.

We view Selmer groups as subgroups of the patched cohomology, originally constructed in \cite{Sweeting} as an ultraproduct of classical Galois cohomology groups. In this setting, the notion of prime number is relaxed, leading to the concept of ultraprimes. The existence of ultraprimes satisfying properties not satisfied by any prime is what allows this generalisation of the theory of Kolyvagin systems. This new formalism was used by Loeffler and Zerbes in \cite{ultra_kolyvagin_LZ} to generalise the Euler system argument for bounding Selmer groups associated with non-ordinary Galois representations.

This setting allows us to develop the Kolyvagin system argument directly on the profinite Selmer groups that naturally appear when studying the Galois representation along a $\Z_p$-extension. It is a significant advantage over the previous theory, which requires one to go through the cohomology of the finite quotients. In our new formalism, the full power of the structure theorem of finitely generated Iwasawa modules can be used to describe the structure of the dual Selmer group.

In particular, we develop a higher rank theory of ultra Kolyvagin systems, which are seen as collections of classes in the exterior powers of patched Selmer groups. In this formalism, the Stickelberger ideals associated with a primitive Kolyvagin system are more than bounds of the Selmer group, but can describe its pseudo-isomorphism structure. In addition, the existence of ultra Kolyvagin systems shows, under mild big image assumptions, the absence of finite submodules in Iwasawa Selmer groups.

Using the methods presented in this article, we are able to compute the structure of the fine Selmer group of an elliptic curve and the Bloch-Kato Selmer group of the Rankin-Selberg convolution of modular forms along the cyclotomic $\Z_p$-extension. However, the method can be applied to a wider class of Galois representations for which Euler systems are known.

\subsection{Main Results}

We propose a generalisation of the theory of Kolyvagin systems which allows the computation of the higher Fitting ideals of a Selmer group associated with a Galois representation over the Iwasawa algebra
\[\Lambda:=\Z_p[[X]],\]
where $p\geq 5$ is a prime number. We assume that $\bT$ is a free, finitely generated $\Lambda$-module endowed with a continuous action of the absolute Galois group of a number field $K$ that is ramified only at finitely many primes.

We will consider the Selmer group $H^1_{\FF}(K,\bT)$ defined by a cartesian Selmer structure of positive core rank . In order to develop the theory, we need to impose certain mild big image assumptions and the $p$-primality of the Tamagawa numbers. Precise definitions and assumptions are detailed in \textsection\ref{sec:selmer}-\ref{sec:iwasawa}.

We establish the notion of ultra Kolyvagin systems of higher rank and we show, under the assumptions previously mentioned, that the Iwasawa module of ultra Kolyvagin systems $\bKS(\bT,\FF)$ is free of rank one. Hence we can fix a generator $\kappa$ of $\bKS(\bT,\FF)$, which is referred to as a primitive Kolyvagin system.

In addition, when $\kappa\in \bKS(\bT,\FF)$ is a Kolyvagin system, we can define a sequence of ideals $\bTheta_i(\kappa)$, where $i\geq 0$, from the divisibility of the different classes in $\kappa$ (see Definition \ref{def:kol_theta}). 

This sequence of ideals can be used to describe the structure of the dual Selmer group. By local duality, it is possible to define a dual Selmer structure $\FF^*$, that will determine a Selmer group on the cohomology of the Cartier dual $\bT^*$. It is the structure of its Pontryagin dual $H^1_{\FF^*}(K,\bT^*)^\vee$ that can be determined using the ideals $\bTheta_i(\kappa)$.

\begin{theorem*}(Theorem \ref{th:iwasawa:kolyvagin_theta_fitting})
Let $\bT$ and $\FF$ be as above and let $\kappa$ be a primitive Kolyvagin system. Then the ideals $\bTheta_i(\kappa)$ coincide, up to finite index, with the Fitting ideals $\Fitt^i_\Lambda\Bigl(H^1_{\FF^*}(K,\bT^*)^\vee\Bigr)$ of the dual Selmer group. 
\label{th:intro:pseudo}
\end{theorem*}

Theorem \ref{th:intro:pseudo} gives enough information about the Fitting ideals to compute the Iwasawa Selmer module up to pseudo-isomorphism. However, it can be refined in order to prove the absence of finite submodules in the Selmer group.

\begin{theorem*}(Theorem \ref{th:iwasawa:ultrakol:theta_fitting_equal})
Under the assumptions of Theorem \ref{th:intro:pseudo}, suppose that $\FF$ also satisfies Assumption \ref{ass:classic} from section \textsection\ref{sec:cl:ass}. Then 
\[\Theta_i(\kappa)=\Fitt^i_\Lambda\Bigl(H^1_{\FF^*}(K,\bT^*)^\vee\Bigr)\ \ \forall i\in \Z_{\geq 0}.\]
\label{th:intro:equal}
\end{theorem*}

This finer computation of Fitting ideals can be used to show that the Pontryagin duals of Iwasawa Selmer groups contain no finite submodules or, equivalently, that they have projective dimension one. The next result is a generalisation of previous theorems in \cite[Propositions 4.8 and 4.9]{Greenberg}, concerning the Selmer group of an elliptic curve, and \cite[Proposition 2.9]{Kur2014}, which assumed stronger assumptions such as the torsionness of the Selmer group and the vanishing of the $\mu$-invariant associated with elliptic curves.

\begin{corollary*}(Theorem \ref{th:iwasawa:nofinsub})
In the setting of Theorem \ref{th:intro:equal}, the Selmer group $H^1_{\FF^*}(K,\bT^*)^\vee$ contains no finite Iwasawa submodules.
\label{cor:intro:finsub}
\end{corollary*}

These results, particularly Theorem \ref{th:intro:equal}, can be applied to compute the higher Fitting ideal of the Iwasawa Selmer groups of some specific representations for which an Euler system is known. Particularly, we apply our theory to the following two situations:

\begin{theorem*}(Theorem \ref{th:kato})
Let $E/\Q$ be an elliptic curve defined over $\Q$ and let $T_pE$ be its $p$-adic Tate module. Let $\kappa_\Kato$ be the Kolyvagin systems obtained from Kato's Euler system \cite{Kato}. Under Assumptions \ref{Eord}-\ref{EIMC} in \textsection\ref{sec:app:ec}, 
\[\bTheta_i(\kappa_{\Kato})=\Fitt^i_\Lambda\Bigl(\Sel_0(\Q_\infty,E[p^\infty])^\vee\Bigr)\ \ \forall i\in \Z_{\geq 0},\]
where $\Sel_0(\Q_\infty,E[p^\infty])$ denotes the fine Selmer group, with restricted local condition at $p$.
\label{th:intro:kato}
\end{theorem*}

\begin{theorem*}(Theorem \ref{th:llz})
Let $T$ be the Galois representation associated with the Rankin-Selberg convolution of two ordinary modular forms which satisfies Assumptions \ref{RSord}-\ref{RSquot} in \textsection\ref{sec:app:rs} and the Iwasawa main conjecture. Let $\kappa_\LLZ$ be the Kolyvagin system obtained from the Euler system of Lei-Loeffler-Zerbes \cite{LeiLoefflerZerbes}. Then
\[\bTheta_i(\kappa_{\LLZ})=\Fitt^i_\Lambda\Bigl(H^1_{\FBK}(\Q_\infty,T^*)^\vee\Bigr)\ \ \forall i\in \Z_{\geq 0}.\]
\end{theorem*}

However, our methods can be applied in greater generality to Galois representations for which the existence of Euler systems is known.

\subsection{Relation with the existing literature}

The existing literature on the theory of Kolyvagin systems was limited to the computation of the characteristic ideal of the dual Selmer group using the leading term of the system. Here we generalise this theory to compute its full structure.

This article is a natural generalisation of \cite{BurnsSakamotoSano2}. In loc.~cit., a sequence of ideals $\Theta_i(\kappa)$ was used to establish bounds on the Fitting ideals of the Selmer group and it is conjectured that these bounds are sharp. 

Here we can give an affirmative answer to this conjecture for representations with coefficients over the Iwasawa algebra by developing the theory in the patched cohomology groups constructed in \cite{Sweeting}. It allows a more powerful use of the structure theorem of Iwasawa modules that leads to the full computation of these Selmer groups.

Patched cohomology has been previously used in \cite{ultra_kolyvagin_LZ} for developing the Euler system argument in order to bound Selmer groups that were out of the scope of the original Euler system argument. In loc.~cit., a patched version of the Kolyvagin derivative is developed so one can obtain, from an Euler system, ultra Kolyvagin systems satisfying appropriate local conditions in non-ordinary settings. This generalisation allowed the application of the Euler system machinery to new Galois representations, establishing explicit bounds to their associated Selmer groups.

The generalisation made in this article also allows a simplification of the method in \cite{BurnsSakamotoSano2}, since Kolyvagin systems can be constructed as collections of classes in the exterior powers of Iwasawa Selmer groups. The reason is that, under our assumptions, the profinite Selmer groups are free Iwasawa modules, so there is a canonical identification between their exterior powers and their exterior biduals, the objects used in \cite{BurnsSakamotoSano2} to construct higher rank Kolyvagin systems.

Our main results here have been applied to determine the structure of the Selmer groups associated with an elliptic curve and with the Rankin-Selberg convolution of two ordinary modular forms. The method used to obtain ultra Kolyvagin systems was to compare them with the limit of classical Kolyvagin systems, which can be obtained from an Euler system using the standard argument in \cite{MazurRubin}. However, this comparison with classical Kolyvagin systems is not needed, since ultra Kolyvagin systems can be directly obtained from Euler systems using the patched version of the Kolyvagin derivative developed in \cite{ultra_kolyvagin_LZ}.

Previous computations of higher Fitting ideals of Iwasawa Selmer groups in the literature were limited to particular cases or required much stronger assumptions. The first results in this direction appear in \cite{Kurihara_gausssum}, \cite{Ohshita_circular} and \cite{Ohshita_elliptic}, concerning the limit of class groups along a $\Z_p$-extension. Those works were generalised in \cite{Kur2014} and \cite{Ohshita}, computing the higher Fitting ideals of Selmer groups defined for a Galois representation that is not residually self-dual. Our results can be seen as a generalisation of these results, since they do not require those assumptions about self-duality. 

Specially interesting is the recent preprint \cite{RoncheLongoVigni}, which studies the higher Fitting ideals for the self-dual representation associated with the Tate module of an elliptic curve in the anticyclotomic setting. This Selmer group is out of the scope of the methods of this article, since it is defined from a Selmer structure of core rank $0$. It remains an interesting question to extend the theory of ultra Kolyvagin systems of Iwasawa modules to obtain information about these Selmer groups, generalising the results in \cite{Angurel} in this equivariant setting.

\subsection{Structure}

The structure of this article is the following. \textsection \ref{sec:ultrafilters}-\textsection\ref{sec:selmer} are dedicated to outlining the construction of the patched cohomology from \cite{Sweeting}, as well as including the technical results needed for the development of this theory. 

In \textsection \ref{sec:iwasawa}, we specialize the (patched) Selmer groups obtained from Galois representations in which the coefficient ring is the Iwasawa algebra. The algebraic preliminaries, including the construction of the Fitting ideals and the rank reduction maps of exterior powers are included in \textsection\ref{sec:algebra}.

\textsection\ref{sec:iwasawa:stark} and \textsection\ref{sec:iwasawa:kol} are dedicated to the construction of ultra Stark and ultra Kolyvagin systems, respectively, whose classes live in the exterior powers of patched Selmer groups. \textsection\ref{sec:iwasawa:kol} includes the first main result in this article, Theorem \ref{th:iwasawa:kolyvagin_theta_fitting}, that describes the Fitting ideals of the Selmer group up to finite index.

\textsection\ref{sec:classical} compares the newly constructed Kolyvagin system with the classical theory of Kolyvagin systems in \cite{BurnsSakamotoSano2}. This comparison allows a finer description of the Fitting ideal (see Theorem \ref{th:iwasawa:ultrakol:theta_fitting_equal}), which are now completely determined by the ideals $\Theta_i(\kappa)$.

Finally, an application of this theorem to the ultra Kolyvagin systems obtained from the Euler systems constructed by Kato \cite{Kato} and Lei-Loeffler-Zerbes \cite{LeiLoefflerZerbes} appears in \textsection \ref{sec:app}.

\subsection*{Acknowledgements}

The author would like to thank Christian Wuthrich for numerous discussions during the preparation of this manuscript. The author is also grateful to Dominik Bullach, Kâzim Büyükboduk, David Loeffler and Chris Williams for enlightening conversations and comments about the topics in this manuscript.

\section{Ultrafilters}
\label{sec:ultrafilters}

\subsection{Ultraproducts}

\begin{definition}
A \emph{filter} in the natural numbers is a collection of sets $\UU$ in the power set $\mathcal P(\N)$ such that 
\begin{itemize}
\item \namedlabel{filter_empty}{(F0)} The empty set does not belong to $\UU$.
\item \namedlabel{filter_supset}{(F1)} If $S_1\subset S_2$ and $S_1\in \UU$, then $S_2\in \UU$.
\item \namedlabel{filter_intersection}{(F2)} If $S_1,S_2\in \UU$, then $S_1\cap S_2\in \UU$.
\end{itemize}
We say that a filter is an \emph{ultrafilter} if it also satisfies the following condition
\begin{itemize}
\item \namedlabel{ultrafilter_axiom}{(UF)} For every set $S\in \PP(\N)$, either $S\in \UU$ or $\N\setminus S\in \UU$.
\end{itemize}
\end{definition}

The key property of ultrafilters is that \ref{ultrafilter_axiom} generalises to finite partitions, i.e., ultrafilters contain exactly one set in every finite partition of $\N$.
\begin{proposition} (\cite[Proposition 2.1.2]{Sweeting})
Let $\UU$ be an ultrafilter, let $S\in \UU$ and let $C$ be a finite set. For every map $f:\ S\to C$, there exists a unique $c\in C$ such that $f^{-1}(c)\in \UU$. 
\label{prop:ultrafilter_finite}
\end{proposition}

\begin{notation}
When we say that some property is satisfied for $\UU$-many $i\in \N$, we mean that there is a set $S\in \UU$ such that the property holds for every $i\in S$.
\end{notation}

For every $a\in \N$, the set
\[\UU_a=\{S\subset \N:\ a\in S\}\]
is an ultrafilter. We say that an ultrafilter is \emph{principal} if there is an element $a\in \N$ such that $\UU=\UU_a$.

For this theory, interesting ultrafilters are non-principal ones. They cannot be constructed explicitly, but their existence is guaranteed by the axiom of choice. Note that every non-principal ultrafilter contains the Fréchet filter, formed by the subset with finite complement.

Ultrafilters will be used to define the ultraproduct, which is a functor that clusters together infinitely many groups.

\begin{definition}
Let $\UU$ be an ultrafilter and let $(M_n)_{n\in \mathbb N}$ be a sequence of sets. The \emph{ultraproduct} $\UU_n(M_n)$ is the set defined as
\[\UU_n(M_n)=\prod_{n\in \N} M_n\biggm /\sim,\]
where the product is the direct product of sets and $\sim$ is the equivalence relation defined as $(m_n)\sim (m'_n)$ if $m_n=m'_n$ for $\UU$-many $n$.
\label{def:ultraproduct}
\end{definition}

\begin{notation}
We will also denote the ultraproduct by $\UU(M_n)_{n\in \N}$ when there is no risk of confusion. In addition, for every set $M$, we will denote by $\UU(M)$ the ultraproduct of the sequence $(M_n)_{n\in \N}$ in which $M_n=M$ for all $n$. The latter will be also called the \emph{ultrapower} of $M$.
\end{notation}

\begin{remark}
If the sets $M_n$ are groups (resp.\ rings, modules, etc.), the ultraproduct $\UU(M_n)$ has a natural group structure (resp.\ ring structure, module structure, etc.).
\end{remark}

\begin{proposition} (\cite[Propositions 2.1.4 and 2.1.5]{Sweeting})
The ultraproduct $\UU$ is functorial. In addition, it is exact when seen in the category of pointed sets.
\label{prop:ultrafilter_exact}
\end{proposition}

In general, it is difficult to compute the ultraproduct, but there is a special case in which it is explicit: the ultraproduct of a sequence of finite sets of bounded order.

\begin{lemma}(\cite[Proposition 2.1.4. (ii)]{Sweeting})
If $M$ is a finite set, the diagonal map $\Delta:\ M\to \UU(M)$ is an isomorphism.
\label{lem:ultrafilter_finite_diagonal}
\end{lemma}

\begin{corollary}
Let $R$ be a finite ring and let $(M_n)_{n\in \N}$ be a sequence of $R$-modules of bounded order such that there exists $S\in \UU$ satisfying that 
\[\mathcal{A}_S:=\{\#M_n:\ n\in S\}\]
is bounded. Then $\UU(M_n)\cong M_n$ for $\UU$-many $n$.
\label{cor:ultrafilter_finite_modules}
\end{corollary}

\begin{proof}
Let $C$ be a bound of $\AA_S$. Since $R$ is finite, there are only finitely many isomorphism classes of $R$-modules with order smaller than $C$. Then Proposition \ref{prop:ultrafilter_finite} implies that exactly one of these isomorphism classes, say $N$ contains $\UU$-many $M_n$. Hence the ultraproduct $\UU(M_n)$ coindide with the ultrapower $\UU(N)$. By Lemma \ref{lem:ultrafilter_finite_diagonal}, we have that $\UU(M_n)\cong N$.
\end{proof}

Now we state some properties of the ultraproduct.
\begin{remark}
If $(A_n)_{n\in \N}$ and $(B_n)_{n\in \N}$ are sequences of sets, modules or rings, there is a canonical identification
\[\UU(A_n\times B_n)=\UU(A_n)\times \UU(B_n),\]
which is well defined because the ultrafilter satisfies \ref{filter_intersection}.
\label{rem:ultraproduct_product}
\end{remark}

\begin{proposition}
Let $(A_n)_{n\in \N}$, $(B_n)_{n\in \N}$ be two sequences in $R$-modules. Then there is an injection
\[\Psi:\ \UU(\Hom_R(A_n,B_n))\hookrightarrow \Hom_R\Bigl(\UU(A_n),\UU(B_n)\Bigr).\]
In addition, $\Psi$ is an isomorphism when $(B_n)$ is a constant sequence of finite sets.
\label{prop:ultraproduct_hom}
\end{proposition}

\begin{proof}
Let $(\varphi_n)$ be the sequence representing an element in the ultraproduct $\UU(\Hom_R(A_n,B_n))$ and let $(a_n)$ be a sequence representing an element in $\UU(A_n)$. We assign the sequence $\varphi_n(a_n)\in \UU(B_n)$. Clearly, this assignment behaves well under the ultrafilter equivalences in both $\UU(\Hom(A_n,B_n))$ and $\UU(A_n)$, so it defines a map
\[\Psi: \UU(\Hom(A_n,B_n))\to \Hom\Bigl(\UU(A_n),\UU(B_n)\Bigr).\]

In order to check injectivity, consider another sequence $(\psi_n)$ induces the same element in $\Hom\Bigl(\UU(A_n),\UU(B_n)\Bigr)$. For the sake of contradiction, assume that $(\varphi_n)$ and $(\psi_n)$ are not equivalent, i.e., the set
\[S=\{n\in \N:\ \varphi_n\neq \psi_n\}\in \UU.\]
Choose a sequence $(a_n)\in (A_n)$ where $\varphi_n(a_n)\neq \psi_n(a_n)$ for every $n\in S$. Hence $(\varphi_n(a_n))$ and $(\psi_n(a_n))$ represent different elements in $\UU(B_n)$, so $\Psi(\varphi_n)\neq \Psi(\psi_n)$.

Assume that $B_n=B$ for all $n\in \N$ and $B$ is a finite group. Then every element in $\Hom\Bigl(\UU(A_n),B\Bigr)$ lifts to a homomorphism in 
\[\Hom\left(\prod_{n\in \N} A_n,B\right)=\prod_{n\in \N}\left(\Hom(A_n,B)\right),\]
so it is also in the image of $\Psi$.
\end{proof}

\subsection{Ultraprimes}

For the rest of this paper, fix a non-principal ultrafilter $\UU$ and a number field $K$.

\begin{notation}
Denote by $\MM_K$ the set of places of $K$.
\end{notation}

\begin{definition}
An \emph{ultraprime} $\ku$ is an element of $\UU(\MM_K)$. More specifically, it is represented by a sequence of prime ideals $(\ell_n)_{n\in \N}$, and two sequences represent the same ultraprime if they coincide in $\UU$-many primes.
\end{definition}

\begin{remark}
The primes in $\MM_K$ are contained in the ultraprimes $\UU(\MM_K)$ via the diagonal map, i.e., a prime $\ell$ is identified with the equivalence class of the constant sequence $(\ell)$. The elements in the image of the map $\MM_K\hookrightarrow \UU(\MM_K)$ are sometimes referred to as \emph{constant ultraprimes}.
\end{remark}

We can use Proposition \ref{prop:ultrafilter_finite} to define the Frobenius element associated with an ultraprime $\ku$ in the absolute Galois group $G_K$.

\begin{definition}
Let $\ku=(\ell_n)_{n\in \N}$ be an ultraprime and let $L/K$ be a finite Galois extension of number fields. Then there exists a unique conjugacy class $c$ in $\Gal(L/K)$ such that $\Frob_{\ell_n}|_{L/K}\in c$ for $\UU$-many $n$. This class is called the \emph{Frobenius conjugacy class} of $\ku$ at $L/K$ and any element in it is called a \emph{Frobenius automorphism} $\Frob_\ku$.
\end{definition}

\begin{definition}
Let $\ku=(\ell_n)_{n\in \N}$ be an ultraprime. The (absolute) Frobenius conjugacy class $\Frob_\ku$ is defined as the conjugacy class
\[\Frob_{\ku}=\left(\Frob_{\ku}|_{L/K}\right)_{L/K}\subset \varprojlim_{L/K} \Gal(L/K)=G_K.\]
Any automorphism in this class is called a Frobenius element $\Frob_\ku$.
\label{def:frobenius}
\end{definition}

\begin{proposition}(\cite[Lemma 5.4]{ultra_kolyvagin_LZ})
Let $K$ be a number field and let $c\subset G_K$ be a conjugacy class. Then there exists an ultraprime $\ku$ such that $\Frob_\ku=c$.
\label{prop:ultraprimes_chebotarev}
\end{proposition}

The construction of the Frobenius was used in \cite{Sweeting} to artificially define the local Galois group at the ultraprime. It is a generalization of the tame quotient of the classical local Galois groups.

\begin{definition}
Let $\ku$ be an ultraprime represented by a sequence $(\ell_i)_{i\in \N}$. The local Galois group $G_\ku$ is defined as the profinite completion of $\UU(G_{\ell_i})$. Similarly, the inertia subgroup $I_\ku$ is obtained as the profinite completion of $\UU(\II_{\ell_i})$.
\label{def:local_galois_group}
\end{definition}

\begin{remark}
When $\ku$ is a non-constant ultraprime, the \emph{local Galois group} $G_\ku$ coincides with the semidirect product 
\[G_\ku:=\widehat \Z(1)\rtimes \langle \Frob_\ku\rangle,\]
 where $\langle \Frob_\ku\rangle $ is the free profinite group generated by one element which acts by $\Frob_\ku\in G_K$ on $\widehat \Z(1)$. Note that this action is well-defined since $K(\mu_{p^\infty})/K$ is an abelian exension.
 
 The \emph{inertia subgroup} $I_\ku\subset G_\ku$ is identified with the normal subgroup $\widehat \Z(1)\subset G_\ku$.
\label{rem:local_galois_group}
\end{remark}

\begin{remark}
When $\ku$ is a constant ultraprime, the semidirect product in Remark \ref{rem:local_galois_group} represents the tame inertia quotient of the Galois group $G_\ku$.
\end{remark}

\begin{remark}
Assume $T$ is a finite $G_{\ell}$-module for every prime $\ell$. By Proposition \ref{prop:ultrafilter_finite}, there is a well-defined action of $G_\ku$ on $T$ given by the following description: an element $(g_i)_{i\in \N}\in \UU(G_{\ell_i})$ acts by $\rho\in \Aut(T)$ if, for $\UU$-many $i$, the element $g_i\in G_K$ acts on $T$ by $\rho$. This construction extends to the situations in which $T$ is either profinite or ind-finite by taking limits.
\end{remark}

\begin{remark}
If $T$ is unramified for $\UU$-many primes $(\ell_i)$ in any representing sequence of $\ku$, then the action of $G_\ku$ is also unramified. That is always the case when $\ku$ is non-constant and $T$ ramifies only at finitely many primes.
\end{remark}

\subsection{Product of ultraprimes}

There is a well-defined notion of products of ultraprimes, obtaining classes in $\UU(\II_K)$, where $\II_K$ denotes the set of modulus of $K$. Indeed, for every $s\in \N$, there is a well defined map
\[\UU(\MM_K)^s\to \UU(\II_K):\ (\ku_1,\ldots,\ku_s)\mapsto \ku_1\cdots\ku_s.\]
This product satisfies a weaker analogue of the unique prime factorisation theorem on $\II_K$, since the factorizations are unique but not all of the elements in $\UU(\II_K)$ admit an ultraprime factorization.

Since the product of ultraprimes is commutative, we can express it in the following way.

\begin{definition}
The product of ultraprimes is a map 
\[\UU(\MM_K)^s/S_s\to \UU(\II_K),\]
where $S_s$ denotes the symmetric group, acting on $\UU(\MM_K)$ in the natural way. The map is defined as follows: for every finite set of ultraprimes 
\[\ku_1=\left(\ell^{(1)}_k\right)_{k\in \N},\ldots, \ku_s=\left(\ell^{(s)}_k\right)_{k\in \N},\] 
we define
\[\ku_1\cdots\ku_s =(\ell^{1}_k\cdots \ell^{(s)}_k)_{k\in \N}\in \UU(\II_K).\]
\label{def:square-free_product}
\end{definition}

\begin{remark}
Definition \ref{def:square-free_product} produces a well-defined class in $\UU(\II_K)$, which is represented by a sequence $(n_i)_{i\in \N}$ in which $\nu(n_i)=s$ for all $i\in \N$, where $\nu(I)$ denotes the number of prime ideals dividing $I$, counting multiplicities. Hence, the product of $s$ ultraprimes cannot be equal to the product of $s'$ ultraprimes unless $s=s'$.
\label{rem:square-free_welldef}
\end{remark}

We can show that the map in Definition \ref{def:square-free_product} is injective.
\begin{proposition}
The product in Definition \ref{def:square-free_product} induces an injective map
\[\bigsqcup_{s>0}:\ \UU(\MM_K)^s/S_s\to \UU(\II_K).\]
\label{prop:square-free_product_injective}
\end{proposition}

\begin{proof}
Assume that $\kn=\ku_1\cdots\ku_s=\kq_1\cdots\kq_{s'}\in \UU(\II_K)$. By Remark \ref{rem:square-free_welldef}, we have that $s=s'$. Choose representatives 
\[\begin{array}{cc}
    \ku_i=\Bigl(\ell_k^{(i)}\Bigr)_{k\in \N},\ & \kq_j=\Bigl(q_k^{(j)}\Bigr)_{k\in \N}.
\end{array}\]
    By the definition of $\UU(\II_K)$,
    \[\ell_k^{(1)}\cdots\ell_k^{(s)}=q_k^{(1)}\cdots q_k^{(s)}\textrm{ for $\UU$-many $k$.}\]
    By Proposition \ref{prop:ultrafilter_finite} and the unique factorization in $\II_K$, for every $i\leq s$, there is some $j\leq s$ such that $\ell_k^{(i)}=q_k^{(j)}$ for $\UU$-many $k$, leading to $\ku_i=\kq_j$. This shows that $(\kq_1,\ldots,\kq_s)$ is a reordering of $(\ku_1,\ldots,\ku_s)$.
\end{proof}

Remark \ref{rem:square-free_welldef} and Proposition \ref{prop:square-free_product_injective} prove the uniqueness of the ultraprime factorization. Now it makes sense to mention square-free products of ultraprimes.

\begin{definition}
A product of ultraprimes is called \emph{square-free} if no two $\ku_i$ are equal.
\end{definition}

As mentioned before, the ultraprime factorisation does not exist for every $\kn\in \UU(\II_K)$, but only when it can be represented by a sequence $(n_k)_{k\in \N}$ in which the number of prime divisors of $n_k$, counting multiplicities, remains bounded.
\begin{proposition}
Assume an element in $\kn\in\UU(\II_K)$ can be represented by a sequence $(n_k)_{k\in \N}$ such that $\nu(n_k)$ is bounded. Then $\kn$ can be obtained as a product of ultraprimes. 

Assume, in addition, that $n_k$ is a square-free product for $\UU$-many $k$, then $\kn$ is also a square-free product.
\label{prop:patched:product_ultraprimes_ultraproduct_naturals}
\end{proposition}

\begin{proof}
By Proposition \ref{prop:ultrafilter_finite}, there exists $s\in \N$ and $S\in \UU$ such that $\nu(n_k)=s$ for all $k\in S$. It is possible to choose $s$ sequences of prime numbers $\Bigl(\ell_k^{(j)}\Bigr)_{k\in \N}$ such that, whenever $k\in S$, then
\[n_k=\ell_k^{(1)}\cdots \ell_k^{(s)}.\]
By construction, $\kn=\ku_1\cdots\ku_s$, where $\ku_j$ is the ultraprime represented by the sequence $\Bigl(\ell_k^{(j)}\Bigr)_{k\in \N}$.

Note that, if $\ku_i=\ku_j$ for some $i$ and $j$, every representative $(n_i)_{i\in \N}$ of $\ku_1\cdots\ku_s$ will not contain $\UU$-many square-free $n_i$, so the last statement holds.
\end{proof}

\section{Patched cohomology}
\label{sec:patched}

In this section, we use the ultraproduct defined in the previous one to outline the construction of patched cohomology, which was already constructed in \cite{Sweeting}. Instead of focusing on global patched cohomology, we make this construction in a more general way, so we can also recover local patched cohomology as a particular case.

In order to have control over the patched cohomology groups, one should define it first for finite coefficient rings as an ultraproduct of cohomology groups, and then we extend the definition to either profinite or ind-finite coefficient rings by taking limits.

\begin{definition}
Let $T$ be a finite, abelian group endowed with an action from every group in a sequence\footnote{Technically, it is only needed that the action is well defined for $\UU$-many $n$} of groups $G=(G_n)_{n\in \mathbb N}\in \UU(\{\textrm{groups}\})$. The $\UU$-patched cohomology group is defined as the ultraproduct
\[\bH^i(G,T):=\UU(H^i(G_n,T)).\]
If $T$ is a profinite abelian group, we define the patched cohomology as the inverse limit
\[\bH^i(G,T):=\varprojlim_{T\twoheadrightarrow T'}\bH^i(G,T/T'),\]
where the limit is taken over all the finite quotients of $T$.

Similarly, when $T$ is an ind-finite abelian group, the patched cohomology is defined as
\[\bH^i(G,T):=\varinjlim_{T'\hookrightarrow T} \bH^i(G,T'),\]
where the limit is taken over all the finite subgroups of $T$.
\label{def:patched_cohomology}
\end{definition}

\begin{proposition}
The assignment 
\[T\mapsto \bH^i(G,T)\]
is a functor from the category of either finite abelian groups, profinite abelian groups and ind-finite abelian groups to the category of abelian groups.
\end{proposition}

\begin{proof}
It follows from the functorial properties of cohomology, ultraproducts, and inverse and direct limits.
\end{proof}

When the coefficient group $T$ is finite, there is a long exact sequence analogously to classical group cohomology.
\begin{proposition}
Let 
\[\xymatrix{0\ar[r] & A\ar[r]^{\mu} & B\ar[r]^{\varepsilon} & C\ar[r] & 0}\]
be an exact sequence of finite abelian groups. Assume $A$, $B$ and $C$ are endowed with an action of each group in a sequence $G=(G_n)_{n\in \N}$. Then there is an exact sequence

\begin{center}
\begin{tikzpicture}[descr/.style={fill=white,inner sep=1.5pt}]
        \matrix (m) [
            matrix of math nodes,
            row sep=3.5em,
            column sep=2.5em,
            text height=1.5ex, text depth=0.25ex
        ]
        {  \cdots& \bH^i(G,A) & \bH^i(G,B) & \bH^i(G,C) \\
            & \bH^{i+1}(G,A) & \bH^{i+1}(G,B)& \bH^{i+1}(G,C) &\\
            &\bH^{i+2}(G,A)&\cdots&&\\
        };

        \path[overlay,->, font=\scriptsize,>=latex]
        (m-1-1) edge (m-1-2)
        (m-1-2) edge (m-1-3)
        (m-1-3) edge (m-1-4)
        (m-1-4) edge [out=355,in=175] (m-2-2) 
        (m-2-2) edge (m-2-3)
        (m-2-3) edge (m-2-4)
        (m-3-2) edge (m-3-3)
        (m-2-4) edge [out=355,in=175] (m-3-2);    
\end{tikzpicture}
\end{center}
\label{prop:patched:long_finite}
\end{proposition}

\begin{proof}
Since $A$, $B$ and $C$ are finite, we have that 
\[\bH^i(G,-)=\UU_n\Bigl(H^i(G_n,-)\Bigr).\]
Then the result follows from Proposition \ref{prop:ultrafilter_exact} and the classical long exact sequence for Galois cohomology.
\end{proof}

Under the following assumption, we can construct a long cohomological exact sequence for the patched cohomology with profinite coefficients.

\begin{definition}
We say a finite, abelian group $T$ endowed with an action of each group in a sequence $G=(G_n)_{n\in \N}$ satisfies the \emph{bounded cohomology property} for some index $i\in \Z_{\geq 0}$ if there is a constant $K$ such that 
\[\#H^i(G_n,T)\leq K\ \textrm{for $\UU$-many $n$.}\]
\label{def:BCP}
\end{definition}

\begin{remark}
The bounded cohomology property guarantees, using Corollary \ref{cor:ultrafilter_finite_modules}, that the patched cohomology group $\bH^1(G,T)$ is finite.
\label{rem:BCP_finite_patched}
\end{remark}

Assuming the bounded cohomology property, we can generalise Proposition \ref{prop:patched:long_finite} to profinite groups. The following result is proven\footnote{The proof in loc.~cit. is done in a slightly different setting, but the argument works even only assuming the bounded cohomology property.} in \cite[Lemma 2.4.12]{Sweeting}.
\begin{proposition}
Let 
\[\xymatrix{0\ar[r] & A\ar[r]^{\mu} & B\ar[r]^{\varepsilon} & C\ar[r] & 0}\]
be an exact sequence of continuous maps of profinite abelian groups. Assume $A$, $B$ and $C$ are endowed with an action of the groups in $G=(G_n)_{n\in \N}$ and satisfy the bounded cohomology property. Then there is an exact sequence
\begin{center}
\begin{tikzpicture}[descr/.style={fill=white,inner sep=1.5pt}]
        \matrix (m) [
            matrix of math nodes,
            row sep=3.5em,
            column sep=2.5em,
            text height=1.5ex, text depth=0.25ex
        ]
        {  \cdots& \bH^i(G,A) & \bH^i(G,B) & \bH^i(G,C) \\
            & \bH^{i+1}(G,A) & \bH^{i+1}(G,B)& \bH^{i+1}(G,C) &\\
            &\bH^{i+2}(G,A)&\cdots&&\\
        };

        \path[overlay,->, font=\scriptsize,>=latex]
        (m-1-1) edge (m-1-2)
        (m-1-2) edge (m-1-3)
        (m-1-3) edge (m-1-4)
        (m-1-4) edge [out=355,in=175] (m-2-2) 
        (m-2-2) edge (m-2-3)
        (m-2-3) edge (m-2-4)
        (m-3-2) edge (m-3-3)
        (m-2-4) edge [out=355,in=175] (m-3-2);    
\end{tikzpicture}
\end{center}

\label{prop:patched:long_profinite}
\end{proposition}

\subsection{Local patched cohomology}

Local patched cohomology groups can be constructed for every ultraprime. It is the patched cohomology of the local Galois groups of the primes in a representing sequence of $\ku$. We will show it coincides with the group cohomology of the Galois group $G_\ku$ constructed in Definition \ref{def:local_galois_group}.

\begin{definition}
Let $\ku$ be an ultraprime represented by the sequence $(\ell_n)_{n\in}$. The local patched cohomology group is defined as
\[\bH^i(K_\ku,T):=\bH^i\Bigl((G_{K_{\ell_n}}),T\Bigr).\]
\end{definition}

\begin{remark}
Let $T$ be a finite group endowed with an action of the Galois group $G_K$ that is ramified only at finitely many primes. After fixing inclusions $\overline{\Q}\subset \overline{\Q_\ell}$ for every prime $\ell$, there is a natural action of $G_\ku$ on $T$ for every ultraprime $\ku$. Then $T$ satisfies the bounded cohomology property (see Definition \ref{def:BCP}). 

Indeed, if $\ku$ is represented by the sequence $(\ell_i)_{i\in \N}$, then, for $\UU$-many $i$, $\ell_i$ does not divide the cardinality of $T$, denoted by $\#T$, and the action of $G_{K_{\ell_i}}$ on $T$ is unramified. For those $i$, local duality implies that 
\[\# H^1(K_{\ell_i},T)\leq 2\# T,\]
so the bounded cohomology property holds. Then Proposition \ref{prop:patched:long_profinite} proves the existence of a long exact sequence in the profinite case.
\label{rem:patched:BCP_local}
\end{remark}

We can mimic the construction to obtain the cohomology of the inertia subgroup. That will lead to finite and singular cohomology groups.

\begin{definition}
Let $\ku$ be an ultraprime represented by the sequence $(\ell_n)_{n\in \N}$. The cohomology of the inertia subgroup $I_\ku$ is defined as the patched cohomology group
\[\bH^i(I_\ku,T):=\bH^i\Bigl(\II_{\ell_n},T\Bigr).\]
\end{definition}

\begin{definition}
Let $\ku$ be either a non-constant ultraprime or a constant one that does not ramify on $T$. The \emph{finite patched cohomology group} is defined as
\[\bH^1_{\f}(K_\ku,T):=\ker\biggl(\bH^1(K_\ku,T)\to \bH^1(I_\ku,T)\biggr).\]
\label{def:patched:finite}
\end{definition}

\begin{remark}
Definition \ref{def:patched:finite} can be extended to ramifying primes as in \cite[Definition I.3.4]{Rubin}.
\end{remark}

\begin{definition}
We define the \emph{singular patched cohomology group} as
\[\bH^1_\s(K_\ku,T):=\frac{\bH^1(K_\ku,T)}{\bH_\f^1(K_\ku,T)}.\]
\label{def:patched:singular}
\end{definition}

The local patched cohomology coincides with the group cohomology of the local Galois group defined in Definition \ref{def:local_galois_group}.

\begin{proposition}(\cite[Proposition 2.3.2]{Sweeting})
Let $T$ be either a finite, profinite or ind-finite abelian group, and let $\ku$ be an ultraprime represented by the sequence $(\ell_n)_{n\in \N}$. Then
\[\begin{array}{cc}
    \bH^i\Bigl(K_{\ku},T\Bigr)=H^i(G_\ku,T),\ & \bH^i\Bigl(\II_{\ku},T\Bigr)=H^i(I_\ku,T),
\end{array}\]
where $G_\ku$ and $I_\ku$ are the patched local Galois group defined in Definition \ref{def:local_galois_group}. Recall that $\bH$ denotes patched cohomology while $H$ denotes standard group cohomology.
\label{prop:local_patched_galois_group}
\end{proposition}

The study of local patched cohomology groups can take advantage of the following fact.
\begin{corollary}(\cite[Proposition 2.3.2]{Sweeting})
Let $\ku=(\ell_n)_{n\in \N}$ be an ultraprime and let $T$ be a finite group endowed with an action of $G_K$. For $\UU$-many $i$, there are isomorphisms
\[\begin{array}{cc}
    \varphi_i^\ku:\ \bH^1(K_\ku,T)\cong H^1(K_{\ell_i},T),\ & 
    \varphi_i^{I_\ku}:\ \bH^1(I_\ku,T)\cong H^1(\II_{\ell_i},T),\\
    \varphi_i^{\ku,\f}:\ \bH^1_\f(K_\ku,T)\cong H^1_\f(K_{\ell_i},T),\ &
    \varphi_i^{\ku,\s}:\ \bH^1_\s(K_\ku,T)\cong H^1_\s(K_{\ell_i},T).    
\end{array}\]
\label{cor:patched:local_constant}
\end{corollary}

Patched cohomology also satisfies an analogue of local duality.

\begin{notation}
If $T$ is a Galois representation, its Cartier dual will be denoted by $T^*$. Similarly, for any group $M$, its Pontryagin dual is denoted by $M^\vee$.
\label{not:dualities}
\end{notation}

\begin{proposition}(\cite[Proposition 2.6.4]{Sweeting})
Let $T$ be either a finite, profinite or ind-finite Galois module and let $\ku=(\ell_n)_{n\in \N}$ be an ultraprime, which is not a constant archimedean prime. For $i=0,1,2$, there is a non-degenerate pairing 
\[\bH^i(K_\ku,T)\times \bH^{2-i}(K_\ku,T^*) \to \Q/\Z.\]
Under this pairing, $\bH^1_\f(K_\ku,T)$ and $\bH^1_\f(K_\ku,T^*)$ are annihilators of each other.
\label{prop:patched:local_duality}
\end{proposition}

\subsection{Global patched cohomology}

The goal of this section is to outline the construction of global patched cohomology groups. We do this by generalising the notion of Galois cohomology of the maximal extension $K^\Sigma/K$ unramified outside the certain finite set of prime numbers $\Sigma$. Hence, we aim to construct the patched cohomology group unramified outside a finite set of ultraprimes.

\begin{definition}(\cite[Definition 2.4.2]{Sweeting})
Let $T$ be either a finite, profinite or ind-finite abelian group endowed with an action of $G_K$ that is ramified outside finitely many primes. Let $\kn\in \NN$ be a square-free product of ultraprimes represented by the sequence $(n_k)_{k\in \N}$. We define the maximal patched cohomology group unramified at $\kn$ by
\[\bH^i(K^\kn/K,T):=\bH^i\Bigl(\Gal(K^{n_k}/K),T^{G_{K^{n_k}}}\Bigr),\]
where $K^{n_k}$ represents the maximal extension of $K$ unramified outside the prime divisors of $n_k$. Note that this definition is independent of the sequence representing $\kn$.
\end{definition}

\begin{notation}
If $\Sigma$ is a finite set of distinct ultraprimes and $\kn$ is the product of all the ultraprimes in $\Sigma$, we will also denote $\bH^i(K^\kn/K,T)$ by $\bH^i(K^\Sigma/K,T)$.
\end{notation}

The most important property of this construction is that the global patched cohomology groups satisfy the bonded cohomology property from Definition \ref{def:BCP}. In order to show it, we take advantage of the fact that every sequence $(n_i)_{i\in \N}$ representing the product of ultraprimes $\kn$ satisfies that $\nu(n_i)=k$ is constant for $\UU$-many $i$ and some $k\in \N$.

\begin{proposition}(\cite[Lemma 2.4.5]{Sweeting}, see also \cite[Theorem 4.10]{Milne})
Let $T$ be a finite abelian group endowed with an action of $G_K$ that is ramified outside finitely many primes and let $\kn\in \UU(\II_K)$ be a finite product of ultraprimes. Denote by $n'$ the square-free product of constant prime divisors of $\kn$. Then $T':=T^{\Gal(K^{n'}/K)}$ and the natural action of
\[G(K^{\kn}/K):=(\Gal(K^{n_i}/K))_{i\in \N}\]
satisfy the bounded cohomology property stated in Definition \ref{def:BCP}. Note that the action is only defined when $n_i$ is not divisible by any archimedean or $p$-adic primes or any primes at which $T$ is ramified. However, the conditions on $T$ and the fact that no constant ultraprimes divides $\kn$ implies that it is defined for $\UU$-many $i$.
\label{prop:BCP_global}
\end{proposition}

\begin{proof}
Say that $\kn$ is the product of $S$ ultraprimes. Then $\nu(n_i)=s$ for $\UU$-many $i$. Then the proposition follows from \cite[Lemma 2.4.5]{Sweeting} (see also \cite[Theorem 4.10]{Milne}).
\end{proof}

\begin{corollary}
Let $T$ be a finite group and let $\Sigma\subset \UU(\MM_K)$ be a finite set of ultraprimes. The patched cohomology groups $\bH^i(K^\Sigma/K,T)$ are finite for all $i\geq 0$.
\label{cor:global_finite}
\end{corollary}

Global patched cohomology also satisfies a Tate's formula for the  Euler characteristic.
\begin{corollary}
Let $\kn$ be a square-free product of ultraprimes. Then the sequence $\Gal(K^\kn/K)$ has (patched) $p$-cohomological dimension $2$ and for any finite group $T$ endowed with an action of $\Gal(K^\kn/K)$, we have that 
\begin{equation}
\frac{\#\bH^2(K^\kn/K,T) \#\bH^0(K^\kn/K,T)}{\#\bH^1(K^\kn/K,T)}=\prod_{v\in \Sigma_\infty}\frac{\#\bH^0(K_v,T)}{\#T^{[K_v:\R]}},
\label{eq:euler_char}
\end{equation}
where $\Sigma_\infty$ denotes the set of archimedean places of $K$.
\label{cor:euler_char}
\end{corollary}

\begin{proof}
Let $(n_k)_{k\in \N}$ be a sequence representing $\kn$. Since $\Gal(K^{n_k}/K)$ has $p$-cohomological dimension $2$ by \cite[Proposition 8.3.18]{NSW}, then $\Gal(K^{\kn}/K)$ also has cohomological dimension $2$. In addition, Corollary \ref{cor:ultrafilter_finite_modules} and Proposition \ref{prop:BCP_global} imply the existence of $\UU$-many $k$ such that 
\[\bH^i(K^{\kn}/K,T)\cong H^i(K^{n_k}/K,T)\ \forall k=0,1,2.\]
Then \eqref{eq:euler_char} holds from its classical analogue in \cite[Theorem 8.7.4]{NSW}.
\end{proof}

The following corollary holds by combining Propositions \ref{prop:patched:long_profinite} and \ref{prop:BCP_global}.

\begin{corollary}
Let 
\[\xymatrix{0\ar[r] & A\ar[r]^{\mu} & B\ar[r]^{\varepsilon} & C\ar[r] & 0}\]
be an exact sequence of continuous maps of profinite $G_K$-modules that are unramified outside finitely many primes. If $\Sigma$ denotes a finite set of ultraprimes that does not contain any ramifying prime of $T$, there is an exact sequence
\begin{center}
\begin{tikzpicture}[descr/.style={fill=white,inner sep=1.5pt}]
        \matrix (m) [
            matrix of math nodes,
            row sep=3.5em,
            column sep=2em,
            text height=1.5ex, text depth=0.25ex
        ]
        {  \cdots& \bH^i(K^\Sigma/K,A) & \bH^i(K^\Sigma/K,B) & \bH^i(K^\Sigma/K,C) \\
            & \bH^{i+1}(K^\Sigma/K,A) & \bH^{i+1}(K^\Sigma/K,B)& \bH^{i+1}(K^\Sigma/K,C) &\\
            &\bH^{i+2}(K^\Sigma/K,A)&\cdots&&\\
        };

        \path[overlay,->, font=\scriptsize,>=latex]
        (m-1-1) edge (m-1-2)
        (m-1-2) edge (m-1-3)
        (m-1-3) edge (m-1-4)
        (m-1-4) edge [out=355,in=175] (m-2-2) 
        (m-2-2) edge (m-2-3)
        (m-2-3) edge (m-2-4)
        (m-3-2) edge (m-3-3)
        (m-2-4) edge [out=355,in=175] (m-3-2);    
\end{tikzpicture}
\end{center}
    \label{cor:patched:global_long_exact}
\end{corollary}

\begin{remark}
If $\Sigma_1\subset \Sigma_2\subset \UU(\MM_K)$ are two finite sets of ultraprimes. Then there is a natural map
\[\bH^i(K^{\Sigma_1}/K,T)\hookrightarrow \bH^i(K^{\Sigma_2}/K,T).\]
When $T$ is finite, it is induced by a sequence of inflation maps. The general case, when $T$ is profinite and ind-finite, follows by taking limits.
\label{rem:patched_global_maps}
\end{remark}

We can use the maps in Remark \ref{rem:patched_global_maps} to construct the absolute global patched cohomology group.
\begin{definition}(\cite[\textsection 2.4.6]{Sweeting})
The absolute global patched cohomology group is defined as 
\[\bH^1(K,T):=\varinjlim_{\Sigma\subset \UU(\MM_K)} \bH(K^{\Sigma}/K,T),\]
where the limit is taken over all the finite sets and the transition maps are the ones defined in Remark \ref{rem:patched_global_maps}.
\label{def:absolute_galois}
\end{definition}

\begin{remark}
Let $\Sigma$ be a finite set of ultraprimes and let $\ku$ be an ultraprime. There exists a restriction map
\[\res:\ \bH^1(K^\Sigma/K,T)\to \bH^1(K_\ku,T),\]
induced, when $T$ is finite, by the restriction map in every factor of the ultraproduct. When $T$ is profinite (resp. ind-finite), the restriction map is obtained as the limit of the restriction map in the cohomology of the finite quotients (resp. submodules).
\end{remark}

The patched cohomology unramified outside a set of primes $\Sigma$ can be recovered as the unramified part of the absolute global patched cohomology group.

\begin{proposition}(\cite[Proposition 2.4.11]{Sweeting})
Let $\Sigma\subset \UU(\MM_K)$ be a finite set of ultraprimes. Then 
\[\bH^1(K^\Sigma/K,T)=\ker\left(\bH^1(K,T)\to \prod_{\ku\in \UU(\MM_K)\setminus \Sigma} \frac{\bH^1(K_\ku,T)}{\bH^1(I_\ku,T)}\right).\]
\label{prop:patched:unramified}
\end{proposition}

Local and glocal patched cohomology also satisfy an analogue of Poitou-Tate $9$-term exact sequence.

\begin{proposition}
Let $T$ be either a finite, profinite or ind-finite group endowed with an action of $G_K$ and let $\Sigma$ be a finite set of ultraprimes containing all the archimedean ones, p-adic ones and all ramifying primes of $T$. Then there is a  $9$-term exact sequence
\begin{center}
\begin{tikzpicture}[descr/.style={fill=white,inner sep=1.5pt}]
        \matrix (m) [
            matrix of math nodes,
            row sep=3.5em,
            column sep=0.8em,
            text height=1.5ex, text depth=0.25ex
        ]
        {  0& \bH^0(K^{\Sigma}/K,T) & \displaystyle{\prod_{\ku\in \Sigma}} \bH^0(K_\ku,T) & \bH^2(K^{\Sigma}/K,T^*)^\vee \\
            & \bH^1(K^{\Sigma}/K,T) & \displaystyle{\prod_{\ku\in \Sigma}} \bH^1(K_\ku,T) & \bH^1(K^{\Sigma}/K,T^*)^\vee &\\
            &\bH^2(K^{\Sigma}/K,T) & \displaystyle{\prod_{\ku\in \Sigma}} \bH^0(K_\ku,T^*)^\vee & \bH^0(K^{\Sigma}/K,T^*)^\vee & 0.\\
        };

        \path[overlay,->, font=\scriptsize,>=latex]
        (m-1-1) edge (m-1-2)
        (m-1-2) edge (m-1-3)
        (m-1-3) edge (m-1-4)
        (m-1-4) edge [out=355,in=175] (m-2-2) 
        (m-2-2) edge (m-2-3)
        (m-2-3) edge (m-2-4)
        (m-3-2) edge (m-3-3)
        (m-2-4) edge [out=355,in=175] (m-3-2)
        (m-3-3) edge (m-3-4)
        (m-3-4) edge (m-3-5)
        ;    
\end{tikzpicture}
\end{center}
\label{prop:patched:poitou-tate}
\end{proposition}

\begin{proof}
When $T$ is finite, it follows from Poitou-Tate duality in classical Galois cohomology (\cite[Theorem 8.4.4]{NSW}) and the exactness of the ultraproduct in Proposition \ref{prop:ultrafilter_exact}. When $T$ is either profinite or ind-finite, it follows by taking limits.
\end{proof}

\section{Patched Selmer groups}
\label{sec:selmer}

\subsection{Patched Selmer structures}

Following \cite{Sweeting}, we can now define the Selmer structures in this setting. The main innovation is that they also include local conditions at non-constant ultraprimes. However, classical Selmer structures can be recovered when the local condition is the finite cohomology group for every non-constant ultraprime.

\begin{definition}
A \emph{patched Selmer structure} $\FF$ consists of the following data:
\begin{itemize}
\item A finite set $\Sigma_\FF$ of $\UU(\MM_K)$ containing all constant ultraprimes lying over $p$, $\infty$ and ramified (constant) ultraprimes of $T$.
\item For each $\ku\in \Sigma_{\FF}$, a closed $R$-submodule
\[\bH^1_\FF(K_\ku,T) \subset \bH^1(K_\ku,T).\]
\end{itemize}
\label{def:patched:selmer}
\end{definition}

\begin{definition}
The Selmer module of a Selmer structure $\FF$ is defined as
\[\bH^1_{\FF}(K,T):=\ker\Biggl(\bH^1(K^{\Sigma_\FF}/K,T) \to \prod_{\ku\in \Sigma_\FF} \frac{\bH^1(K_\ku,T)}{\bH^1_{\FF}(K_\ku,T)}\Biggr).\]
\end{definition}

\begin{remark}
By Proposition \ref{prop:patched:unramified}, a Selmer structure depends only on the local conditions. If we set $\bH^1_{\FF}(K_\ku,T):=\bH^1_\f(K_\ku,T)$ whenever $\ku\notin \Sigma_\FF$, then 
\[\bH^1_{\FF}(K,T)=\ker\Biggl(\bH^1(K,T) \to \prod_{\ku\in \UU(\MM_K)} \frac{\bH^1(K_\ku,T)}{\bH^1_{\FF}(K_\ku,T)}\Biggr).\]
\label{rem:patched:sigma_independent}
\end{remark}

\begin{remark}
The global patched cohomology groups $\bH^1(K^\kn/K)$ can be recovered as the Selmer group of the maximal, everywhere unramified outside $\kn$ Selmer structure $\FF^{\ur,\kn}$, defined by the local conditions
\[\bH^1_{\FF^{\ur,\kn}}(K_\ku,T)=\left\{\begin{aligned}
&\bH^1_\f(K_\ku,T)\ \textrm{if $\ku\nmid \kn$},\\
&\bH^1(K_\ku,T)\ \textrm{if $\ku\mid \kn$}.
\end{aligned}\right.\]
\end{remark}

Local conditions propagate to quotients and submodules in the following way.

\begin{definition}
Let $\ T\hookrightarrow T''$ be a quotient map. For every ultraprime $\ku$, it induces a map
\[\varepsilon:\ \bH^1(K,T)\to\bH^1(K,T'')\]

A local condition at $T$ propagates to $T'$ as
\[\bH^1_\FF(K_\ku,T''):=\varepsilon\Bigl(\bH^1_\FF(K_\ku,T)\Bigr).\]
\label{def:propagation_quotients}
\end{definition}

\begin{definition}
Let $\ T'\twoheadrightarrow T$ be a submodule. For every ultraprime $\ku$, this inclusion induces a map
\[\mu:\ \bH^1(K_\ku,T')\to \bH^1(K_\ku,T).\]

A local condition at $T$ propagates to $T'$ as
\[\bH^1_\FF(K_\ku,T'):=\mu^{-1}\Bigl(\bH^1_\FF(K_\ku,T)\Bigr).\]
\label{def:propagation_submodules}
\end{definition}

Then Selmer groups can be recovered as the limit of the propagated Selmer groups with finite coefficients.

\begin{proposition}(\cite[Proposition 2.5.6]{Sweeting})
Let $T$ be a countably profinite abelian group and let $\FF$ be a Selmer structure defined on $T$. Then
\[\bH^1_{\FF}(K,T)=\varprojlim_{T\twoheadrightarrow T'} \bH^1_{\FF}(K,T'),\]
where the limit is taken over the finite quotients $T'$ of $T$ and the Selmer group $\bH^1_{\FF}(K,T')$ is obtained from the propagated Selmer structure, as defined in Definition \ref{def:propagation_quotients}.
\label{prop:selmer_profinite}
\end{proposition}

\begin{proposition}(\cite[Proposition 2.5.6]{Sweeting})
Let $T$ be an ind-finite group and let $\FF$ be a Selmer structure defined on $T$. Then
\[\bH^1_{\FF}(K,T)=\varinjlim_{T'\hookrightarrow  T} \bH^1_{\FF}(K,T'),\]
where the limit is taken over the finite submodules $T'$ of $T$ and $\bH^1_{\FF}(K,T')$ is the Selmer group of the propagated Selmer structure, as defined in Definition \ref{def:propagation_submodules}.
\label{prop:selmer_indfinite}
\end{proposition}

\subsection{Dual patched Selmer structures}

In this section, we extend the theory in \cite{Sweeting} to construct a global duality exact sequence of patched Selmer groups.
\begin{definition}[Dual Selmer structure]
    Let $\FF$ be a Selmer structure on $T$. We can define a \emph{dual Selmer structure} $\FF^*$ on the Cartier dual $T^*$ by defining the local condition $\bH^1_{\FF^*}(K_\ku,T^*)$ as the annihilator of $\bH^1_{\FF}(K_\ku,T)$ under the local duality pairing in Proposition \ref{prop:patched:local_duality}
    \label{def:patched:dual_structure}
\end{definition}

This construction can be used to obtain a global duality exact sequence in the patched setting.

\begin{proposition}(Global duality)
Let $\FF\leq\GG$ be two Selmer structures. Then there is a global duality exact sequence (recall Notation \ref{not:dualities}):

\begin{center}
\begin{tikzpicture}[descr/.style={fill=white,inner sep=1.5pt}]
        \matrix (m) [
            matrix of math nodes,
            row sep=4em,
            column sep=1em,
            text height=1.5ex, text depth=0.25ex
        ]
        {  0& \bH^1_{\FF}(K,T) & \bH^1_{\GG}(K,T) & \displaystyle{\bigoplus_{\ku\in \Sigma_\FF\cup\Sigma_\GG} \frac{\bH^1_{\GG}(K_\ku,T)}{\bH^1_{\FF}(K_\ku,T)}} \\
            & \bH^1_{\FF^*}(K,T^*)^\vee & \bH^1_{\GG^*}(K,T^*)^\vee& 0. &\\
        };

        \path[overlay,->, font=\scriptsize,>=latex]
        (m-1-1) edge (m-1-2)
        (m-1-2) edge (m-1-3)
        (m-1-3) edge (m-1-4)
        (m-1-4) edge [out=355,in=175] (m-2-2) 
        (m-2-2) edge (m-2-3)
        (m-2-3) edge (m-2-4)
        ;    
\end{tikzpicture}
\end{center}
\label{prop:patched:global_duality}
\end{proposition}

\begin{proof}
Assume first that $T$ is finite. Let $\mathfrak n=(n_k)_{k\in \N}$ be the square-free product of all ultraprimes in $\Sigma_\FF\cup \Sigma_{\GG}$. By Lemma \ref{lem:structure_patching} below, there are sequences of classical Selmer structures $(\FF_i)$ and $(\GG_i)$, where $\FF_i\leq \GG_i$, and 
\[\begin{array}{cc}
\bH^1_{\FF}(K_\ku,T)=\UU_i\Bigl(H^1_{\FF_i}(K_{\ell_i},T)\Bigr),\ \bH^1_{\GG}(K_\ku,T)=\UU_i\Bigl(H^1_{\GG_i}(K_{\ell_i},T)\Bigr).
\end{array}\]

In addition, Lemma \ref{lem:structure_patching} below also gives that
\[\begin{array}{cc}
\bH^1_{\FF}(K,T)=\UU_i\Bigl(H^1_{\FF_i}(K,T)\Bigr),\ \bH^1_{\GG}(K,T)=\UU_i\Bigl(H^1_{\GG_i}(K,T)\Bigr).
\end{array}\]

By Proposition \ref{prop:ultrafilter_exact} and Lemma \ref{lem:patching_structures_dual} below, the dual Selmer groups can be also obtained as an ultraproduct of classical Selmer groups:
\[\begin{aligned}
&\bH^1_{\FF^*}(K,T^*)=\UU_i\Bigl(H^1_{(\FF_i)^*}(K,T^*)\Bigr),\\ 
&\bH^1_{\GG^*}(K,T^*)=\UU_i\Bigl(H^1_{(\GG_i)^*}(K,T^*)\Bigr).
\end{aligned}\]

Then the global duality exact sequence is

\begin{center}
\begin{tikzpicture}[descr/.style={fill=white,inner sep=1.5pt}]
        \matrix (m) [
            matrix of math nodes,
            row sep=6em,
            column sep=1.5em,
            text height=1.5ex, text depth=0.25ex
        ]
        {  0& \UU\Bigl(H^1_{\FF_i}(K,T)\Bigr) & \UU\Bigl(H^1_{\GG_i}(K,T)\Bigr) & \UU\left(\displaystyle{\bigoplus_{\ell_i\mid n_i} \frac{H^1_{\GG_i}(K_{\ell_i},T)}{H^1_{\FF_i}(K_{\ell_i},T)}}\right) \\
            & \UU\Bigl(H^1_{\FF^*}(K,T^*)^\vee\Bigr) & \UU\Bigl(H^1_{\GG^*}(K,T^*)^\vee\Bigr)& 0, &\\
        };

        \path[overlay,->, font=\scriptsize,>=latex]
        (m-1-1) edge (m-1-2)
        (m-1-2) edge (m-1-3)
        (m-1-3) edge (m-1-4)
        (m-1-4) edge [out=355,in=175] (m-2-2) 
        (m-2-2) edge (m-2-3)
        (m-2-3) edge (m-2-4)
        ;    
\end{tikzpicture}
\end{center}
which is exact by Proposition \ref{prop:ultrafilter_exact}.

When $T$ is profinite (resp. ind-finite), then Proposition \ref{prop:selmer_profinite} (resp. Proposition \ref{prop:selmer_indfinite}) implies that the above exact sequence can be obtained as the inverse (resp. direct) limit of the associated exact sequences for the finite quotients (resp. submodules) of $T$. The proof is thus concluded since the inverse limit of finite groups (resp. direct limit) is an exact functor.
\end{proof}

\begin{lemma}
Let $\FF\leq \GG$ be two Selmer structures defined on a finite Galois module $T$. Then there are sequences of classical Selmer structures $(\FF_i)$ and $(\GG_i)$ such that $\FF_i\leq \GG_i$ for all $i\in \N$ and for any representative $(\ell_i)_{i\in \N}$ of an ultraprime $\ku$,
\[\begin{array}{cc}
\bH^1_{\FF}(K_\ku,T)=\UU\Bigl(H^1_{\FF_i}(K_{\ell_i},T)\Bigr),\ \bH^1_{\GG}(K_\ku,T)=\UU\Bigl(H^1_{\GG_i}(K_{\ell_i},T)\Bigr).
\end{array}\]
Moreover, we can construct the patched Selmer group as the ultraproduct of classical Selmer groups
\[\begin{array}{cc}
\bH^1_{\FF}(K,T)=\UU\Bigl(H^1_{\FF_i}(K,T)\Bigr),\ \bH^1_{\GG}(K,T)=\UU\Bigl(H^1_{\GG_i}(K,T)\Bigr).
\end{array}\]
\label{lem:structure_patching}
\end{lemma}

\begin{proof}
For an ultraprime $\ku \in \Sigma_\FF\cup \Sigma_\GG$, we have fixed a representing sequence $(\ell_i)$. We can assume, without loss of generality, that no prime appears in two different sequences on the same entry. Let $W_\ku$ be the set of indices such that there is an isomorphism 
\[\varphi_i^\ku:\ H^1(K_{\ell_i},T)\cong \bH^1(K_\ku,T).\]
 Note that Corollary \ref{cor:patched:local_constant} implies that $W_\ku\in \UU$. For every $i\in \N$, define the classical Selmer structure by the set of primes 
\[\Sigma_{\FF_i}=\{\ell_i:\ \ku=(\ell_i)_{i\in \N}\in \Sigma_\FF\cup \Sigma_{\GG}\}\]
 and local conditions 
\[\begin{aligned}
&H^1_{\FF_i}(K_{\ell_i},T)=(\varphi_i^\ku)^{-1} \bH^1_{\FF}(K_\ku,T)&\textrm{ if $i\in W_\ku$,}\\
&H^1_{\FF_i}(K_{\ell_i},T)=0&\textrm{ if $i\notin W_\ku$.}
\end{aligned}\]

This definition implies that $\bH^1_\FF(K_{\ku},T)=\UU_i\Bigl(H^1_{\FF}(K_{\ell_i},T)\Bigr)$. Therefore, Proposition \ref{prop:ultrafilter_exact} implies that
\[\begin{aligned}
    &\ker\Biggl[\bH^1(K^\kn/K,T)\to \prod_{\ku\mid \kn} \bH^1_{/\FF}(K_\ku,T)\Biggr]=\\
    &\ker\left[\UU_i\Bigl(H^1(K^{n_i}/K,T)\Bigr)\to \UU_i\left(\prod_{\ell_i\mid n_i} H^1_{/\FF}(K_{\ell_i},T)\right)\right].
\end{aligned}\]
Again, the exactness of the ultraproduct given in Proposition \ref{prop:ultrafilter_exact} implies that 
\[\bH^1_{\FF}(K,T)=\UU_i\Bigl(H^1_{\FF_i}(K,T)\Bigr).\]

Similarly, define Selmer structures $\GG_i$ by the set of primes 
\[\Sigma_{\GG_i}=\{\ell_i:\ \ku=(\ell_i)_{i\in \N}\in \Sigma_\FF\cup \Sigma_\GG\}\]
 and local conditions
\[\begin{aligned} 
&H^1_{\GG_i}(K_{\ell_i},T)=(\varphi_i^\ku)^{-1} \bH^1_{\GG}(K_\ku,T)&\textrm{ if $i\in W_\ku$,}\\
&H^1_{\GG_i}(K_{\ell_i},T)=H^1(K_{\ku_i},T)&\textrm{ if $i\notin W_\ku$.}
\end{aligned}\]
By construction $\FF_i\leq \GG_i$ and, analogously, 
\[\bH^1_{\GG}(K,T)=\UU_i\Bigl(H^1_{\GG_i}(K,T)\Bigr).\qedhere\]
\end{proof}

\begin{lemma}
Let $\FF$ a patched Selmer structure defined on a finite group $T$ and let $\FF_i$ be classical Selmer structures such that 
\[\bH^1_{\FF}(K_\ku,T)=\UU_i\Bigl(H^1_{\FF_i}(K_{\ell_i},T)\Bigr)\]
for every ultraprime $\ku=(\ell_i)_{i\in \N}$. Then, we have that
\[\bH^1_{\FF^*}(K_\ku,T^*)=\UU_i\Bigl(H^1_{\FF_i^*}(K_{\ell_i},T^*)\Bigr).\]
\label{lem:patching_structures_dual}
\end{lemma}

\begin{proof}
By Definition \ref{def:patched:dual_structure}, for every ultraprime $\ku=(\ell_i)_{i\in \N}$, we have that
\[\bH^1_{\FF^*}(K_\ku,T^*)=\ker\left(H^1(K_\ku,T^*)\to \Hom\left(H^1_{\FF}(K_\ku,T), \frac{(\# T)^{-1} \Z}{\Z}\right)\right), \]
where the map is induced by local duality, as stated in Proposition \ref{prop:patched:local_duality}. This map can be also written as the composition
\[\begin{aligned} 
    \UU_i\Bigl(H^1(K_{\ell_i},T^*)\Bigr)\to \UU_i\left(\Hom\left(  H^1_{\FF_i}(K_{\ell_i},T),\frac{(\# T)^{-1} \Z}{\Z}\right)\right)= \\ \Hom\left(\UU_i\Bigl(H^1_{\FF_i}(K_{\ell_i},T)\Bigr),\UU\left(\frac{(\# T)^{-1} \Z}{\Z}\right)\right),
\end{aligned}\]
where the equality follows from Proposition \ref{prop:ultraproduct_hom}. By local Tate duality for Galois cohomology and \ref{prop:ultrafilter_exact}, we can compute the kernel of this map as
\[\bH^1_{\FF^*}(K_\ku,T^*)=\UU\Bigl(H^1_{\FF_i^*}(K_{\ell_i},T^*)\Bigr).\qedhere\]
\end{proof}

\subsection{Selmer modules}
\label{sec:selmer:modules}

For the rest of this article, we will assume that $T$ admits an $R$-module structure where $R$ is a ring satisfying the following assumption.

\begin{namedass}{(Reg)}
Let $R$ be a commutative, complete, noetherian, regular local domain with maximal ideal $\m$ and finite residue field $k$ of characteristic $p$. 
\label{ass:R_lim}
\end{namedass}

\begin{remark}
By Auslander-Buchsbaum theorem, we now that $R$ is a unique factorisation domain. Fix, once and forall, a regular system of parameters $\{\pi_1,\ldots,\pi_n\}$. For each $k\in \N_{\geq 0}$, define the ideal and the quotient
\[\begin{array}{cc}
    I_k:=(\pi_1^k,\ldots,\pi_n^k),\ &R_k=R/I_k 
\end{array}\]
Since $R$ is complete, it can be recovered as the inverse limit.
\[R=\varprojlim_{k\in \N} R_k,\]
Since $R$ is a unique factorisation domain, $R_k[\m]$ is a one-dimensional $k$-vector space, generated by $(\pi_1^{k-1}\cdots \pi_n^{k-1})$.
\label{rem:R_lim}
\end{remark}

\begin{remark}
For every $m<n$, the ideal $R_n[I_m]$ is a principal ideal of $R_n$. Indeed, since $R_n$ is self-injective, it can be seen as
\[R_n[I_m]:=\ker\left(\Hom(R_n,R_n)\to \Hom(I_m,R_n)\right)=\Hom(R_m,R_n)\]
Since $R_m$ is self-injective, then $R_m[\m]$ is a one-dimensional $k$-vector space. Hence $R_n[I_m]/\m$ is also one-dimensional as a $k$-vector space, so Nayayama's lemma implies that $R_n[I_m]$ is principal.
\label{rem:ann_prin}
\end{remark}

\begin{notation}
In addition, let $T$ be a $R[[G_K]]$-module that is finitely generated as an $R$-module, in which the Galois action ramifies only at finitely many primes. Similarly, we denote $T_k=T\otimes_R R_k$. It is important to note that, since $R$ is complete and $T$ is finitely generated, then $T_k$ is a cofinal sequence, so 
\[T=\varprojlim_{n\in \N} T_k.\]
\end{notation}

\begin{proposition}
Let $T$ be as in Assumption \ref{ass:R_lim} and let $\kn$ be a square-free product of ultraprimes. Then $\bH^1(K^\kn/K,T)$ is a finitely generated $R$-module.
\label{prop:fingen}
\end{proposition}

\begin{proof}
    We proceed by induction on the Krull dimension of $R$. When the dimension is zero, then $R$ is a finite field, so the proposition follows from Corollary \ref{cor:global_finite}.
    
    For the general case when $\dim(R)>0$, we can pick an element $\pi\in \m\setminus \m^2$. Then $R/(\pi)$ is a regular ring satisfying Assumption \ref{ass:R_lim} and with Krull dimension $\dim(R/(\pi))$ equal to $\dim(R)-1$.

    From Propositions \ref{prop:patched:long_profinite} and \ref{prop:BCP_global} applied to the short exact sequence
    \[\xymatrix{0\ar[r] & T\ar[r]^{\pi} & T\ar[r] & T/\pi T\ar[r] &0,}\]
    we obtain an injection
    \[\bH^1(K^\kn/K,T)\otimes_R R/(\pi)\hookrightarrow \bH^1(K^\kn/K,T/\pi T).\]
    From the induction hypothesis, we deduce that $\bH^1(K^\kn/K,T)\otimes_R R/(\pi)$ is a finitely-generated $R/(\pi)$-module. Therefore, Nakayama's lemma implies that $\bH^1(K^\kn/K,T)$ is finitely generated as $R$-module.
\end{proof}

The advantage of patched cohomology is that we can develop a theory of Kolyvagin systems for infinite rings, such as those satisfying Assumption \ref{ass:R_lim}. That was not possible in the classical setting, since the finiteness of the ring $R$ was a requirement for constructing Kolyvagin primes using the Chebotarev density theorem. In order to develop the theory of Kolyvagin systems, we need to make the standard big image assumption.

\begin{namedass}{(BI)}
We assume the following:
\begin{itemize}
\item \namedlabel{pTirred}{(T1)} $T/\m T$ is an irreducible $k[[G_K]]$-module.
\item \namedlabel{pTtau}{(T2)} There exists $\tau\in G_{K(\mu_{p^\infty})}$ such that $T/(\tau-1)T\cong R$ as $R$-modules.
\item \namedlabel{pTcoh}{(T3)} $H^1(K(T)_{p^\infty}/K,T)=H^1(K(T)_{p^\infty}/K,T^*(1))=0$, where $K(T)_{p^\infty}$ is the minimal extension $F/K(\mu_{p^\infty})$ such that $G_F$ acts trivially on $T$. 
\end{itemize}
\label{ass:patched:basic}
\end{namedass}

\begin{remark}
By \ref{pTirred}, we know that $(T/\m T)^{G_K}$ is either $0$ or $T/\m T$. However, the latter is not possible since it does not satisfy $\ref{pTcoh}$.

That implies that $(T/IT)^{G_K}=0$ for every ideal $I$ of $R$ and, dualizing, $(T^*[I])^{G_K}$. It can be also seen in terms of patched cohomology with the equalities
\[\bH^0(K,T/I)=\bH^0(K,T^*[I])=0.\]
\label{rem:patched:H0vanish}
\end{remark}

We can use this to show that, under these assumptions, profinite Selmer groups are torsion free.

\begin{proposition}
Let $T$ be a Galois representation satisfying Assumption \ref{ass:patched:basic}. For every Selmer structure $\FF$, the Selmer group
\[\bH^1_{\FF}(K,T)\]
is a torsion-free $R$-module.
\label{prop:patched:selmer_torsion_free}
\end{proposition}

\begin{proof}
Let $a$ in $R$ and consider the exact sequence
\[\xymatrix{0\ar[r] & T\ar[r] & T\ar[r] & T/aT\ar[r] & 0.}\]
By Corollary \ref{cor:patched:global_long_exact} and Remark \ref{rem:patched:H0vanish}, we have that 
\[\bH^1(K^{\Sigma_\FF}/K,T)[a]=0.\]
Since that holds for every $a\in R$ and Definition \ref{def:patched:selmer} defines the Selmer group as a subgroup of the global patched cohomology group, we can deduce that 
\[\bH^1_{\FF}(K,T)_{\tors}=0.\qedhere\]
\end{proof}

\subsection{Cartesian Selmer structures and core rank}

In order to develop this theory, we need to impose the cartesian condition on Selmer structures, as defined below.

\begin{definition}[Cartesian local condition]
A local condition $\bH^1_{\FF}(K_\ku,T)$ is called \emph{cartesian} if the quotient 
\[\bH^1_{/\FF}(K_\ku,T)=\frac{\bH^1(K_\ku,T)}{\bH^1_{\FF}(K_\ku,T)}\]
is a torsion-free $R$-module. A Selmer structure is said to be cartesian if all of its local conditions are cartesian.
\label{def:patched:cartesian}
\end{definition}

\begin{remark}
The connection of Definition \ref{def:patched:cartesian} with the cartesian condition appearing in the previous theory of Kolyvagin systems, first introduced in \cite[Definition 1.1.4]{MazurRubin} lies in \cite[Lemma 3.7.1]{MazurRubin}, where it is proven that both definitions are equivalent when $R$ is a discrete valuation ring. When $R$ is a more complicated ring, such as the Iwasawa algebra, Definition \ref{def:patched:cartesian} is a milder assumption (see Proposition \ref{prop:iwasawa:free_cartesian_propagated} below).
\end{remark}

\begin{remark}
In order to check that a Selmer structure $\FF$ is cartesian, it is enough to see that local conditions are cartesians for ultraprimes $\ku\in \Sigma_{\FF}$. Indeed, finite cohomology groups are always cartesian local conditions since 
\[\bH^1_s(K_\ku,T)\subset \Hom(I_\ku,T)\]
is $R$-torsion-free since $R$ is a domain and $T$ is a free module.
\end{remark}

\begin{remark}
Considering the computations in Proposition \ref{prop:iwasawa:free_cartesian_propagated} below when $R$ is the Iwasawa algebra, it might seem reasonable to define cartesian conditions imposing that $\bH^1_{\FF}(K_\ku,T)$ is a free $R$-module. However, that definition presents some technical issues since we cannot guarantee that $\bH^1_\s(K_\ku,T)$ is a free module in this generality .
\end{remark}

The main advantage of cartesian structures is that they allow the study of Selmer groups of $T$ by projecting to the cohomology of $T_1$. In order to do that, we assume the following mild assumption.

\begin{namedass}{(Proj)}
There is a square-free product of ultraprimes $\kn$, divisible by all the primes in $\Sigma_\FF$, such that the following projection map is injective:
\[\bH^1(K^{\kn}/K,T)\hookrightarrow \bH^1(K^{\kn}/K,T_1).\]
\label{ass:patched:quotient}
\end{namedass}

\begin{remark}
When Assumption \ref{ass:patched:basic} hold and $R$ is a principal ring, then Assumption \ref{ass:patched:quotient} also holds for every square-free product of ultraprimes $\kn$. Indeed, Remark \ref{rem:patched:H0vanish} implies that $\bH^0(K^{\kn}/K,T)=0$ for every square-free product of ultraprimes $\kn$. Then the short exact sequence
\[\xymatrix{0\ar[r] & T\ar[r] & T\ar[r] & T_1\ar[r] & 0}\]
induces, by Corollary \ref{cor:patched:global_long_exact}, an injection
\[\bH^1(K^{\kn}/K,T)\otimes R/\m\hookrightarrow \bH^1(K^{\kn}/K,T_1).\]
\label{rem:patched:ass_quotient}
\end{remark}

\begin{remark}
When $R$ is the Iwasawa algebra, it will be shown in Proposition \ref{prop:iwasawa:ass_equiv} that Assumption \ref{ass:patched:quotient} is equivalent to the absence finite submodules in the local cohomology groups.
\end{remark}

\begin{proposition}
Let $\FF$ be a cartesian Selmer structure. If $\kn$ is a square-free product of all the ultraprimes in $\Sigma_{\FF}$, for every ideal $I$ of $R$, there is an injection
\begin{equation}
\bH^1_{\FF}(K,T)/I\hookrightarrow \bH^1(K^{\kn}/K,T)/I.
\label{eq:patched:quot_selmer_inj}
\end{equation}
\label{prop:patched:quot_selmer_inj}
\end{proposition}

\begin{proof}
By Definition \ref{def:patched:selmer} and Remark \ref{rem:patched:sigma_independent}, there is an injection
\[\frac{\bH^1(K^{\kn}/K,T)}{\bH^1_{\FF}(K,T)}\hookrightarrow \prod_{\ku\mid \kn} \frac{\bH^1(K_\ku,T)}{\bH^1_{\FF}(K_\ku,T)}.\]
In this map, the cartesian condition in Definition \ref{def:patched:cartesian} implies the codomain is $R$-torsion free. It implies that
\begin{equation}
\left(\frac{\bH^1(K^{\kn}/K,T)}{\bH^1_{\FF}(K,T)}\right)[I]=0.
\label{eq:patched:selmer_quotient}
\end{equation}
Then the proof is completed by Lemma \ref{lem:quot_tors_free} below.
\end{proof}

\begin{lemma}
Let $\mu:\ A\hookrightarrow B$ be an injection of $R$-modules such that the quotient $B/\mu(A)$ is $I$-torsion free. Then, for every ideal $I\subset R$, the map $\mu$ induces an injection
\[\overline{\mu}:\ A/I\hookrightarrow B/I.\]
\label{lem:quot_tors_free}
\end{lemma}

\begin{proof}
Assume for the sake of contradiction that $\overline{\mu}$ is not injective. Then there exists some $a\in A$ such that 
\[\begin{array}{cc}
    a\notin IA,\ & \mu(a)\in IB.
\end{array}\]
Then there exists some $b\in B$ and $\lambda\in I$ such that $\mu(a)=\lambda b$. Since $a$ is not $I$-divisible in $A$, then $b\notin\mu(A)$. Hence $b+\mu(A)$ induces a non-trivial $I$-torsion class in the quotient $B/\mu(A)$, which leads to a contradiction.
\end{proof}

\begin{corollary}
If Assumption \ref{ass:patched:quotient} holds and $\FF$ is a cartesian Selmer structure, the following projection map is injective.
\[\bH^1_{\FF}(K,T)\otimes R/\m\hookrightarrow \bH^1_{\FF}(K,T_1).\]
\label{cor:patched:selmer_quotient}
\end{corollary}

\begin{proof}
Let $\kn$ be the square-free product of ultraprimes given by Assumption \ref{ass:patched:quotient}. Consider the commutative diagram
\[\xymatrix{
\bH^1_{\FF}(K,T)\otimes R/\m \ar[r] \ar@{>->}[d]  & \bH^1_{\FF}(K,T_1)\ar[d]\\
\bH^1(K^{\kn}/K,T)\otimes R/\m \ar@{>->}[r]   & \bH^1(K^{\kn}/K,T_1).
}\]
In this diagram, the left vertical map is injective by Proposition \ref{prop:patched:quot_selmer_inj} and the bottom horizontal map is injective by Assumption \ref{ass:patched:quotient}. Hence the top horizontal map needs to be injective.
\end{proof}

The previous corollary implies that the analogous maps induced by the quotients $R_k$ are also injective.
\begin{proposition}
If Assumption \ref{ass:patched:quotient} hold and $\FF$ is a cartesian Selmer structure, the map
\[\bH^1_{\FF}(K,T)\otimes_R R_k\hookrightarrow \bH^1_{\FF}(K,T_k)\]
is injective for all $k\in \N$.
\label{prop:patched:selmer_quotient_higherk}
\end{proposition}

\begin{proof}
Let
\[x\in \ker\biggl(\bH^1_\FF(K,T)\to \bH^1_{\FF}(K,T_{k})\biggr).\]
Then, $x$ is also in the kernel of the projection map to $\bH^1_{\FF}(K,T_{1})$, so Corollary \ref{cor:patched:selmer_quotient} implies the existence of an element $a\in \m$ and $y\in\bH^1_{\FF}(K,T)$ such that $x=ay$.

We repeat this process with the class $y$ until we find an element that is not in the kernel of the projection map to $\bH^1_{\FF}(K,T_1)$, so we construct an element $z\in \bH^1_{\FF}(K,T)$ such that $x=bz$ for some $b\in R$. We aim to show that $b\in I_k$. Note that, in case that the condition on $z$ is never achieved, we will eventually find an element $b\in I_k$, which would prove the proposition.

Assume by contradiction that it is not and let $\overline{b}$ be the projection in $R_k$. By assumption, 
\[\Pi(z)\in \bH^1_{\FF}(K,T_k)[\overline{b}],\]
where $\Pi$ represents the canonical projection map. However, since $\overline{b}\neq 0$ divides any element in $R_k[\m]$, Lemma \ref{lem:proj_finite_kernel} below implies that the projection of $z$ to the Selmer group of $T_1$ vanishes, which leads to a contradiction.
\end{proof}

\begin{lemma}
Let $I=R_k[\m]$. Then the following holds:
\[\ker\Bigl(\bH^1_{\FF}(K,T_k)\to \bH^1_{\FF}(K,T_1)\Bigr)=\bH^1_{\FF}(K,T_k)[I],\]
where the map is induced by the projection map $T_k\twoheadrightarrow T_1$.
\label{lem:proj_finite_kernel}
\end{lemma}

\begin{proof}
Since $R_k$ is self-injective, $I$ is the minimal non-zero ideal of $R_k$, so it is principal. Choose a generator $x\in I$. Since $T_k$ is a free $R_k$-module, there is a short exact sequence
\[\xymatrix{0\ar[r] &T_1\ar[r]^{[x]} \ar[r]& T_k \ar[r] & T_k/T_1\ar[r] & 0.}\]
Since $T_k/T_1=T/IT$, where $I$ is seen as an ideal of $R$ containing $I_k$, Corollary \ref{cor:patched:global_long_exact} and Remark \ref{rem:patched:H0vanish}  induces an injection
\[[x]:\ \bH^1_{\FF}(K,T_1)\hookrightarrow \bH^1_{\FF}(K,T_k).\]
Note that the following composition is the multiplication by $x$:
\[\xymatrix{\bH^1_{\FF}(K,T_k) \ar[r] & \bH^1_{\FF}(K,T_1)\ar@{>->}[r]^{[x]} &  \bH^1_{\FF}(K,T_k).}\]
Since the second map is injective, then
\[\ker\Bigl(\bH^1_{\FF}(K,T_k)\to \bH^1_{\FF}(K,T_1)\Bigr)=\bH^1_{\FF}(K,T_k)[x].\qedhere\]
\end{proof}

The next result can be seen as an analogue of the previous theory in the dual representation $T^*$, concerning submodules instead of quotients. Its prove is the same as the one in \cite[Proposition 3.5 and Corollary 3.8]{BurnsSakamotoSano2}

\begin{proposition}(\cite[Corollary 3.8]{BurnsSakamotoSano2})
Under Assumptions \ref{ass:patched:basic}, for every finitely generated ideal $I$ of $R$, the inclusion $T^*[I]\hookrightarrow T^*$ induces an isomorphism
\[\bH^1_{\FF^*}(K,T^*[I])\cong \bH^1_{\FF^*}(K,T^*)[I].\]
\label{prop:patched:selmer_torsion}
\end{proposition}

\begin{remark}
After dualizing, Proposition \ref{prop:patched:selmer_torsion} can be restated as
\[\bH^1_{\FF^*}(K,T^*[I])^\vee\cong \bH^1_{\FF^*}(K,T^*)^\vee\otimes R/I.\]
\label{rem:patched:selmer_torsion_dual}
\end{remark}

\begin{corollary}
If Assumption \ref{ass:patched:basic} holds, $\bH^1_{\FF^*}(K,T^*)^\vee$ is a finitely generated $R$-module.
\label{cor:fingen_dual}
\end{corollary}

\begin{proof}
By Corollary \ref{cor:global_finite} and Remark \ref{rem:patched:selmer_torsion_dual}, $\bH^1_{\FF^*}(K,T^*)^\vee\otimes R/\m$ is finite. Hence, Nakayama's lemma implies that $\bH^1_{\FF^*}(K,T^*)^\vee$ is a finitely generated $R$-module.
\end{proof}

We can now define the concept of core rank. We give a definition that applies for any Selmer structure $\FF$, not necessarily cartesian.

\begin{definition}
    Let $\FF$ be a Selmer structure. We define the \emph{core rank} of $\FF$ as
    \[\chi(\FF):=\rank_R\ \bH^1_{\FF}(K,T)-\rank_R\ \bH^1_{\FF^*}(K,T^*)^\vee.\]
    \label{def:patched:core_rank}
\end{definition}

\begin{remark}
The core rank is well defined since $\bH^1_{\FF}(K,T)$ and $\bH^1_{\FF^*}(K,T^*)^\vee$ are both finitely generated modules by Proposition \ref{prop:fingen} and Corollary \ref{cor:fingen_dual}.
\end{remark}

\subsection{Kolyvagin ultraprimes}

In the patched cohomology setting, we can define Kolyvagin ultraprimes in a similar way as was done in \cite{MazurRubin}. Recall that Assumption \ref{ass:patched:basic} implies the existence of $\tau\in G_{K_{p^\infty}}$ such that $T/(\tau-1)$ is a free $R$-module of rank one. Once and for all, we fix an automorphism $\tau\in G_K$ satisfying this assumption.

\begin{definition}
Let $T$ be a Galois representation and let $\FF$ be a Selmer structure defined on $T$. An ultraprime $\ku$ is said to be a \emph{Kolyvagin ultraprime} if $\ku\notin \Sigma_\FF$ and $\Frob_\ku$ is conjugate to $\tau$ in $\Gal(K(T)_{p^\infty}/K)$. 
\label{def:kolyvagin_ultraprime}
\end{definition}

\begin{notation}
We denote by $\PP(T)$ the set of Kolyvagin ultraprimes, and we denote by $\NN(\PP)$ and $\NN_i(\PP)$ the set of square-free products of Kolyvagin primes and exactly $i$ Kolyvagin primes, respectively. When it is clear from the context, we will also denote them by $\PP$, $\NN$ or $\NN_i$.
\end{notation}

\begin{notation}
For $\kn\in \NN$, we denote by $\nu(\kn)$ the number of ultraprimes dividing $\kn$.
\end{notation}

\begin{remark}
The existence of Kolyvagin ultraprimes is guaranteed by Proposition \ref{prop:ultraprimes_chebotarev} . Note that $K(T)_{p^\infty}/K$ might be an infinite extension, so we cannot apply the classical version of the Chebotarev density theorem to guarantee the existence of Kolyvagin primes.
\end{remark}

We can describe explicitly the local cohomology of Kolyvagin ultraprimes.
\begin{proposition}
Let $\ku$ be a Kolyvagin ultraprime. Then there is an isomorphism
\[\bH^1_\f(K_\ku,T)\cong T/(\Frob_\ku -1)T\cong T/(\tau-1)T\cong R.\]
\label{prop:patched:finite_descr}
\end{proposition}

\begin{proof}
Recall that $T_k=T\otimes_R R_k$ and let $(\ell_i)_{i\in \N}$ be the Kolyvagin ultraprime. For every $k\in \N$, denote by $\alpha_k\in \N$ the minimal natural number such that $p^{\alpha_k} R_k=0$. 

By construction, for $\UU$-many $i$, the Frobenius element $\Frob_{\ell_i}$ is conjugate to $\tau$ in $\Gal(K(T_k)_{p^{\alpha_n}}/K)$. Therefore, $\ell_i$ is a Kolyvagin prime for $\UU$-many $i$. For those $i$, \cite[Lemma 1.2.1]{MazurRubin} implies that 
\[H^1_\f(K_{\ell_i},T_k)\cong \frac{T_k}{(\Frob_{\ell_i} -1)T_k}\cong \frac{T_k}{(\tau-1)T_k}\cong R_k.\]
By Proposition \ref{prop:ultrafilter_finite}, we obtain that
\[\bH^1_\f(K_\ku,T_k)=\UU_i\Bigl(H^1_{\f}(K_{\ell_i},T_k)\Bigr)\cong \frac{T_k}{(\Frob_\ku-1)T_k}\cong \frac{T_k}{(\tau-1)T_k}\cong R_k.\]
Taking limits, we get that 
\[\bH^1_\f(K_\ku,T)\cong \frac{T}{(\Frob_\ku-1)T}\cong \frac{T}{(\tau-1)T}\cong R.\]
Note that the limit compatibility follows from the surjectivity of the transition maps, deduced by from the surjectivity of the analogous maps in Galois cohomology, and the fact that $R$ is a Noetherian ring, all the modules in the previous identifications are free $R$-modules of rank one.
\end{proof}

We can use a similar argument for constructing the patched transverse local conditions.

\begin{definition}
Let $\ku=(\ell_i)_{i\in \N}\in\PP $ be a Kolyvagin ultraprime. For every $i\in \N$, denote by $K(\ell_i)$  the maximal $p$-extension inside the ray class field modulo $\ell_i$.  We define the transverse local condition as 
\[\bH^1_\tr(K_\ku,T)=\varprojlim_{k\in \N} \UU_i\biggl(\Im\left(\inf:\ H^1(K(\ell_i)_{\ell_i}/K,T_k)\to H^1(K_{\ell_i},T_k)\right)\biggr).\]
\end{definition}

We can now prove the basic properties of the transverse cohomology used to develop the theory of Kolyvagin systems.

\begin{proposition}
Let $\ku=(\ell_i)_{i\in \N}\in\PP $ be a Kolyvagin ultraprime. Then the transverse cohomology group is a free, cyclic $R$-module and the local patched cohomology splits as
\[\bH^1(K_\ku,T)=\bH^1_\f(K_\ku,T)\oplus \bH^1_\tr(K_\ku,T).\]
\label{prop:patched_transverse}
\end{proposition}

\begin{proof}
For every $k\in \N$, denote by $S_k$ the set of indices such that $\ell_i\in \PP(T_k)$. Note that $S_k\in \UU$ for all $k\in \N$. For $i\in S_k$, \cite[Lemmas 1.2.1 and 1.2.4]{MazurRubin} imply that 
\[\begin{array}{cc}
    H^1_{\tr}(K_{\ell_i},T_k)\cong R_k,\ & H^1(K_{\ell_i},T_k)\cong H^1_\f(K_{\ell_i},T_k)\oplus H^1_\tr(K_{\ell_i},T_k).
\end{array}\]
By Lemma \ref{lem:ultrafilter_finite_diagonal}, the finite patched version also satisfies that 
\begin{equation}
\bH^1_{\tr}(K_\ku,T_k)=\UU_i\Bigl(H^1_{\tr}(K_{\ell_i},T_k)\Bigr)\cong R_k.
\label{eq:tr:artinian_iso}
\end{equation}
Taking limits, we obtain that 
\[\bH^1_{\tr}(K_\ku,T)=\varprojlim_{k\in \N}\bH^1_{\tr}(K_\ku,T_k)\cong\varprojlim_{n\in \N} R_k=R.\]
It is worth mentioning that, since the isomorphisms in \eqref{eq:tr:artinian_iso} are not canonical, we cannot guarantee that the transition maps in both limits coincide. However, it is possible to guarantee that the transition maps are surjective, which implies that $\bH^1_{\tr}(K_\ku,T)$ is a free module of rank one.

Remark \ref{rem:ultraproduct_product} implies that 
\[\bH^1(K_\ku,T_k)\cong \bH^1_\f(K_\ku,T_k)\oplus \bH^1_\tr(K_\ku,T_k).\]
And, taking limits, 
\[\bH^1(K_\ku,T)\cong \bH^1_\f(K_\ku,T)\oplus \bH^1_\tr(K_\ku,T).\qedhere\]
\end{proof}

\begin{remark}
Note that Proposition \ref{prop:patched_transverse} gives a canonical isomorphism 
\[\bH^1_{\tr}(K_\ku,T)\cong \bH^1_\s(K_\ku,T).\]
\end{remark}

We can also construct a finite-singular map in the patched setting. In order to do that, we need to define a generalisation of the Galois groups $\GG_\ell:=\Gal(K(\ell)/K)$ for ultraprimes.

\begin{definition}
Let $\ku=(\ell_i)_{i\in \N}\in \UU(\MM_K)$ be an ultraprime. Define the group
\[\GG_\ku:=\varprojlim_{k\in \N} \UU_i\Bigl(\GG_{\ell_i}/p^k\Bigr).\]
\label{def:patched:unramified_global}
\end{definition}

\begin{remark}
Note that Definition \ref{def:patched:unramified_global} is independent of the sequence representing the ultraprime. In addition, if $\ku$ is the constant ultraprime $(\ell)$, then $\GG_\ku$ coincides with the group $\GG_\ell$.
\end{remark}

\begin{remark}
If $\ku$ is a Kolyvagin ultraprime, then $\GG_\ku$ is isomorphic to $\Z_p$.
\end{remark}

\begin{proposition}
There is a canonical isomorphism 
\[\bH^1_\s(K_\ku,T)\cong \Hom(I_\ku,T)^{\Frob_\ku=1}.\]
\label{prop:patched:singular}
\end{proposition}

\begin{proof}
This is obtained from the inflation-restriction sequence, in the identifications of Proposition \ref{prop:local_patched_galois_group}, and the cohomological dimension of $\langle \Frob_\ku\rangle$ being one.
\end{proof}

We can restate the isomorphism in Proposition \ref{prop:patched:singular} in the following way:
\begin{corollary}
If $\ku$ is a Kolyvagin ultraprime, there is a canonical isomorphism 
\[\bH^1_\s(K_\ku,T)\otimes_{\Z_p} \GG_\ku\cong T^{\Frob_\ku=1}.\]
\label{cor:patched:singular}
\end{corollary}

\begin{proof}
By Definition \ref{def:local_galois_group}, there is an isomorphism
\[\Hom(I_\ku,T)\to \Hom(\Z_p(1),T).\]

Since $\ku=(\ell_i)_{i\in \N}$ is a Kolyvagin ultraprime, then there is a canonical isomorphism $\GG_{\ell_i}/p^k=\mu_{p^k}$ for $\UU$-many $i$ and all $k\in \N$. Those isomorphisms can be patched to give a canonical identification $\GG_{\ku}= \mu_{p^\infty}$. Moreover, the fact that $\ku$ is a Kolyvagin prime implies that $\Frob_\ku$ acts trivially on the $p$-adic roots of unity, so $\Z_p$ and $\Z_p(1)$ are the same as $\langle \Frob_\ku\rangle$-modules. By Proposition \ref{prop:patched:singular}
\[\bH^1_\s(K_\ku,T)\otimes \GG_\ku\cong \Hom(\Z_p(1),T)^{\Frob_\ku=1}\otimes \GG_\ku=T^{\Frob_\ku=1}.\qedhere\]
\end{proof}

Finite-singular maps also appear in this setting.

\begin{proposition}
If $\ku$ is a Kolyvagin ultraprime, there is a canonical isomorphism, known as \emph{finite-singular map}, 
\[\phi^{\fs}_\ku:\ \bH^1_\f(K_\ku,T)\to \bH^1_\s(K_\ku,T)\otimes_{\Z_p} \GG_\ku.\]
\label{prop:patched:finite-singular}
\end{proposition}

\begin{proof}
Let $P_\ku$ be the characteristic polynomial of the Galois action of $\Frob_\ku$ on $T$. The fact that $T/(\Frob_\ku-1)$ is one dimensional, implies that $P_\ku(X)=(X-1) Q_\ku(X)$, where $Q_\ku$ is a polynomial satisfying that $Q_\ku(1)\neq 0$. Note that, since $P_\ku(\Frob_\ku)$ acts trivially on $T$, then $Q_\ku(\Frob_\ku)$ induces a map
\[Q_\ku(\Frob_\ku):\ T/(\Frob_\ku-1)T\to T^{\Frob_\ku=1}.\]
Consider then the composition,
\[\xymatrix{\bH^1_\f(K_\ku,T)\ar[r] & T/(\Frob_\ku-1)T\ar[r] & T^{\Frob_\ku=1}\ar[r] & \bH^1_\s(K_\ku,T)\otimes\GG_\ku}.\]
The first and third maps are isomorphisms by Proposition \ref{prop:patched:finite_descr} and Corollary \ref{cor:patched:singular}. Hence it only remains to see that $Q_\ku(\Frob_\ku)$ induces an isomorphism. Since $T/\m$ is an $R/\m$-vector space, the following map is an isomorphism:
\[\overline{Q_\ku(\Frob_\ku)}:\ T/(\Frob_\ku-1,\m)T\to (T/\m)^{\Frob_\ku=1}.\]
Nakayama's lemma then implies that $Q_\ku(\Frob_\ku)$ is surjective. Since $T/(\Frob_\ku-1)T$ and $T^{\Frob_\ku=1}$ are free $R$-modules of the same rank, then it is in fact an isomorphism.
\end{proof}

\begin{remark}
Note that, after choosing a generator of $\GG_\ku$, the finite-singular map can be written as a well-determined map
\[\phi^{\fs}_\ku:\ \bH^1_\f(K_\ku,T)\to \bH^1_\s(K_\ku,T).\]
Note that, when $\ku$ is represented by the sequence $(\ell_i)_{i\in \N}$, then choosing a generator of $\GG_\ku$ is the same as choosing a generator of $\GG_{\ell_i}$ for $\UU$-many $i$. In this case, the finite-singular maps $\phi_\ku^{\fs}$ can be obtained by patching $\phi_{\ell_i}^{\fs}$ in a natural way.
\label{rem:finite-singular_generator}
\end{remark}

We can now use these local conditions to modify Selmer structures, analogously to \cite[Example 2.8]{MazurRubin}
\begin{definition}
Let $\FF$ be a Selmer structure and let $\ka$, $\kb$ and $\kc$ be pairwise coprime square-free products of ultraprimes. Assume $\kc\in \NN$. Define the Selmer structure $\FF_\ka^\kb(\kc)$ by the local conditions
\[\bH^1_{\FF_\ka^\kb(\kc)}(K_\ku,T)=\left\{\begin{aligned}
&0\ &\textrm{if }\ku\mid\ka\\
&\bH^1(K_\ku,T)\ &\textrm{if }\ku\mid\kb\\
&\bH^1_{\tr}(K_\ku,T)\ &\textrm{if }\ku\mid\kc\\
&\bH^1_\FF(K_\ku,T)\ &\textrm{otherwise.}
\end{aligned}\right.\]
\label{def:patched:modified}
\end{definition}

Modified Selmer structures are cartesian and their core rank varies in a controlled way. The following result is a generalisation of (\cite[Corollary 3.21]{Sakamoto18}) to the patched cohomology setting.

\begin{proposition}
Let $\FF$ be a cartesian Selmer structure and let $\ka,\kb,\kc\in \NN$ be pairwise coprime. Then $\FF_\ka^\kb(\kc)$ is also cartesian and 
\[\chi(\FF_\ka^\kb(\kc))=\chi(\FF)+\nu(\kb)-\nu(\ka).\]
\label{prop:patched:rank_modified}
\end{proposition}

\begin{proof}
For every Kolyvagin prime $\ku\in \PP$, both $\bH^1(K_\ku,T)$ and $\bH^1_\tr(K_\ku,T)$ are torsion-free, so $\FF_\ka^\kb(\kc)$ is cartesian.

Proposition \ref{prop:patched:global_duality} produces an exact sequence
\begin{center}
\begin{tikzpicture}[descr/.style={fill=white,inner sep=1.5pt}]
        \matrix (m) [
            matrix of math nodes,
            row sep=3.5em,
            column sep=1.5em,
            text height=1.5ex, text depth=0.25ex
        ]
        {  0& \bH^1_{\FF}(K,T) & \bH^1_{\FF^{\kb\kc}}(K,T) & \displaystyle\bigoplus_{\ku\mid \kb\kc}\bH^1_\s(K_\ku,T) \\
            & \bH^1_{\FF^*}(K,T^*)^\vee & \bH^1_{(\FF^*)_{\kb\kc}}(K,T^*)^\vee & 0. \\
        };

        \path[overlay,->, font=\scriptsize,>=latex]
        (m-1-1) edge (m-1-2)
        (m-1-2) edge (m-1-3)
        (m-1-3) edge (m-1-4)
        (m-1-4) edge [out=355,in=175] (m-2-2) 
        (m-2-2) edge (m-2-3)
        (m-2-3) edge (m-2-4)
        ;    
\end{tikzpicture}
\end{center}

Since the rank is additive, then $\chi\bigl(\FF^{\kb\kc}\bigr)=\chi(\FF)+\nu(\kb)+\nu(\kc)$. Proposition \ref{prop:patched:global_duality} produces another exact sequence

\begin{center}
\begin{tikzpicture}[descr/.style={fill=white,inner sep=1.5pt}]
        \matrix (m) [
            matrix of math nodes,
            row sep=3.5em,
            column sep=1em,
            text height=1.5ex, text depth=0.25ex
        ]
        {  0& \bH^1_{\FF_\ka^\kb(\kc)}(K,T) & \bH^1_{\FF^{\kb\kc}}(K,T) & \displaystyle\bigoplus_{\ku\mid \ka\kc}\bH^1_\f(K_\ku,T) \\
            & \bH^1_{(\FF^*)^\ka_\kb(\kc)}(K,T^*)^\vee & \bH^1_{(\FF^*)_{\kb\kc}}(K,T^*)^\vee & 0. \\
        };

        \path[overlay,->, font=\scriptsize,>=latex]
        (m-1-1) edge (m-1-2)
        (m-1-2) edge (m-1-3)
        (m-1-3) edge (m-1-4)
        (m-1-4) edge [out=355,in=175] (m-2-2) 
        (m-2-2) edge (m-2-3)
        (m-2-3) edge (m-2-4)
        ;    
\end{tikzpicture}
\end{center}
Thus,
\[\chi\bigl(\FF^\kb_\ka(\kc)\bigr)=\chi\bigl(\FF^{\kb\kc}\bigr)-\nu(\ka)-\nu(\kc)=\chi\bigl(\FF\bigr)+\nu(\kb)-\nu(\ka).\qedhere\]
\end{proof}

The analogue of \cite[Lemma 4.1.7(ii)]{MazurRubin} is a central technical result when working with Kolyvagin systems.

\begin{lemma}
Let $\FF$ be a Selmer structure and let $\ku\in\PP$ be an ultraprime satisfying that 
\[\loc_\ku:\ \bH^1_{\FF}(K,T)\to \bH^1_\f(K_\ku,T)\]
is surjective. Then $\bH^1_{(\FF^*)(\ku)}(K,T^*)=H^1_{(\FF^*)_\ku}(K,T^*)$.
\label{lem:patched:tr_res}
\end{lemma}

\begin{proof}
The surjectivity of $\loc_\ku$, together with the exact sequence of Proposition \ref{prop:patched:global_duality} with Selmer structures $\FF_\ku$ and $\FF$ implies that 
\begin{equation}
\bH^1_{(\FF^*)^\ku}(K,T^*)=\bH^1_{\FF^*}(K,T^*).
\label{eq:patched:417}
\end{equation}
By construction, we have that 
\[\bH^1_{(\FF^*)_\ku}(K,T^*)=\bH^1_{\FF^*}(K,T^*)\cap \bH^1_{(\FF^*)(\ku)}(K,T^*).\]
By \eqref{eq:patched:417},
\[\bH^1_{(\FF^*)_\ku}(K,T^*)=\bH^1_{(\FF^*)^\ku}(K,T^*)\cap \bH^1_{(\FF^*)(\ku)}(K,T^*).\]
Since $\bH^1_{(\FF^*)(\ku)}(K,T^*)\subset \bH^1_{(\FF^*)^\ku}(K,T^*)$, we get that
\[\bH^1_{(\FF^*)_\ku}(K,T^*)=\bH^1_{(\FF^*)(\ku)}(K,T^*).\qedhere\]
\end{proof}

\subsection{Chebotarev density theorem}
\label{sec:cheb}

This section is devoted to proving analogues of \cite[Lemma 3.9]{BurnsSakamotoSano2} in the patched cohomology formalism. In the most general setting, we only need to work under Assumption \ref{ass:R_lim} on the coefficient ring $R$.

\begin{proposition}
Let $\kn$ be a square-free product of ultraprimes, let 
\[c_1,\ldots,c_s\in \bH^1(K^\kn/K,T)\setminus \m \bH^1(K^\kn/K,T)\] 
and let 
\[d_1,\ldots,d_t\in \bH^1(K^\kn/K,T^*)\setminus\{0\}.\]
 Under Assumptions \ref{ass:patched:basic} and \ref{ass:patched:quotient}, if $s+t<p$, there exists an ultraprime $\ku\in \PP$ such that $\loc_\ku(c_i)$ generates $\bH^1_\f(K_\ku,T)$ for all $i=1,\ldots, s$ and $\loc_\ku(d_j)\neq 0$ for all $j=1,\ldots, t$.
\label{prop:patched:cheb}
\end{proposition}

\begin{proof}
Write 
\[c_i=(c_i)_k\in \varprojlim_{k\in \N} \bH^1(K^{\kn}/K,T_k).\]
Since $c_i\notin \m\bH^1(K^\kn/K,T)$, Assumption \ref{ass:patched:quotient} implies that $(c_i)_1$ does not vanish. Fix a sequence $\Bigl((c_i)_1^{(k)}\Bigr)_{k\in \N}$ representing $(c_i)_1$ in 
\[\bH^1(K^{\kn}/K,T_1)=\UU_k\Bigl(H^1(K^{n_k}/K,T_1)\Bigr).\]
where $(n_k)_{k\in \N}$ is a sequence representing $\kn$. Similarly, there exists $a\in \N$ such that 
\[d_j\in \bH^1(K^{\kn}/K,T^*_a)\ \forall j=1,\ldots, t.\]
Similarly, fix sequences $\Bigl(d_j^{(k)}\Bigr)$ representing $d_j$ in 
\[\bH^1(K^{\kn}/K,T^*_a)=\UU_k\Bigl(H^1(K^{\kn_k}/K,T_a^*)\Bigr).\]

By \cite[Lemma 3.9]{BurnsSakamotoSano2}, for every $k\in \N$, we can find a Kolyvagin prime such that $\ell_k\in \PP(T_k)$ such that 
\[\begin{aligned}
&\loc_{\ell_k}\Bigl((c_i)_1^{(k)}\Bigr)\neq 0\Leftrightarrow \Bigl((c_i)_1^{(k)}\Bigr)\neq 0\ \forall i=1,\ldots, s,\\
&\loc_{\ell_k}\Bigl(d_j^{(k)}\Bigr)\neq 0\Leftrightarrow d_j^{(k)}\neq 0\ \forall j=1,\ldots, t.
\end{aligned}\]

Let $\ku$ be the ultraprime represented by the sequence $(\ell_k)_{k\in \N}$. By construction, $\ku\in \PP$. Since $(c_i)_1^{(k)}\neq 0$ for $\UU$-many $k$, because it represents a non-zero element in the ultraproduct, then $\loc_{\ell_k}\Bigl((c_i)_1^{(k)}\Bigr)\neq 0$ for $\UU$-many $k$. Hence $\loc_\ku\bigl((c_i)_1\bigr)$ is a non-zero element in $\bH^1_\f(K_\ku,T_1)$, so $\loc_\ku(c_i)$ generates $\bH^1_\f(K_\ku,T)$ for all $i$. Similarly, we can prove that $\loc_\ku(d_j)$ is non-zero for all $j$.
\end{proof}

The remainder of the section explains what can be achieved by choosing the right elements in Proposition \ref{prop:patched:cheb}. First, we show that Assumption \ref{ass:patched:quotient} implies that all cartesian Selmer structures contain elements that are not $\m$-divisible.

\begin{proposition}
Suppose that $T$ satisfies Assumptions \ref{ass:patched:basic} and \ref{ass:patched:quotient}. If $\FF$ is a cartesian Selmer structure defined on $T$ and $\bH^1_{\FF}(K,T)\neq 0$, there exists an element 
\[c\in \bH^1_{\FF}(K,T)\setminus\m \bH^1(K^{\Sigma_{\FF}}/\Sigma,T).\]
\label{prop:cheb:non-div_elt}
\end{proposition}

\begin{proof}
By Nakayama's lemma, there exists a class
\[c\in \bH^1_{\FF}(K,T)\setminus \m\bH^1_{\FF}(K,T).\]
Then Proposition \ref{prop:patched:quot_selmer_inj} implies that
\[c\notin \m\bH^1(K^{\kn}/K,T),\]
where $\kn$ is the square-free product of all the primes in $\Sigma_{\FF_1}\cup\cdots\cup\Sigma_{\FF_s}$. 
\end{proof}

Combining the previous two results, we can obtain the following corollary.

\begin{corollary}
Let $\FF_1,\ldots,\FF_s$ be cartesian structures, and let $d_1,\ldots,d_t\in \bH^1(K,T^*)\setminus\{0\}$, where $s+t<p$. Under Assumptions \ref{ass:patched:basic} and \ref{ass:patched:quotient}, if $\bH^1_{\FF_i}(K,T)\neq 0$ for all $i$, there exists an ultraprime $\ku\in \PP$ such that the maps
\[\loc_\ku:\ \bH^1_{\FF_i}(K,T)\to \bH^1_\f(K_\ku,T)\] 
are surjective for all $i=1,\ldots, s$ and $\loc_\ku(d_j)\neq 0$ for all $j=1,\ldots, t$.
\label{cor:patched:cheb}
\end{corollary}

We can also use Proposition \ref{prop:patched:cheb} to make the dual Selmer group smaller.

\begin{proposition}
Assume that $\FF$ is a cartesian Selmer structure on $T$ such that $\bH^1_{\FF^*}(K,T^*)\neq 0$. Then there exists $d\in \bH^1_{\FF^*}(K,T^*)$ such that, for any ultraprime $\ku$, 
\[\loc_\ku(d)\neq 0\Rightarrow \dim_{R/\m} \bH^1_{(\FF^*)_\ku}(K,T^*)[\m]<\dim_{R/\m} \bH^1_{(\FF^*)}(K,T^*)[\m].\]
\label{prop:cheb:torsion_red}
\end{proposition}

\begin{proof}
Nakayama's lemma applied to the Pontryagin dual implies that 
\[\bH^1_{\FF^*}(K,T^*)\neq 0\Leftrightarrow \bH^1_{\FF^*}(K,T^*)[\m]\neq 0.\]
Then the result follows by choosing a non-zero $\m$-torsion element $d$.
\end{proof}

An inductive application of Proposition \ref{prop:cheb:torsion_red} implies the following theorem.

\begin{corollary}
Let $\FF$ be a Selmer structure. Under Assumptions \ref{ass:patched:basic} and \ref{ass:patched:quotient}, there is some $\kn\in \NN$ such that 
\[\bH^1_{\FF^*_\kn}(K,T^*)=0.\]
\label{cor:patched:dual_res_vanish}
\end{corollary}

When $\bH^1_{\FF}(K,T)$ is non-zero, we can construct a vertex $\kn\in \NN$ as above but also satisfying an additional property.
\begin{corollary}
Assume $p\geq 3$ and let $\FF$ be a Selmer structure. Under Assumptions \ref{ass:patched:basic} and \ref{ass:patched:quotient}, if $\bH^1_{\FF}(K,T)\neq 0$, there is some $\kn\in \NN$ such that $\bH^1_{(\FF^*)_\kn}(K,T^*)=0$ and that, for every $\ku\mid\kn$, there is an equality
\[\bH^1_{(\FF^*)(\ku)}(K,T^*)=\bH^1_{(\FF^*)_\ku}(K,T^*).\]
\label{cor:patched:dual_res_vanish_417}
\end{corollary}

\begin{proof}
Choose a basis $\{r_1,\ldots,r_s\}$ of $\bH^1_{\FF^*_\kn}(K,T^*)[\m]$ as a $k$-vector space. By Corollary \ref{cor:patched:cheb}, we can find ultraprimes $\ku_i\in \PP$ such that $\loc_{\ku_i}(r_i)\neq 0$ and the map
\[\loc_{\ku_i}:\ \bH^1_{\FF}(K,T)\to \bH^1_\f(K_{\ku_i},T)\]
is surjective. For $\kn:=\ku_1\ldots\ku_s$, we get $\bH^1_{\FF^*_\kn}(K,T^*)=0$ and, by Lemma \ref{lem:patched:tr_res}, we have that 
\[\bH^1_{(\FF^*)(\ku_i)}(K,T^*)=\bH^1_{(\FF^*)_{\ku_i}}(K,T^*).\qedhere\]
\end{proof}

When the core rank is positive, we can achieve an analogue of Corollary \ref{cor:patched:dual_res_vanish} with the transverse Selmer group.
\begin{corollary}
Assume that $p\geq 3$ and let $\FF$ be a cartesian Selmer structure of core rank $\chi(\FF)=1$. Under Assumptions \ref{ass:patched:basic} and \ref{ass:patched:quotient}, there exists some $\kn\in \NN$ such that $\bH^1_{(\FF^*)(\kn)}(K,T^*)=0$.
\label{cor:patched:dual_tr_vanish}
\end{corollary}

\begin{proof}
Similarly, let $\{r_1,\ldots,r_s\}$ be a basis of $\bH^1_{\FF^*_\kn}(K,T^*)[\m]$ as a $k$-vector space. Suppose inductively that we have constructed ultraprimes $\ku_1,\ldots,\ku_i$ such that $\loc_{\ku_j}(r_j)\neq 0$ for all $j\leq i$ and 
\[\bH^1_{(\FF^*)(\ku_1\ldots\ku_i)}(K,T^*)=\bH^1_{(\FF^*)_{\ku_1\cdots\ku_i}}(K,T^*).\]
Note that Proposition \ref{prop:patched:rank_modified} implies that $\chi(\FF(\ku_1\ldots\ku_i))=\chi(\FF)$ is positive, so, in particular, $\bH^1_{\FF(\ku_1\ldots\ku_i)}(K,T)\neq 0$. By Corollary \ref{cor:patched:cheb}, we can find an ultraprime $\ku_{i+1}\in \PP$ such that $\loc_{\ku_{i+1}}(r_{i+1})\neq 0$ and the map
\[\loc_{\ku_{i+1}}:\ \bH^1_{\FF(\ku_1\ldots\ku_i)}(K,T)\to \bH^1_\f(K_{\ku_{i+1}},T)\]
is surjective. By Lemma \ref{lem:patched:tr_res}, we have that
\[\bH^1_{(\FF^*)(\ku_1\ldots\ku_{i+1})}(K,T^*)=\bH^1_{(\FF^*)_{\ku_{i+1}}(\ku_1\ldots\ku_i)}(K,T^*)=\bH^1_{(\FF^*)_{\ku_1\cdots\ku_{i+1}}}(K,T^*).\]
For $\kn:=\ku_1\ldots\ku_s$, we get $\bH^1_{(\FF^*)(\kn)}(K,T^*)=0$.
\end{proof}

In addition to obtaining a trivial Selmer group by modifying the local condition at certain primes, we will also use a stronger result about reducing the rank of the Selmer group modifying the local condition at only one prime.
\begin{proposition}
Let $\FF$ be a cartesian Selmer structure such that 
\[\rank_R\ \bH^1_{\FF^*}(K,T^*)^\vee>0.\]
Then there exists $d\in \bH^1_{\FF^*}(K,T^*)$ such that, for every ultraprime $\ku$, then 
\[\loc_\ku(d)\neq 0\Rightarrow \rank_R\ \bH^1_{(\FF^*)_\ku}(K,T^*)^\vee=\rank_R\ \bH^1_{\FF^*}(K,T^*)^\vee-1.\]
\label{prop:patched:cheb_rank_red}
\end{proposition}

\begin{proof}
By Lemma \ref{lem:quot_tors_free}, there is an injection
\[\mu:\ \bH^1_{\FF^*}(K,T^*)^\vee_{\tors}\otimes R/\m\hookrightarrow  \bH^1_{\FF^*}(K,T^*)^\vee\otimes R/\m.\]

Choose a $k$-basis $\{f_1,\ldots,f_s\}$ of $\mu\Bigl(\bH^1_{\FF^*}(K,T^*)^\vee_{\tors}/\m\Bigr)$ and a set $\{g_1,\ldots,g_t\}$ extending it to a basis of $\bH^1_{\FF^*}(K,T^*)^\vee/\m$.

The localisation $(\bH^1_{\FF^*}(K,T^*)^\vee_{/\tors})_{(0)}$ is a finitely-generated $R_{(0)}$ vector space of dimension $r$ coinciding with the rank. Then any set of generators of $\bH^1_{\FF^*}(K,T^*)^\vee_{/\tors}$ contains at least $r$ elements, so Nakayama's lemma implies that $t\geq r>0$.

Let $d\in \bH^1_{\FF^*}(K,T^*)[\m]$ be the element of the dual basis associated with $g_1$, i.e., the unique element satisfying that  $g_1(d)\equiv 1\mod \m$ and $g_i(d)=0$ for $i\geq 2$ and $f_i(d)=0$ for all $i$.

If $\ku\in \PP$ is a Kolyvagin ultraprime satisfying that $\loc_\ku(d)\neq 0$. Consider then the following diagram:
\[\xymatrix{
    \bH^1_\f(K_\ku,T^*)^\vee \ar[r]^{(\loc_\ku)^\vee} \ar[d] & \bH^1_{\FF^*}(K,T^*)^\vee \ar[d] \\
    \bH^1_\f(K_\ku,T^*)^\vee\otimes R/\m \ar[r]^{(\loc_\ku)^\vee} & \bH^1_{\FF^*}(K,T^*)^\vee\otimes R/\m.
}\]
The image of the bottom horizontal map contains an element $h$ satisfying that $h(d)\neq 0$. Hence the ${g_1}$-coordinate of ${d}$ in the basis $\{f_1,\ldots,f_s,g_1,\ldots,g_t\}$ is non-zero. In particular, it implies that any lift of $h$ in $\bH^1_{\FF^*}(K,T^*)^\vee$ is not a torsion element. 

Nevertheless, the image of the top horizontal arrow contains a lift of $h$, which will be a non-torsion element. In particular, the image has rank one. Since the rank is an additive function,  we can conclude that 
\[\rank_R\ \bH^1_{(\FF^*)_\ku}(K,T^*)=\rank_R\ \bH^1_{\FF^*}(K,T^*)-1.\qedhere\]

\end{proof}

\begin{proposition}
Suppose that $p\geq 3$ and that $T$ satisfies Assumptions \ref{ass:patched:basic} and \ref{ass:patched:quotient}. Let $\FF$ be a cartesian Selmer structure such that $\bH^1_{\FF}(K,T)\neq 0$. Assume that 
\[\rank_R\ \bH^1_{\FF^*}(K,T^*)^\vee>0.\]
Then there is an ultraprime $\ku\in \PP$ such that 
\[\rank_R\ \bH^1_{(\FF^*)(\ku)}(K,T^*)^\vee=\rank_R\ \bH^1_{\FF^*}(K,T^*)^\vee-1.\]
\label{cor:patched:cheb_rank_red}
\end{proposition}

\begin{proof}
If is a direct application of Proposition \ref{prop:patched:cheb} choosing the element in $\bH^1_{\FF}(K,T)$ given by Proposition \ref{prop:cheb:non-div_elt} and the element in $\bH^1_{\FF^*}(K,T^*)$ given by Proposition \ref{prop:patched:cheb_rank_red}.
\end{proof}

\section{Iwasawa Selmer groups}
\label{sec:iwasawa}

From now on, we reduced the generality on the coefficient ring $R$ and work with Selmer groups with coefficients in the Iwasawa algebra.

\begin{namedass}{(Iw)}
Let $\Lambda=\Z_p[[X]]$ be the classical Iwasawa algebra and let $\bT$ be a $\Lambda[[G_K]]$-module that is free and finitely generated as a $\Lambda$-module. In addition, we assume the Galois action is only ramified at finitely many primes.
\end{namedass}

\begin{notation}
Following Remark \ref{rem:R_lim}, we denote
\[\Lambda_k=\Lambda/(p^k,X^k).\]
Note that the Iwasawa algebra satisfies Assumption \ref{ass:R_lim}, so the results in the previous section also hold in this case.
\label{not:iw_limit}
\end{notation}

Since the maximal ideal is not principal, we cannot use Remark \ref{rem:patched:ass_quotient} to prove Assumption \ref{ass:patched:quotient}, which was required for the results in \textsection\ref{sec:cheb}. However, we can show that Assumption \ref{ass:patched:quotient} is equivalent to the following which, although it depends on a global cohomology group, it can be checked from the local cohomology using Lemma \ref{lem:H2} below.

\begin{namedass}{(H2)}
There is a square-free product of ultraprimes $\kn$ divisible by all the primes in $\Sigma_\FF$ such that the patched cohomology group $\bH^2(K^{\kn}/K,\bT)$ contains no finite $\Lambda$-submodules.
\label{ass:H2}
\end{namedass}

\begin{remark}
Although Assumption \ref{ass:H2} might seem hard to verify, it will be reinterpreted, in \textsection\ref{sec:H2} below, in terms of the local zero cohomology groups. 
\end{remark}

\begin{proposition}
Under Assumption \ref{ass:patched:basic}, Assumptions \ref{ass:patched:quotient} and \ref{ass:H2} are equivalent.
\label{prop:iwasawa:ass_equiv}
\end{proposition}

\begin{proof}
Fix a square-free product of ultraprimes $\kn$ divisible by all the primes in $\Sigma_{\FF}$ and consider the exact sequence
\[\xymatrix{0\ar[r] & \bT\ar[r] & \bT\ar[r] & \bT/X\bT\ar[r] & 0.}\]
By Remark \ref{rem:patched:H0vanish} and Proposition \ref{prop:patched:global_duality}, it induces another exact sequence
\begin{equation}
\bH^1(K^{\kn}/K,\bT)\otimes \Lambda/(X)\hookrightarrow \bH^1(K^{\kn}/K,\bT/X\bT)\twoheadrightarrow \bH^2(K^{\kn}/K,\bT)[X].
\label{eq:ass_eq_ses1}
\end{equation}
In addition, consider the exact sequence
\[\xymatrix{0\ar[r] & \bT/X\bT\ar[r]^p & \bT/X\bT\ar[r] & \bT/\m\bT\ar[r] & 0.}\]
Again, Proposition \ref{prop:patched:global_duality} and Remark \ref{rem:patched:H0vanish} show that 
\[\bH^1\Bigl(K^{\kn}/K,\bT/X\bT\Bigr)[p]=0\]
and that the following map is injective:
\begin{equation}
\bH^1\Bigl(K^{\kn}/K,\bT/X\bT\Bigr)\otimes \Lambda/(p)\hookrightarrow \bH^1\Bigl(K^{\kn}/K,\bT/\m\bT\Bigr).
\label{eq:ass_eq_inj}
\end{equation}
The snake lemma applied to the short exact sequence in \eqref{eq:ass_eq_ses1} and the multiplication by $p$ induces another exact sequence
\begin{equation}
\bH^2(K^{\kn}/K,\bT)[\m]\hookrightarrow    \bH^1(K^{\kn}/K,\bT)\otimes \Lambda/\m\twoheadrightarrow \bH^1(K^{\kn}/K,\bT/(X))\otimes \Lambda/(p).
\label{eq:iwasawa:coh_inj_max_partial}
\end{equation}
Combining \eqref{eq:ass_eq_inj} and \eqref{eq:iwasawa:coh_inj_max_partial}, we obtain an exact sequence
\[\xymatrix{0\ar[r] & \bH^2(K^{\kn}/K,\bT)[\m]\ar[r] & \bH^1(K^{\kn}/K,\bT)\otimes \Lambda/\m \ar[r] & \bH^1(K^{\kn}/K,\bT/\m\bT),}\]
that leads to the equivalence between Assumptions \ref{ass:patched:quotient} and \ref{ass:H2} for the particular choice of $\kn$.
\end{proof}

Under Assumption \ref{ass:patched:basic}, all the profinite Selmer groups are free modules.
\begin{proposition}
Let $\FF$ be a Selmer structure defined on a Galois representation $\bT$ satisfying Assumptions \ref{ass:patched:basic}. Then the Selmer group $\bH^1_{\FF}(K,T)$ is a free, finitely generated $\Lambda$-module.
\label{prop:iwasawa:selmer_free}
\end{proposition}

\begin{proof}
By Proposition \ref{prop:patched:selmer_torsion_free}, $\bH^1_{\FF}(K,\bT)$ is a finitely generated, torsion free Iwasawa module, so the structure theorem (see \cite[Theorem 5.3.8]{NSW}) implies that it is pseudo-isomorphic to a free module. In addition, Proposition \ref{prop:patched:selmer_torsion_free} implies that the Selmer group contains no finite submodules, so there is a free module $F$ and a finite module $C$ fitting in the exact sequence
\[\xymatrix{0\ar[r] & \bH^1_{\FF}(K,\bT)\ar[r] & F\ar[r] &C \ar[r] &0.}\]
Assume for the sake of contradiction that $C$ is a non-trivial module. The snake lemma induces another exact sequence
\[\xymatrix{0\ar[r] & C[X]\ar[r] & \bH^1_{\FF}(K,\bT)\otimes \Lambda/(X)\ar[r] & F/XF \ar[r] & C/XC\ar[r] & 0.}\]
When $C$ is non-trivial, then $C[X]\neq 0$, so $\bH^1_{\FF}(K,\bT)\otimes \Lambda/(X)$ contains a finite submodule.

Consider the commutative diagram
\[\xymatrix{
\bH^1_{\FF}(K,T)\otimes \Lambda/(X) \ar[r] \ar@{>->}[d]  & \bH^1_{\FF}\Bigl(K,\bT/(X)\Bigr)\ar[d]\\
\bH^1(K^{\kn}/K,T)\otimes \Lambda/(X) \ar@{>->}[r]   & \bH^1\Bigl(K^{\kn}/K,\bT/(X)\Bigr),
}\]
in which the left vertical map is injective by Proposition \ref{prop:patched:quot_selmer_inj} and the bottom horizontal map is injective by the argument in Proposition \ref{prop:iwasawa:ass_equiv}. It implies that 
\[\bH^1\Bigl(K^\kn/K,\bT/X\bT\Bigr)[p]\neq 0,\]
 contradicting Proposition \ref{prop:patched:selmer_torsion_free} for the Galois representation $\bT/(X)$.
\end{proof}

We can use the Euler characteristic to compute the core rank of Iwasawa Selmer modules.

\begin{proposition}
Let $\FF$ be a Selmer structure defined on a Galois representation $T$ satisfying Assumption \ref{ass:patched:basic}. Denote 
\[d_-(T):=\sum_{\substack{v\in \Sigma_{\infty}\\v\ \textrm{real}}} \rank_{\Lambda} \bT^{\sigma_v=-1} +\sum_{\substack{v\in \Sigma_{\infty}\\v\ \textrm{complex}}} \rank_{\Lambda} \bT.\]
where $\Sigma_\infty$ denotes the archimedean places of $K$ and $\sigma_v$ denotes any complex conjugation. Then
\begin{equation}
\chi(\FF)=d_-(T)+\sum_{\ku\in \Sigma_{\FF}} \biggl(\rank_\Lambda\bH^1_{/\FF}(K_\ku,T)-\rank_\Lambda \bH^0(K_\ku,T^*)^\vee\biggr).
\label{eq:core_rank}
\end{equation}
\label{prop:core_rank}
\end{proposition}

\begin{proof}
Combining Propositions \ref{prop:patched:poitou-tate} and \ref{prop:patched:global_duality}, there is an exact sequence
\begin{center}
\begin{tikzpicture}[descr/.style={fill=white,inner sep=1.5pt}]
        \matrix (m) [
            matrix of math nodes,
            row sep=4em,
            column sep=1em,
            text height=1.5ex, text depth=0.25ex
        ]
        {  0& \bH^1_{\FF}(K,T) & \bH^1(K^{\Sigma_\FF}/K,T) & \displaystyle\bigoplus_{\ku\in \Sigma_\FF} \bH^1_{/\FF}(K_\ku,T) \\
            & \bH^1_{\FF^*}(K,T^*)^\vee & \bH^2(K^{\Sigma_\FF}/K,T)& \displaystyle\bigoplus_{\ku\in \Sigma_\FF} \bH^2(K_\ku,T) &0.\\
        };

        \path[overlay,->, font=\scriptsize,>=latex]
        (m-1-1) edge (m-1-2)
        (m-1-2) edge (m-1-3)
        (m-1-3) edge (m-1-4)
        (m-1-4) edge [out=355,in=175] (m-2-2) 
        (m-2-2) edge (m-2-3)
        (m-2-3) edge (m-2-4)
        (m-2-4) edge (m-2-5)
        ;    
\end{tikzpicture}
\end{center}
By Remark \ref{rem:patched:H0vanish}, $H^0(K^{\Sigma_\FF}/K,T)=0$. Then Corollary \ref{cor:euler_char} and the standard limit argument imply that
\[d_-(T)=\rank_\Lambda\ \bH^1(K^{\Sigma_\FF}/K,T)-\rank_\Lambda\ \bH^2(K^{\Sigma_\FF}/K,T).\]
By Proposition \ref{prop:patched:local_duality}, there is an identification $\bH^2(K_\ku,T)\cong\bH^0(K_\ku,T^*)^\vee$ for every ultraprime $\ku$. Then the formula follows from the additivity of the rank.
\end{proof}

\subsection{Core vertices}

There are some special vertices in $\NN$ that control the Kolyvagin and Stark systems defined below. In particular, Stark systems are controlled by weak core vertices and Kolyvagin systems by core vertices.

\begin{definition}
A \emph{weak core vertex} is a square-free product $\kc\in \NN$ of Kolyvagin ultraprimes such that $\bH^1_{(\FF^*)_\kc}(K,\bT^*)=0$.
\label{def:core_vertex_weak}
\end{definition}

\begin{proposition}
If $\kc\in \NN$ is a weak core vertex, then $\bH^1_{\FF^\kc}(K,\bT)$ is a free module of rank $\chi(\FF)+\nu(\kc)$.
\label{prop:weak_core_pseudo_free}
\end{proposition}

\begin{proof}
Note that $(\FF^\kc)^*=(\FF^*)_\kc$. By Proposition \ref{prop:patched:rank_modified}, $\chi(\FF^c)=\chi(\FF)+\nu(c)$, so the rank of $\bH^1_{\FF^{\kc}}(K,\bT)$ is $\chi(\FF)+\nu(c)$. By Proposition \ref{prop:iwasawa:selmer_free}, 
\[\bH^1_{\FF^{\kc}}(K,\bT)\cong \Lambda^{\chi(\FF)+\nu(\kc)}.\qedhere\]
\end{proof}

We can use Corollary \ref{cor:patched:dual_res_vanish} to prove the existence of a weak core vertex.
\begin{proposition}
Assume that $\bT$ satisfies Assumption \ref{ass:patched:basic}. Let $\kn\in \NN$ be a square-free product of Kolyvagin ultraprimes. Then there exists a weak core vertex $\kc\in \NN$ such that $\kn\mid \kc$.
\label{prop:weak_core_vertex_existence}
\end{proposition}

\begin{proof}
It follows from a direct application of Corollary \ref{cor:patched:dual_res_vanish} to the Selmer structure $\FF_\kn$.
\end{proof}

\begin{remark}
Note that the following converse result also holds: when $\kc\in \NN$ is a weak core vertex, every $\kn\in \NN$ divisible by $\kc$ is also a weak core vertex.
\label{rem:weak_core_div}
\end{remark}

As we will show below in \textsection\ref{sec:iwasawa:stark}, weak core vertices are used to control Stark systems. In order to have an analogue for Kolyvagin systems, we need a stronger version of core vertices.

\begin{definition}
A \emph{core vertex} is a square-free product $\kc\in \NN$ of Kolyvagin ultraprimes such that $\bH^1_{(\FF^*)(\kc)}(K,\bT^*)=0$.
\label{def:core_vertex}
\end{definition}

\begin{proposition}
Suppose that $\bT$ satisfies Assumption \ref{ass:patched:basic}. If $\kc\in \NN$ is a core vertex,  $\bH^1_{\FF(\kc)}(K,\bT)$ is a free module of rank $\chi(\FF)$.
\label{prop:core_pseudo_free}
\end{proposition}

\begin{proof}
By Proposition \ref{prop:patched:rank_modified}, then $\chi(\FF(\kc))=\chi(\FF)$, so the rank of $\bH^1_{\FF(\kc)}(K,\bT)$ is $\chi(\FF)$. By Proposition \ref{prop:iwasawa:selmer_free}, we have that 
\[\bH^1_{\FF^\kc}(K,\bT)\cong \Lambda^{\chi(\FF)}.\qedhere\]
\end{proof}

We can also show the existence of core vertices, using Corollary \ref{cor:patched:dual_tr_vanish}.

\begin{proposition}
Consider a square-free product of Kolyvagin ultraprimes $\kn\in \NN$. If $\bT$ satisfies Assumptions \ref{ass:patched:basic} and \ref{ass:H2}, then there is a core vertex $\kc\in \NN$ such that $\kn\mid\kc$.
\label{prop:core_vertex_existence}
\end{proposition}

\begin{proof}
It is an immediate consequence of Corollary \ref{cor:patched:dual_tr_vanish}, which is applicable to this case due to Proposition \ref{prop:iwasawa:ass_equiv}.
\end{proof}

\begin{remark}
Proposition \ref{prop:core_vertex_existence} is the only place where Assumption \ref{ass:H2} is needed. However, the existence of core vertices is a central step in the control of Kolyvagin systems, as will be shown in Theorem \ref{th:kol_core_projection} below.
\end{remark}

\begin{remark}
There is no analogue of Remark \ref{rem:weak_core_div} for core vertices, since there might exist $\kn\in \NN$ and $\ku\in \PP$ such that $\bH^1_{(\FF^*)(\kn)}(K,\bT^*)$ vanishes but $\bH^1_{(\FF^*)(\kn\ku)}(K,\bT^*)$ does not.
\label{rem:core_div}
\end{remark}

\subsection{Assumption on the second cohomology}
\label{sec:H2}

In this section, we reinterpret Assumption \ref{ass:H2} on the lack finite $\Lambda$-submodules of $\bH^2(K^\kn/K,\bT)$ in terms of the local cohomology of bad primes.

\begin{namedass}{(Tam)}
For every $\ku\in \Sigma_{\FF}$, $\bH^0(K_\ku,\bT^*)^\vee$ contains no finite $\Lambda$-submodule.
\label{ass:iwasawa:local_no_fin_sub}
\end{namedass}

\begin{remark}
By Proposition \ref{prop:patched:local_duality}, Assumption \ref{ass:iwasawa:local_no_fin_sub} is equivalent to the fact that $\bH^2(K_\ku,\bT)$ does not containing any finite submodules for any $\ku\in \Sigma_{\FF}$.
\label{rem:iwasawa:local_duality_finsub}
\end{remark}

\begin{lemma}
Under Assumption \ref{ass:patched:basic}, Assumption \ref{ass:iwasawa:local_no_fin_sub} implies Assumption \ref{ass:H2}.
\label{lem:H2}
\end{lemma}

\begin{proof}
Let $\kn$ be a weak core vertex for $\FF$, which exists by Proposition \ref{prop:weak_core_vertex_existence}, and let $\Sigma$ be the set consisting of the primes in $\Sigma_\FF$ and the prime divisors of $\kn$. Since the kernel of the map
\begin{equation}
\bH^1(K^\Sigma/K,\bT^*)\to \bigoplus_{\ku\in \Sigma}\bH^1(K_\ku,\bT)
\label{eq:H2_inj}
\end{equation}
is contained in $\bH^1_{(\FF^*)_\kn}(K,\bT)$, this kernel is trivial since $\kn$ is a weak core vertex.

By Proposition \ref{prop:patched:poitou-tate}, there is an exact sequence
\begin{center}
\begin{tikzpicture}[descr/.style={fill=white,inner sep=1.5pt}]
        \matrix (m) [
            matrix of math nodes,
            row sep=3.5em,
            column sep=1em,
            text height=1.5ex, text depth=0.25ex
        ]
        {  \displaystyle\bigoplus_{\ku\in \Sigma}\bH^1(K_\ku,\bT) & \bH^1(K^{\Sigma}/K,\bT^*)^\vee &
            \bH^2(K^\Sigma/K,\bT) & \displaystyle\bigoplus_{\ku\in \Sigma}\bH^0(K_\ku,\bT^*)^\vee.  \\
        };

        \path[overlay,->, font=\scriptsize,>=latex]
        (m-1-1) edge (m-1-2)
        (m-1-2) edge (m-1-3) 
        (m-1-3) edge (m-1-4)
        ;    
\end{tikzpicture}
\end{center}
Since $\kn$ is a weak core vertex, the first map is surjective since it is the dual of the map in \eqref{eq:H2_inj}. Hence the proof of the lemma is reduced to showing that $\bH^0(K_\ku,\bT^*)^\vee$ contains no finite non-trivial submodules for any $\ku\in \Sigma$. When $\ku\in \Sigma_\FF$, this is an assumption, and when $\ku\mid \kn$, then $\ku$ is a Kolyvagin prime, so 
\[\bH^0(K_\ku,\bT^*)^\vee\cong \Lambda\]
also contains no finite non-zero submodules. Then Assumption \ref{ass:H2} holds by considering the square free product of primes in $\Sigma$.
\end{proof}

\begin{remark}
When $\ell$ is a constant ultraprime, one condition implying that $\bH^0(K_\ell,\bT^*)^\vee$ contains no-finite submodules is the $p$-divisibility of the dual group $\bH^0(K_\ell,\bT^*)$.

Suppose that $\bT$ comes from a $p$-adic representation $T$ by Shapiro's lemma, meaning that 
\[\bT=T\otimes_{\Z_p}\Z_p[[\Gal(K_\infty/K)]],\]
where $K_\infty/K$ is a $\Z_p$-extension at which no-finite prime splits completely. Then Shapiro's lemma gives an isomorphism
\[\bH^0(K_\ell,\bT^*)=\bigoplus_{v\mid \ell}\bH^0\Bigl((K_\infty)_v,T^*\Bigr),\]
where the sum is taken over all the primes of $K_\infty$ above $\ell$. 

The divisibility of $\bH^0(K_\ell,\bT^*)$ can be guaranteed by assuming the $p$-primality of the Tamagawa numbers associated with a Galois representation. The $p$-part of the Tamagawa number at $\ell$ is defined as the order of the torsion subgroup of $\bH^0\Bigl(K_\ell^{\ur},T\Bigr)$. Therefore, if the Tamagawa number is prime to $p$, then $\bH^0\Bigl(K_\ell^{\ur},T\Bigr)$ is a free $\Z_p$-module, so $\bH^0\Bigl(K_\ell^{\ur},T^*\Bigr)$ is divisible.

Since $\ell$ does not split completely at $K_\infty/K$, then $(K_\infty)_v/K_\ell$ is the unramified $\Z_p$-extension of $K_\ell$. Hence the profinite degree $K_\ell^{\ur}/(K_\infty)_v$ is prime to $p$, so its action on 
\[\Aut\biggl(\bH^0\Bigl(K_\ell^{\ur},T^*\Bigr)^\vee\biggr)\cong \GL_d(\Z_p),\]
where $d\in \Z_{\geq 0}$ is at most the rank of $\bT$, has finite image. Hence the eigenvalues of all the matrices in the image are roots of unity, implying that $\bH^0\Bigl((K_\infty)_v,T^*\Bigr)$ is also a divisible group.

\label{rem:iwasawa:local_nofinsub_divisible}
\end{remark}

\section{Algebraic preliminaries}
\label{sec:algebra}

In this section, we outline the algebraic preliminaries required for studying the structure of Iwasawa modules using Kolyvagin systems. The contents in this section include an introduction to the theory of Fitting ideals and the construction of the rank reduction map of exterior powers.

\subsection{Fitting ideals}

In this section, we define the Fitting ideals of modules and outline some results about their description of the structure of Iwasawa modules. For a more complete treatment, we refer the reader to \cite[\textsection 2.2]{BurnsBullach} and \cite[\textsection 2]{RoncheLongoVigni}.

\begin{definition}
Let $R$ be a ring and let $M$ be a finitely presented $R$-module. Choose a resolution
\[\xymatrix{R^n\ar[r]^{A} & R^m\ar[r] & M \ar[r] & 0,}\]
where the map $R^n\to R^m$ is represented by the matrix $A$. For every $i\geq 0$, we define the $i^{\textrm{th}}$ Fitting ideal $\Fitt^i_R(M)$ as the ideal of $R$ generated by the minors of size $(m-i)$ of $A$. By convention, $\Fitt^i_R(M)$ is equal to $R$ when $i\geq m$.
\label{def:fitting}
\end{definition}

\begin{remark}
The $0^{\textrm{th}}$ Fitting ideal coincides with the image of the following map, induced by the matrix $A$:
\[\Fitt_R^0(M)=\textrm{Im}\left(\bigwedge^{m} R^n\to \bigwedge^{m} R^m\cong R\right).\]
\label{rem:exterior_power_fitting}
\end{remark}

It can be shown that Fitting ideals are well defined.
\begin{proposition}(\cite[Corollary 20.4]{Eisenbud})
The Fitting ideals $\Fitt^i_R(M)$ are independent of the chosen resolution.
\end{proposition}

The following results can be useful.
\begin{proposition}
Let $R\to S$ be a ring-homomorphism, let $M$ be a finitely presented $R$-module and let $i\in \Z_{\geq 0}$. Then $\Fitt^i_S\Bigl(M\otimes_R S\Bigr)$ is the ideal generated by the image of $\Fitt^i_R(M)$ in $S$.
\label{prop:fitting_tensor}
\end{proposition}

\begin{proof}
Consider a resolution of $M$ given by 
\[\xymatrix{R^n\ar[r]^{A} & R^m\ar[r] & M \ar[r] & 0.}\]
Since the tensor product is a right exact functor, we can obtain a resolution of $M\otimes S$:
\[\xymatrix{S^n\ar[r]^{A} & S^m\ar[r] & M\otimes_R S \ar[r] & 0.}\]
Then the proposition follows from the definition of Fitting ideals.
\end{proof}

\begin{proposition}
Assume that $R$ is a domain and $M$ is a torsion $R$-module of projective dimension at most one. Then $\Fitt^0_R(M)$ is a principal ideal.
    \label{prop:fitting_principal}
\end{proposition}

\begin{proof}
Since the projective dimension of $M$ is at most one, it admits a resolution
\[\xymatrix{0\ar[r] & R^n\ar[r]^{A} & R^m\ar[r] & M \ar[r] & 0.}\]
Since $R$ is a domain, the rank is an additive function, so $m=n$. Therefore, $\Fitt^0(M)$ is generated by $\det(A)$.
\end{proof}
 
When $R$ is the Iwasawa algebra $\Z_p[[X]]$, the sequence of Fitting ideals can recover the pseudo-isomorphism class of the module. Moreover, they see a finer structure, since they can be also used to study the pseudo-null submodules.

In order to determine the pseudo-isomorphism structure of a module, we do not need the complete knowledge of the Fitting ideals, but only up to finite index.
\begin{notation}
When $M$ is an Iwasawa module and $I$ and $J$ are two submodules, we say that $I$ and $J$ are pseudo-equal if, for every height one prime ideal $\BB$, then $I_\BB=J_\BB$ in $M_\BB$. We then denote $I\approx J$. Note that this definition also applies to two ideals $I$ and $J$ of the Iwasawa algebra. If, in addition, $I\subset J$, we will write $I\subset_{f.i.} J$.
\label{not:pseudo_equal}
\end{notation}

\begin{proposition}(\cite[Theorem 2.8]{RoncheLongoVigni})
Let $\Lambda=\Z_p[[X]]$ be the Iwasawa algebra, and let $M$ be a finitely generated $\Lambda$-module. Then the Fitting ideals determine $\Lambda$ up to pseudo-isomorphism.
\label{prop:fitting_iwasawa}
\end{proposition}

The pseudo-equality class of the $0^{\textrm{th}}$-Fitting ideal is additive in short exact sequences.
\begin{corollary}(\cite[Theorem 2.8]{RoncheLongoVigni})
Let
\[\xymatrix{0\ar[r] & A\ar[r] & B\ar[r] & C\ar[r] & 0}\]
be a short exact sequence of Iwasawa modules. Then there is a pseudo-equality
\[\Fitt^0_\Lambda(B)\approx\Fitt^0_\Lambda(A)\Fitt^0_\Lambda(C).\]
\label{cor:fitting_add_iwasawa}
\end{corollary}

Fitting ideals can determine more than the structure of the module up to pseudo-isomorphism. In fact, they can determine whether a torsion Iwasawa module contains finite submodules.
\begin{proposition}
Let $M$ be a torsion Iwasawa module. Then $\Fitt^0_\Lambda(M)$ is a principal ideal if and only if $M$ contains no finite modules.
\label{prop:fitting_ppal_finsub}
\end{proposition}

\begin{proof}
By \cite[Proposition 5.5.3.(iv)]{NSW}, $M$ contains no finite submodules if and only if its projective dimension is at most one. By Proposition \ref{prop:fitting_principal} and \cite[Lemma 15.8.10]{stacks-project}, that equivalent to $\Fitt^0_\Lambda(M)$ being a principal ideal.
\end{proof}

\subsection{Rank reduction of exterior powers}
\label{sec:algebra:rank_reduction}

The goal of this section is to introduce a rank reduction map between exterior powers of Iwasawa modules. It will be based on the construction of \cite[\textsection 2]{BurnsBullach}.
\begin{definition}
The \emph{exterior bidual} of rank $r$ of a module $M$ is defined as
\[\bigcap^{r}_\Lambda M=\Hom\left(\bigwedge^r M^*,R\right),\]
where $\displaystyle\bigwedge^r M^*$ denotes the exterior power of $M$.
\end{definition}

\begin{proposition}(\cite[Lemma 2.17.]{BurnsBullach})
Let $s$ be a natural number and consider the exact sequence of $\Lambda$-modules
\[\xymatrix{0\ar[r] & M\ar[r]^\mu & N\ar[r]^\varepsilon & \Lambda^s \ar[r]&C\ar[r] & 0.}\]
For every $r\geq s$, there is a canonical map
\[\phi_{N,M}:\ \bigcap^{r} N \to \bigcap^{r-s} M.\]
\label{prop:bidual_map}
\end{proposition}

\begin{proof}
It follows from \cite[Lemma 2.17.]{BurnsBullach}, since the Iwasawa algebra is a Gorenstein ring by \cite[Lemma 2.3]{BurnsBullach}.
\end{proof}

Although this construction is more natural using biduals, it can be expressed as a rank reduction of exterior powers when $M$ and $N$ are free modules.
\begin{corollary}
In the notation of Proposition \ref{prop:bidual_map}, assume in addition that $M$ and $N$ are free modules. Then there is a canonical rank reduction map
\[\phi_{N,M}:\ \bigwedge^{r} N \to \bigwedge^{r-s} M\]
\label{cor:exterior_rankred}
\end{corollary}

\begin{proof}
It follows from the canonical isomorphism 
\[\bigcap^r_\Lambda F\cong \bigwedge^{r}_\Lambda F\]
for any free Iwasawa module.
\end{proof}

\begin{remark}
The map in Corollary \ref{cor:exterior_rankred} can be constructed directly. See \cite[\textsection 2.1]{BurnsSano} for details.
\end{remark}

The rank reduction map can be used to compute Fitting ideals.

\begin{definition}
Let $M$ be an Iwasawa module module and let $a\in M$. Consider the canonical map into the bidual module
\[\Phi:\ M\to M^{++}: a\in M\mapsto \Bigl[\varphi\in \Hom(M,R)\mapsto \varphi(a)\Bigr].\]
The \emph{index} of $a$ is defined as
\[\ind(a,M):=\Im(\Phi(a)).\]
\label{def:index}
\end{definition}

\begin{proposition}
Assume, in the notation of Proposition \ref{prop:bidual_map}, that $N$ is also a torsion-free module of rank $r$ and let $\psi$ be a generator in $\bigcap^{r} N$. Then 
\[\ind(\phi_{N,M})\approx\Fitt^0_\Lambda(C).\]
\label{prop:bidual_map_fitting}
\end{proposition}

\begin{remark}
The structure theorem implies that $N$ is pseudo-isomorphic to $\Lambda^r$. Dualizing, since $\Ext^1(C,\Lambda)=0$ for all finite Iwasawa modules $C$ (see \cite[Corollary 5.5.9]{NSW}), we get that $N^+\cong \Lambda^r$, so 
\[\bigcap^r N\cong \bigwedge^r \Lambda^r=\Lambda.\]
Hence there exists a generator $\psi$ of $\bigcap^{r} N$.
\label{rem:dual_fin_index}
\end{remark}

\begin{proof}[Proof of Proposition \ref{prop:bidual_map_fitting}]
Let $\BB$ be a height one prime ideal of $\Lambda$. Note that $\ind\Bigl(\Phi_{N,M} (\psi)\Bigr)_\BB$ coincides with the image of the composition
\[\xymatrix{\bigwedge^{r-s} N_\BB^+\otimes \det(F^+)_\BB\ar[r] & \bigwedge^{r-s} M_\BB^+ \otimes \det(F^+)_\BB \ar[r] & \bigwedge^{r} N_\BB^+\ar[r]^{\ \ \ \ \psi} & \Lambda_\BB.}\]

After localising with a height $1$ prime ideal $\BB$ of $\Lambda$ and fixing an isomorphism $\det(F)_\BB\cong \Lambda_\BB$, Proposition \ref{prop:fitting_tensor} and Remark \ref{rem:exterior_power_fitting} imply that 
\[\Fitt_\Lambda^0(C)_\BB=\Im\left(\bigwedge^s N_\BB\to \det(F)_\BB\right).\]
Note that $N_\BB$ is a free $\Lambda_\BB$-module of rank $r$ and consider the image of the dual map
\[I:=\Im\left(\det(F^+)_\BB\to \bigwedge^s N^+_\BB\right).\]
Since $\Lambda_\BB$ is a discrete valuation ring, every generator $\alpha$ of $I$ has index $\Fitt^0_R(C)_\BB$, so there is an element $\beta\in \bigwedge^s N^+_\BB$ and a generator $\gamma$ of $\Fitt^0_R(C)_\BB$ such that $\alpha =\gamma\beta$ and $\beta$ projects to a non-zero element 
\[\overline \beta\in \bigwedge^s (N^+_\BB\otimes \Lambda_\BB/\BB).\]
That implies that there exists a linearly independent set in the $\Lambda_\BB/\BB\Lambda_\BB$ vector space $N^+_\BB/\BB N^+_\BB$, $S=\{e_1,\ldots,e_s\}$, such that 
\[\overline{\beta}=e_1\wedge\cdots \wedge e_s.\]
Consider a set $\{e_{s+1},\ldots,e_r\}$ extending $S$ to a basis of $N^+_\BB/\BB N^+_\BB$. Then,
\[\alpha\wedge e_{s+1}\wedge\cdots \wedge e_r\in \bigwedge^r N^+_\BB\]
generates the submodule $\left(\Fitt^0_R(C)_\BB\right)\cdot\bigwedge^r N^+_\BB$. In addition, every element in 
\[I\otimes\bigwedge^{r-s} N^+_\BB \subset \bigwedge^s N^+_\BB\]
is $\Fitt^0_R(C)_\BB$-divisible in the last exterior power. Hence,
\[I\otimes\bigwedge^{r-s} N^+_\BB =\Fitt^0_R(C) \bigwedge^s N^+_\BB.\]
Using the comment at the beginning of this proof, we get that
\[\ind\Bigl(\Phi_{N,M} (\psi)\Bigr)_\BB=\Fitt^0_R(C)_\BB.\qedhere\]
\end{proof}

There is one special case in which we can prove that the index is contained in the Fitting ideal. That is because we can show that this index is a principal ideal, so it contains all ideals to which it is pseudo-equal (see Notation \ref{not:pseudo_equal}).
\begin{proposition}
Consider the exact sequence of finitely generated Iwasawa modules
\[\xymatrix{0\ar[r] & M\ar[r]^\mu & N\ar[r]^\varepsilon & F\ar[r] & C\ar[r] &0,}\]
where $M$, $N$ and $F$ are torsion-free Iwasawa modules of ranks 
\[\begin{array}{cc}
\rank_\Lambda N=r,\ &\rank_\Lambda F=s,
\end{array}\]
for some natural numbers $r\geq s$. Let $\psi$ be a generator of $\bigcap^{r} N$. Then the index $\ind\Bigl(\Phi_{N,M} (\psi)\Bigr)$ is the unique principal ideal containing $\Fitt^0_\Lambda(C)$ with finite index.
\label{prop:bidual_map_iwasawa_fitting_inclusion}
\end{proposition}

\begin{proof}
Since $M$ is torsion-free, it is pseudo-isomorphic to a free Iwasawa module. Assume that its rank is larger than $r-s$. The additivity of the rank implies that $C$ has positive rank, so
\[\Fitt^0_{\Lambda}(C)=0.\]
Since $\Lambda$ contains no finite ideals, Proposition \ref{prop:bidual_map_fitting} implies that 
\[\ind(\Phi_{N,M})(\psi)=0.\]
Now assume that the rank of $M$ is $r-s$. Then $M$ is pseudo-isomorphic to $\Lambda^{r-s}$. As in Remark \ref{rem:dual_fin_index}, we can conclude that
\[\bigcap^{r-s} M\cong \bigcap^{r-s} \Lambda^{r-s} =\Lambda.\]
Then the index of every element in $\bigcap^{r-s} M$ is a principal ideal and, in particular, the index of $\phi_{N,M}(\psi)$. Since the principal ideals of the Iwasawa algebra contain all the ideals they are pseudo-equal to, we can conclude by Proposition \ref{prop:bidual_map_fitting} that 
\[\Fitt^0_\Lambda(C)\subset\ind\Bigl(\Phi_{N,M} (\psi)\Bigr).\qedhere\]
\end{proof}

The rank reduction maps satisfy the following commutative property. The following proposition follows from \cite[Proposition 2.4]{Sakamoto18}.
\begin{proposition}(\cite[Proposition 2.4]{Sakamoto18})
Consider the following exact sequences of Iwasawa modules,
\[\begin{aligned}
&\xymatrix{0\ar[r] & M_1 \ar[r]^{\mu_{12}} & M_2\ar[r]^{\pi_{12}} & \Lambda^s\ar[r] &0,}\\
&\xymatrix{0\ar[r] & M_2 \ar[r]^{\mu_{23}} & M_3\ar[r]^{\pi_{23}} & \Lambda^{s'}\ar[r] &0,}\\
&\xymatrix{0\ar[r] & M_1 \ar[r]^{\mu_{13}} & M_3\ar[r]^{\pi_{13}\ \ \ } & \Lambda^{s+s'}\ar[r] &0,}\\
\end{aligned}\]
satisfying that 
\begin{equation}
\begin{array}{ccc}
    \mu_{13}=\mu_{23}\circ\mu_{12},\ & \pi_{23}=(\pi_{13})_2,\ & \pi_{12}=(\pi_{13})_1\circ\mu_{23},
    \label{eq:compatibility}
\end{array}
\end{equation}
where $(\pi_{13})_i$ for $i=1,2$ denotes the composition of $\pi_{13}$ projection to the first or second component, respectively, in $\Lambda^{s+s'}=\Lambda^s\oplus \Lambda^{s'}$. Then
\[\phi_{M_3,M_1}=\phi_{M_3,M_2}\circ\phi_{M_2,M_1}.\]
\label{prop:bidual_map_comp}
\end{proposition}

It is possible to compute the kernel of the rank reduction maps.
\begin{proposition}(\cite[Lemma 2.17]{BurnsBullach})
Let $M$ and $N$ be modules over a ring $R$ which is either self-injective, a discrete valuation ring or the Iwasawa algebra. In addition, let
\[\phi_{N,M}:\ \bigcap^{r} N\to \bigcap^{r-1} M\]
be the rank reduction map constructed in Proposition \ref{prop:bidual_map} for the exact sequence
\[\xymatrix{0\ar[r] & M\ar[r]^\mu & N\ar[r]^\varepsilon &  R,}\]
Consider also the map 
\[\iota:\ \bigcap^r M\hookrightarrow \bigcap^r N.\]
Then $\ker(\phi_{N,M})=\Im(\iota)$.
\label{prop:bidual_map_iwasawa_kernel}
\end{proposition}

\begin{remark}
All the results in this section can be applied to the map constructed in Corollary \ref{cor:exterior_rankred}, provided that $M$ and $N$ are free Iwasawa modules.
\end{remark}

\section{Stark Systems}
\label{sec:iwasawa:stark}

The goal of this section is to give a construction of the ultra Stark systems, analogous to the classical case in \cite{Sakamoto18} and \cite{BurnsSano}. In order to do that, we assume that we are given a cartesian Selmer structure defined on a Galois representation satisfying Assumption \ref{ass:patched:basic}.

Let $\kn,\kr\in \NN$ be square-free products of Kolyvagin ultraprimes such that $\kn\mid \kr$. There is an exact sequence
\[\xymatrix{0\ar[r] & \bH^1_{\FF^\kn}(K,\bT)\ar[r] & \bH^1_{\FF^\kr}(K,\bT)\ar[r] & \prod_{\ku\mid \frac{\kr}{\kn}} \bH^1_{\s}(K_\ku,\bT).}\]

We want to use Corollary \ref{cor:exterior_rankred} to construct a map
\[\phi_{\kr,\kn}:\ \bigwedge^{\chi(\FF)+\nu(\kr)} \bH^1_{\FF^\kr}(K,\bT)\to \bigwedge^{\chi(\FF)+\nu(\kn)} \bH^1_{\FF^\kn}(K,\bT).\]
In order to do that, we fix, for every Kolyvagin ultraprime $\ku\in \PP$, an isomorphism 
\[\psi_\ku:\ \bH^1_{\s}(K_\ku,\bT)\cong \Lambda.\]
In addition, we need to fix an order in the set $\PP$. Such ordering fixes an isomorphism
\[\prod_{\ku\mid \frac{\kr}{\kn}} \bH^1_{\s}(K_\ku,\bT)\cong \Lambda^{\nu(\kr)-\nu(\kn)}.\]
Then the map $\phi_{\kr,\kn}$ is the one given by Corollary \ref{cor:exterior_rankred}.

In order to guarantee the commutative property of these maps, we need to modify the sign convention. Following \cite[\textsection 3.1]{BurnsSano}, we make the following definition.
\begin{definition}
For every pair of elements $\kn,\kr\in \NN$ such that $\kn\mid \kr$, we denote by $\ku_1,\ldots,\ku_{\nu(\kn)}$ and $\ku'_{1},\ldots, \ku'_{\nu(\kr/\kn)}$ the prime divisors of $\kn$ and $\kr/\kn$, ordered by the predetermined ordering. Similarly, let $\ku''_{1},\ldots, \ku''_{\nu(\kr)}$ the ordered set of prime divisors of $\kr$. Define $\sign(\kr,\kn)$ as the sign of the permutation
\[(\ku'_{1},\ldots, \ku'_{\nu(\kr/\kn)},\ku_1,\ldots,\ku_{\nu(\kn)})\mapsto (\ku''_{1},\ldots, \ku''_{\nu(\kr)}).\]
\label{def:sign}
\end{definition}

\begin{definition}
We define the map 
\[\Phi_{\kr,\kn}=\sign(\kr,\kn)\cdot \phi_{\kr,\kn}:\ \bigwedge^{\chi(\FF)+\nu(\kr)} \bH^1_{\FF^\kr}(K,\bT)\to \bigwedge^{\chi(\FF)+\nu(\kn)} \bH^1_{\FF^\kn}(K,\bT),\]
where $\sign(\kr,\kn)$ is defined in Definition \ref{def:sign} and the map $\phi_{\kr, \kn}$ is the map given in Corollary \ref{cor:exterior_rankred}. Note that this map is well defined only because both Selmer groups are free (see Proposition \ref{prop:iwasawa:selmer_free}).
\label{def:stark:transition_map}
\end{definition}

\begin{lemma}
Let $\kn,\kr,\ks\in \NN$ be square-free products of Kolyvagin ultraprimes such that $\kn\mid \kr\mid \ks$. Then
\[\Phi_{\ks,\kn}=\Phi_{\kr,\kn}\circ \Phi_{\ks,\kr}.\]
\label{lem:com} 
\end{lemma}

\begin{proof}
By Proposition \ref{prop:bidual_map_comp}, the map 
\[\sign(\ks,\kr)\cdot\sign(\kr,\kn)\cdot\bigl(\Phi_{\ks,\kr}\circ \Phi_{\kr,\kn}\bigr)\]
is the map induced from Corollary \ref{cor:exterior_rankred} by the exact sequence
\[\xymatrix{\bH^1_{\FF^\kn}(K,\bT)\ar@{>->}[r] & \bH^1_{\FF^\ks}(K,\bT)\ar[r] & \displaystyle{\prod_{\ku\mid \frac{\kr}{\kn}} \bH^1_{\s}(K_\ku,\bT)\times \prod_{\ku\mid \frac{\ks}{\kr}} \bH^1_{\s}(K_\ku,\bT)}}\]
and the isomorphisms
\[\begin{array}{cc}
\displaystyle{\prod_{\ku\mid \frac{\kr}{\kn}} \bH^1_{\s}(K_\ku,\bT)\cong \Lambda^{\nu(\kr/\kn)}},&\displaystyle{\prod_{\ku\mid \frac{\ks}{\kr}} \bH^1_{\s}(K_\ku,\bT) \cong \Lambda^{\nu(\ks/\kr)},}
\end{array}\]
induced by the order choice.

Similarly, the map $\sign(\ks,\kn)\Phi_{\ks,\kn}$ is induced by the exact sequence
\[\xymatrix{0\ar[r] & \bH^1_{\FF^\kn}(K,\bT)\ar[r] & \bH^1_{\FF^\ks}(K,\bT)\ar[r] & \displaystyle{\prod_{\ku\mid \frac{\ks}{\kn}} \bH^1_{\s}(K_\ku,\bT)}}\]
and the isomorphism
\[\prod_{\ku\mid \frac{\ks}{\kn}} \bH^1_{\s}(K_\ku,T)\cong \Lambda^{\nu(\ks/\kn)}.\]

Consider the following diagram in which the vertical isomorphisms are induced by the ones above mentioned and the horizontal ones are the natural maps:
\[\xymatrix{
\displaystyle{\det\left(\prod_{\ku\mid \frac{\kr}{\kn}} \bH^1_{\s}(K_\ku,\bT)\times \prod_{\ku\mid \frac{\ks}{\kr}} \bH^1_{\s}(K_\ku,\bT)\right)} \ar[r] \ar[d] & \displaystyle{\det\left(\prod_{\ku\mid \frac{\ks}{\kn}} \bH^1_{\s}(K_\ku,\bT)\right)}\ar[d] \\
\det\left(R^{\nu(\kr/\kn)}\times R^{\nu(\ks/\kr)}\right)\ar[r] & \det\left(R^{\nu(\ks/\kn)}\right).}
\]
Then this diagram is commutative up to the sign 
\[\sign(\ks,\kr)\sign(\ks,\kn)\sign(\kr,\kn).\]

By the identifications made above, we get that 
\[\Phi_{\ks,\kn}=\Phi_{\ks,\kr}\circ \Phi_{\kr,\kn}.\qedhere\]
\end{proof}

Therefore, the set of maps $\Phi_{\kr,\kn}$ forms an inverse system, so it makes sense to consider the elements in the inverse limit.

\begin{definition}
The module of Stark systems of $\FF$ is defined as the inverse limit
\[\bSS(\bT,\FF):=\varprojlim_{\kn\in \NN} \left(\bigwedge^{\chi(\FF)+\nu(n)} \bH^1_{\FF^\kn}(K,\bT)\right).\]
\label{def:stark_systems}
\end{definition}

\begin{remark}
Unlike the construction in \cite{BurnsSakamotoSano2}, we opted to construct Stark systems as collections in the exterior powers of Selmer groups instead of exterior biduals. The reason for this choice is that Iwasawa Selmer groups are free modules, so the rank reduction maps are well defined on exterior powers. It is then unnecessary to consider the exterior biduals, although it will be useful in \textsection\ref{sec:classical} to compare the ultra Kolyvagin systems with the classical theory.
\end{remark}

The link between Stark systems and the Fitting ideals of Selmer groups involves the construction of a sequence of theta ideals.

\begin{definition}
Let $\varepsilon\in \bSS(\bT,\FF)$ be a Stark system. We define its theta ideals as
\[\bTheta_i(\varepsilon):=\sum_{\kn\in \NN_i} \ind\left(\varepsilon_\kn,\bigwedge^{\chi(\FF)+\nu(\kn)} \bH^1_{\FF^\kn}(K,\bT)\right).\]
\label{def:stark_theta}
\end{definition}

\subsection{The module of Stark systems}

The main goal of this section will be to prove the following theorem, which says that Stark systems are controlled by weak core vertices.

\begin{theorem}
Suppose that $\bT$ satisfies Assumption \ref{ass:patched:basic} and let $\kc\in \NN$ be a weak core vertex. Then the projection map
\[\bSS(\bT,\FF)\to \bigwedge^{\chi(\FF)+\nu(\kc)} \bH^1_{\FF^\kc}(K,\bT)\]
is an isomorphism.
\label{th:stark_core_projection}
\end{theorem}

\begin{proof}
By Proposition \ref{prop:weak_core_vertex_existence}, we only need to prove that if $\kc\in \NN$ is a weak core vertex and $\ku\in \PP$ does not divide $\kc$, the map
\[\Phi_{\kc\ku,\kc}:\ \bigwedge^{r+\nu(\kc\ku)} \bH^1_{\FF^{\kc\ku}} (K,\bT)\to \bigwedge^{r+\nu(\kc)} \bH^1_{\FF^\kc} (K,\bT)\]
is an isomorphism. This map is induced by the exact sequence
\[\xymatrix{0\ar[r] & \bH^1_{\FF^\kc}(K,\bT)\ar[r]&\bH^1_{\FF^{\kc\ku}}(K,\bT)\ar[r] & \bH^1_\s(K_\ku, \bT) \ar[r] &0.}\]
Then Proposition \ref{prop:bidual_map_fitting} implies that $\Phi_{\kc\ku,\kc}$ is surjective so, being both domain and codomain free modules of rank one, it is an isomorphism.
\end{proof}

The above control theorem allows the computation of the module of Stark systems.

\begin{corollary}
If $\bT$ satisfies Assumption \ref{ass:patched:basic}, the module of Stark systems $\bSS(\bT,\FF)$ is a free $\Lambda$-module of rank one.
\label{cor:stark_free}
\end{corollary}

\begin{proof}
Let $\kc\in \NN$ be a weak core vertex. By Proposition \ref{prop:core_pseudo_free}, $\bH^1_{\FF^\kc}(K,\bT)$ is free of rank ${\chi(\FF)+\nu(\kc)}$. Thus, Theorem \ref{th:stark_core_projection} implies that 
\[\bSS(\bT,\FF)\cong \bigwedge^{\chi(\FF)+\nu(\kc)} \bH^1_{\FF^\kc}(K,\bT)\cong \bigwedge^{\chi(\FF)+\nu(\kc)} \Lambda^{\chi(\FF)+\nu(\kc)} =\Lambda.\qedhere\]
\end{proof}

Therefore, it makes sense to talk about generators of the module of Stark systems.

\begin{definition}
The generators of $\bSS(\bT,\FF)$ are called \emph{primitive} Stark systems.
\end{definition}

\begin{theorem}
Suppose that Assumption \ref{ass:patched:basic} holds and let $\varepsilon=(\varepsilon_\kn)_{\kn\in \NN}\in \bSS(\bT,\FF)$ be a primitive Stark system. For every $\kn\in \NN$, we have that $\ind(\varepsilon_n)$ is a principal ideal such that

\[\Fitt^0_\Lambda\Bigl(\bH^1_{\FF_\kn^*}(K,\bT^*)^\vee\Bigr)\subset_{f.i.}\ind(\varepsilon_\kn).\]
\label{th:stark:fitt_sel_res}
\end{theorem}

\begin{proof}
Fix $\kn\in \NN$. By Remark \ref{rem:weak_core_div}, there exists a core vertex $\kc$ such that $\kn\mid \kc$. Consider then the map in Definition \ref{def:stark:transition_map}:
\[\Phi_{\kc,\kn}:\ \bigwedge^{\chi(\FF)+\nu(\kc)} \bH^1_{\FF^\kc}(K,\bT)\to \bigwedge^{\chi(\FF)+\nu(\kn)} \bH^1_{\FF^\kn}(K,\bT).\]

Since $\varepsilon$ generates $\bSS(\bT,\FF)$, Theorem \ref{th:stark_core_projection} implies that $\varepsilon_{\kc} $ generates 
\[\bigwedge^{\chi(\FF)+\nu(\kc)}\bH^1_{\FF^\kc}(K,\bT).\]
Since $\varepsilon_\kn=\Phi_{\kc,\kn}(\varepsilon_\kc)$, then Proposition  \ref{prop:bidual_map_iwasawa_fitting_inclusion} implies that 
\[ \Fitt^0_\Lambda\Bigl(\bH^1_{\FF_\kn^*}(K,\bT^*)^\vee\Bigr)\subset_{f.i.} \ind(\varepsilon_\kn).\qedhere\]
\end{proof}

\subsection{Higher Fitting ideals of the Selmer group}

We can also use Stark systems to compute higher Fitting ideals of the Selmer group. Indeed, we will see the higher Fitting ideals coincide up to finite index with the theta ideals defined in Definition \ref{def:stark_theta}

\begin{theorem}
If Assumption \ref{ass:patched:basic} holds, every primitive Stark system $\varepsilon\in \bSS(\bT,\FF)$ satisfies that
\[\Fitt^i_\Lambda\Bigl(\bH^1_{\FF^*}(K,\bT^*)^\vee\Bigr)\subset_{f.i.}\bTheta_i(\varepsilon).\]
\label{th:iwasawa:stark_theta_fitting}
\end{theorem}

\begin{proof}
Choose a weak core vertex $\kc\in \NN$, which induces an exact sequence
\[\xymatrix{\bH^1_{\FF}(K,\bT)\ar@{>->}[r] & \bH^1_{\FF^{\kc}}(K,\bT)\ar[r] &\displaystyle{\bigoplus_{\ku\mid \kc}\bH^1_\s(K_\ku,\bT)}\ar@{->>}[r] & \bH^1_{\FF^*}(K,\bT^*)^\vee.}\]

For every $\kq\mid \kc$, there is another exact sequence 
\[\xymatrix{\bH^1_{\FF^\kq}(K,\bT)\ar@{>->}[r] & \bH^1_{\FF^{\kc}}(K,\bT)\ar[r] &\displaystyle{\bigoplus_{\ku\mid \frac{\kc}{\kq}}\bH^1_\s(K_\ku,\bT)}\ar@{->>}[r] & \bH^1_{(\FF^*)_\kq}(K,\bT^*)^\vee.}\]

By the definition of the Fitting ideals, we can deduce that 
\[\Fitt^i_\Lambda\Bigl(\bH^1_{\FF^*}(K,\bT^*)^\vee\Bigr)=\sum_{\kq\mid \kc} \Fitt^{i-1}_\Lambda\Bigl(\bH^1_{(\FF^*)_\kq}(K,\bT^*)^\vee\Bigr).\] 

Proposition \ref{prop:weak_core_vertex_existence} implies that 
\[\Fitt^i_\Lambda\Bigl(\bH^1_{\FF^*}(K,\bT^*)^\vee\Bigr)=\sum_{\kq\in \PP} \Fitt^{i-1}_\Lambda\Bigl(\bH^1_{(\FF^*)_\kq}(K,\bT^*)^\vee\Bigr).\] 

We can now proceed to prove Theorem \ref{th:iwasawa:stark_theta_fitting} by induction. For $i=0$, it is Theorem \ref{th:stark:fitt_sel_res} for $n=1$. Assuming the result holds true for $i-1$, we can apply it to all Selmer structures $\FF^\kq$ to obtain that 
\[\Fitt^i_\Lambda\Bigl(\bH^1_{\FF^*}(K,\bT^*)^\vee\Bigr)=\sum_{\kq\in \PP} \sum_{\kn\in \NN_{i-1},\ \kq\nmid \kn} \Fitt^0_\Lambda\Bigl(\bH^1_{(\FF^*)_{\kq\kn}}(K,\bT^*)^\vee\Bigr).\]
Rearranging the sum,
\[\Fitt^i_\Lambda\Bigl(\bH^1_{\FF^*}(K,\bT^*)^\vee\Bigr)=\sum_{\kn\in \NN_{i}} \Fitt^0_\Lambda\Bigl(\bH^1_{(\FF^*)_{\kn}}(K,\bT^*)^\vee\Bigr).\]
By Theorem \ref{th:stark:fitt_sel_res}, we obtain that 
\[\Fitt^i_\Lambda\Bigl(\bH^1_{\FF^*}(K,\bT^*)^\vee\Bigr)\lessapprox \sum_{\kn\in \NN_{i}}\ind(\varepsilon_\kn)=\bTheta_i(\varepsilon).\qedhere\]
\end{proof}

\begin{remark}
It is expected that the ideals $\bTheta_i(\varepsilon)$ can be recovered as a limit of the analogues of classical Stark systems, as those appearing in \cite{BurnsSakamotoSano2}. If that comparison can be done, Theorem \ref{th:iwasawa:stark_theta_fitting} can be seen as a weaker version of \cite[Theorem 4.12]{BurnsSakamotoSano2}.
\end{remark}

\section{Kolyvagin systems}
\label{sec:iwasawa:kol}

In this section, we develop the theory of ultra Kolyvagin systems in order to give a computation of the Fitting ideals of Selmer groups. Although that has already been done using Stark systems, the computations of the theta ideals associated with Kolyvagin systems are more explicit from the classes in the Euler system.
\subsection{Definition of Kolyvagin systems}
Unlike Stark systems, Kolyvagin systems can be only defined when the core rank is positive. The reason is that Kolyvagin system relations will be defined on an exterior power of rank $\chi(\FF)-1$, which is only constructed when the core rank is positive. In addition, we need to assume that the prime $p$ is at least $5$. This assumption was already made in \cite{MazurRubin} and it is required so we can construct suitable primes using Proposition \ref{prop:patched:cheb}.
\begin{namedass}{(Pos)}
$\FF$ is a cartesian Selmer structure of core rank $\chi(\FF)\geq 1$ and the prime $p$ is at least $5$.
\label{ass:core_rank_positve}
\end{namedass}

For every $\kn\in \NN$ and $\ku\in \PP$, there are exact sequences
\[\xymatrix{0\ar[r] & \bH^1_{{\FF(\kn)_{\ku}}}(K,\bT)\ar[r] & \bH^1_{\FF(\kn)}(K,\bT)\ar[r]& \bH^1_{\f}(K_\ku,\bT),}\] 
\[\xymatrix{0\ar[r] & \bH^1_{{\FF(\kn)_{\ku}}}(K,\bT)\ar[r] & \bH^1_{\FF(\kn\ku)}(K,\bT)\ar[r]& \bH^1_{\tr}(K_\ku,\bT).}\] 
Recall that, in the previous section, we had defined an isomorphism
\[\psi_\ku:\ \bH^1_{\s}(K_\ku,\bT)\to R.\]
Together with the finite-singular map defined in Proposition \ref{prop:patched:finite-singular}, it induces another isomorphism
\[\phi_\ku=\psi_\ku\circ\phi_\ku^{\fs}:\ \bH^1_\f(K_\ku,\bT)\cong R.\]
Under these identifications, the above exact sequences and Corollary \ref{cor:exterior_rankred} induce maps
\[\alpha_{\kn,\ku}: \bigwedge^{\chi(\FF)} \bH^1_{\FF(\kn)}(K,\bT) \to \bigwedge^{\chi(\FF)-1} \bH^1_{\FF(\kn)_\ku}(K,\bT),\] 
\[\beta_{\kn,\ku}: \bigwedge^{\chi(\FF)} \bH^1_{\FF(\kn\ku)}(K,\bT) \to \bigwedge^{\chi(\FF)-1} \bH^1_{\FF(\kn)_\ku}(K,\bT).\] 
\begin{definition}

A \emph{Kolyvagin system} is a collection
\[(\kappa_\kn)_{\kn\in \NN} \in \prod_{\kn\in \NN} \left(\bigwedge^{\chi(\FF)} \bH^1_{\FF(\kn)}(K,\bT)\right)\]
satisfying, for all $\kn\in \NN$ and $\ku\in \PP$ not dividing $\kn$, that 
\[\alpha_{\kn,\ku}(\kappa_\kn)=\beta_{\kn,\ku}(\kappa_{\kn\ku}).\]
\label{def:kol_highrank}
\end{definition}

\begin{remark}
When $\chi(\FF)=1$, Definition \ref{def:kol_highrank} coincides with the construction in \cite{MazurRubin}, modified as in \cite{Sweeting} to consider local conditions at non-constant ultraprimes. Indeed, one could see that
\[\begin{array}{cc}
    \alpha_{\kn,\ku}(\kappa_\kn)=\phi_\ku(\loc_\ku(\kappa_\kn)),& \beta_{\kn,\ku}(\kappa_{\kn\ku})=\psi_{\ku}(\loc_\ku(\kappa_{\kn\ku})).
\end{array}\]
Only because $\phi_\ku$ and $\psi_{\ku}$ are defined satisfying $\phi_{\ku}=\psi_\ku\circ\phi_\ku^{\fs}$, we can see that Kolyvagin system relations in both settings are equivalent.
\label{rem:kol_relation}
\end{remark}

\begin{remark}
By Proposition \ref{prop:iwasawa:selmer_free}, all the Selmer groups $\bH^1_{\FF(\kn)}(K,\bT)$ are free modules, so there is a canonical identification 
\[\bigwedge^{\chi(\FF)}\bH^1_{\FF(\kn)}(K,\bT)=\bigcap^{\chi(\FF)}\bH^1_{\FF(\kn)}(K,\bT).\]
\end{remark}

Similarly to Definition \ref{def:stark_theta}, we can associate a sequence of theta ideals with every Kolyvagin system.
\begin{definition}
Let $\kappa\in \bKS(\bT,\FF)$ be a Kolyvagin systems. We define its theta ideals as
\[\bTheta_i(\kappa):=\sum_{n\in \NN_i} \ind\left(\kappa_n,\bigwedge^{\chi(\FF)} \bH^1_{\FF(\kn)}(K,\bT)\right).\]
\label{def:kol_theta}
\end{definition}

\subsection{Kolyvagin systems from Stark systems}

In this section, we define a regulator map from the module of Stark systems to the Kolyvagin systems and prove that it is, indeed, an isomorphism.

For every $\kn\in \NN$, there is an exact sequence
\[\xymatrix{0\ar[r] & \bH^1_{\FF(\kn)}(K,\bT)\ar[r] &\bH^1_{\FF^\kn}(K,\bT) \ar[r] & \displaystyle \prod_{\ku\mid \kn} \bH^1_{\f}(K_\ku, \bT).}\] 
With the ordering of the ultraprimes previously established, there is a well defined isomorphism
\[\bigoplus_{\ku\mid \kn}\phi_\ku:\ \prod_{\ku\mid \kn} \bH^1_{\f}(K_\ku, \bT)\cong \Lambda^{\nu(\kn)}.\]
With these identifications, Corollary \ref{cor:exterior_rankred} constructs a map
\[(-1)^{\nu(\kn)}\Reg_\kn:\ \bigwedge^{\chi(\FF)+\nu(\kn)} \bH^1_{\FF^\kn}(K,T)\to \bigwedge^{\chi(\FF)} \bH^1_{\FF(\kn)}(K,T).\]
The sign modification is introduced to guarantee that the regulator maps produce Kolyvagin systems.
\begin{theorem}
Let $\FF$ be a cartesian Selmer structure of core rank $\chi(\FF)\geq 1$. Combined for all $n\in \NN$, the maps $\Reg_n$ induce a regulator map
\[\Reg:\ \bSS(\bT,\FF)\to \bKS(\bT,\FF).\]
\label{th:regulator_bijective}
\end{theorem}

\begin{proof}
We need to show that the image of the regulator of a Stark system satisfies the Kolyvagin system relations, i.e., that for every $\kn\in \NN$ and $\ku\in \PP$ not dividing $\kn$, we have that
\[\Bigl(\alpha_{\kn,\ku}\circ \Reg_\kn\circ \Phi_{\kn\ku,\kn} \Bigr)(\varepsilon_{\kn\ku})=\Bigl(\beta_{\kn,\ku}\circ \Reg_{\kn\ku}\Bigr)(\varepsilon_{\kn\ku}),\]
since the Stark system relation implies that $\Phi_{\kn\ku,\kn}(\varepsilon_{\kn\ku})=\varepsilon_{\kn}$.

The maps $\Bigl(\alpha_{\kn,\ku}\circ \Reg_\kn\circ\Phi_{\kn\ku,\kn}\Bigr)$ and $\Bigl(\beta_{\kn,\ku}\circ \Reg_{\kn\ku}\Bigr)$ can be computed using Corollary \ref{cor:exterior_rankred} and the exact sequences
\[\xymatrix{
    \bH^1_{\FF_\ku(\kn)}(K,\bT)\ar@{>->}[r] & \bH^1_{\FF^{\kn\ku}}(K,\bT) \ar[r] & \displaystyle{\bH^1_{\tr}(K_\ku,\bT)\oplus\bigoplus_{\kq\mid \kn}\bH^1_{\f}(K_\kq,\bT)}}\oplus \bH^1_\f(K_\ku,\bT),\]
     \[\xymatrix{\bH^1_{\FF_\ku(\kn)}(K,\bT)\ar@{>->}[r] & \bH^1_{\FF^{\kn\ku}}(K,\bT) \ar[r] & \displaystyle{\bigoplus_{\kq\mid \kn\ku}\bH^1_{\f}(K_\kq,\bT)\oplus \bH^1_\tr(K_\ku,\bT)}.
    }\]
Note that the ordering matters to establish an isomorphism to $\Lambda^{\nu(\kn\ku)+1}$, used in Corollary \ref{cor:exterior_rankred} to construct the rank reduction maps\footnote{A different order of the summands might result in a sign difference when constructing the rank reduction maps.}. It is important to clarify that, when we write ${\bigoplus_{\kq\mid\ka}}$, we are using the canonical order in the ultraprimes dividing $\ka$.

With these ordering choices, the maps constructed by Corollary \ref{cor:exterior_rankred} differ from $\Bigl(\alpha_{\kn,\ku}\circ \Reg_\kn\circ\Phi_{\kn\ku,\kn}\Bigr)$ and $\Bigl(\beta_{\kn,\ku}\circ \Reg_{\kn\ku}\Bigr)$ in the signs of $(-1)^{\nu(\kn\ku)}\sign(\kn\ku,\kn)$ and $(-1)^{\nu(\kn)}$, respectively. This sign difference is due to a reordering of the summands, so we can conclude that 
\[\Bigl(\alpha_{\kn,\ku}\circ \Reg_\kn\circ\Phi_{\kn\ku,\kn}\Bigr)=\Bigl(\beta_{\kn,\ku}\circ \Reg_{\kn\ku}\Bigr).\qedhere\]
\end{proof}
It is especially important to note that the regulator maps are isomorphisms on core vertices.

\begin{proposition}
Let $\kn\in \NN$ be a core vertex. Then the regulator map
\[\Reg_\kn:\ \bigwedge^{\chi(\FF)+\nu(\kn)} \bH^1_{\FF^\kn}(K,\bT)\to \bigwedge^{\chi(\FF)} \bH^1_{\FF(\kn)}(K,\bT)\]
is an isomorphism.
    \label{prop:reg_core_iso}
\end{proposition}

\begin{proof}
By Proposition \ref{prop:core_pseudo_free},
\[\bigwedge^{\chi(\FF)+\nu(\kn)} \bH^1_{\FF^\kn}(K,\bT)\cong \bigwedge^{\chi(\FF)} \bH^1_{\FF(\kn)}(K,\bT)\cong \Lambda\]
and, by Proposition \ref{prop:bidual_map_fitting}, the regulator of a generator of $\bigwedge^{\chi(\FF)+\nu(\kn)} \bH^1_{\FF^\kn}(K,\bT)$ is an indivisible element in $\bigwedge^{\chi(\FF)} \bH^1_{\FF(\kn)}(K,\bT)$, so $\Reg_\kn$ is an isomorphism.
\end{proof}

\subsection{The module of Kolyvagin systems}

We will see in this section that the module of Kolyvagin systems is controlled by core vertices, similarly to the control made by weak core vertices to Stark systems in Theorem \ref{th:stark_core_projection}.
\begin{theorem}
Suppose that Assumptions \ref{ass:core_rank_positve}, \ref{ass:patched:basic} and \ref{ass:H2} hold and let $\kn\in \NN$ be a core vertex. Then the projection map
\[\bKS(\bT,\FF)\to \bigwedge^{\chi(\FF)} \bH^1_{\FF(\kn)}(K,\bT)\]
is an isomorphism.
\label{th:kol_core_projection}
\end{theorem}

The proof of Theorem \ref{th:kol_core_projection} is more complicated than its analogous in Theorem \ref{th:stark_core_projection}. It will be presented below as a sequence of lemmas. We start by showing the surjectivity.
\begin{lemma}
Under assumptions \ref{ass:core_rank_positve}, \ref{ass:patched:basic} and \ref{ass:H2}, the projection map
\[\bKS(\bT,\FF)\to \bigwedge^{\chi(\FF)} \bH^1_{\FF(\kc)}(K,\bT)\]
is surjective for every core vertex $\kc\in \NN$.
\label{lem:kol_core_projection_sur}
\end{lemma}

\begin{proof}
Consider the following commutative diagram 
\[\xymatrix{\bSS(\bT,\FF)\ar[r]^{\Reg} \ar[d]& \bKS(\bT,\FF)\ar[d]\\
\displaystyle\bigwedge^{r+\nu(c)} \bH^1_{\FF^\kc}(K,\bT)\ar[r]^{\Reg_{\kc}} &  \displaystyle\bigwedge^{r} \bH^1_{\FF(\kc)}(K,\bT).}\]
By Theorem \ref{th:stark_core_projection} and Proposition \ref{prop:reg_core_iso}, the left vertical map and the bottom horizontal map are isomorphisms. Hence, the right vertical map is surjective.
\end{proof}

We next prove the injectivity of the map in Theorem \ref{th:kol_core_projection}. The argument takes a Kolyvagin system in the kernel of this map and proves inductively that all the classes $\kappa_n$ have to vanish. The following set of vertices, containing all core vertices, have control over the Kolyvagin systems.

\begin{definition}
Let $\NN^{(0)}\subset \NN$ be the subset of vertices defined as
\[\NN^{(0)}:=\left\{\kn\in \NN:\ \rank_\Lambda\ \bH^1_{(\FF^*)(\kn)}(K,\bT^*)^\vee=0\right\}.\]
\end{definition}

\begin{notation}
For every $\kn\in \NN$, we define the module $\Upsilon(\kn)$ as the unique cyclic submodule of $\bigwedge^{\chi(\FF)} \bH^1_{\FF(\kn)}(K,\bT)$ such that 
\[\Upsilon(\kn)\approx\biggl(\Fitt^0_{\Lambda}\Bigl(\bH^1_{(\FF^*)(\kn)}(K,\bT^*)^\vee\Bigr)\biggr)\cdot \left(\bigwedge^{\chi(\FF)} \bH^1_{\FF(\kn)}(K,\bT)\right).\]
Note that $\Upsilon(\kn)=0$ whenever $\kn\notin\NN^{(0)}$.
\end{notation}

\begin{notation}
In addition, we denote $\Omega_{\kn,\ku}$ the unique cyclic submodule of $\bigwedge^{\chi(\FF)-1}\bH^1_{\FF_\ku(\kn)}(K,\bT)$ as the unique cyclic submodule satisfying that 
\[\Omega_{\kn,\ku}\approx \biggl(\Fitt^0_\Lambda\Bigl(\bH^1_{\FF^\ku(\kn)}(K,\bT^*)^\vee\Bigr)\biggr)\cdot\left(\bigwedge^{\chi(\FF)-1}\bH^1_{\FF_\ku(\kn)}(K,\bT)\right).\]
\end{notation}

\begin{remark}
Both $\Upsilon(\kn)$ and $\Omega_{\kn,\ku}$ are well-defined since, whenever the Fitting ideal in the product is non-zero, the exterior power is a free Iwasawa module of rank one.
\end{remark}

\begin{lemma}
Suppose that Assumptions \ref{ass:core_rank_positve}, \ref{ass:patched:basic} and \ref{ass:H2} hold and let $\kn\in \NN$ and $\ku\in \PP$ be elements such that $\kn,\kn\ku\in \NN^{(0)}$. Then there is an isomorphism 
\[\gamma_{\kn,\kn\ku}:\Upsilon(\kn)\to \Upsilon(\kn\ku)\]
such that every $\kappa\in \bKS(\bT,\FF)$ satisfying that $\kappa_\kn\in \Upsilon(\kn)$, then $\gamma_{\kn,\kn\ku}(\kappa_\kn)=\kappa_{\kn\ku}$.
\label{lemma:iwasawa:kol_gamma_n_nu}
\end{lemma}

\begin{proof}

Assume first that the rank of $\bH^1_{(\FF^*)^\ku(\kn)}(K,\bT^*)^\vee$ is zero or, equivalently, that
\[\Fitt^0_\Lambda\Bigl(\bH^1_{(\FF^*)^\ku(\kn)}(K,\bT^*)^\vee\Bigr)\not\approx 0.\]
Denote
\[C_{\kn,\ku}:=\coker\biggl(\loc_\kq:\ \bH^1_{\FF(\kn)}(K,\bT)\to \bH^1_\f(K_\kq,\bT)\biggr).\]
By Proposition \ref{prop:patched:global_duality} and Corollary \ref{cor:fitting_add_iwasawa}
\begin{equation}
\Fitt^0_\Lambda \Bigl(\bH^1_{(\FF^*)^\ku(\kn)}(K,\bT^*)^\vee\Bigr)\cong \Fitt^0_\Lambda (C_{\kn,\ku})\Fitt^0_\Lambda \Bigl(\bH^1_{(\FF^*)(\kn)}(K,\bT^*)^\vee\Bigr).
\label{eq:C_fitt_add}
\end{equation}

Since $\kn\in \NN^{(0)}$, $\bigwedge^{\chi(\FF)} \bH^1_{\FF(\kn)}(K,\bT)$ is free of rank one, so  Proposition \ref{prop:bidual_map_fitting} implies that the map $\alpha_{\kn,\ku}$ induces an isomorphism
\[\alpha_{\kn,\ku}:\bigwedge^{\chi(\FF)}\bH^1_{\FF(\kn)}(K,\bT) \to \HH_{\kn,\ku} \cdot\left(\bigwedge^{\chi(\FF)-1}\bH^1_{\FF_\ku(\kn)}(K,\bT)\right),\]
where $\HH_{\kn,\ku}$ denotes the unique principal ideal of $\Lambda$ pseudo-equal to $\Fitt^0_\Lambda (C_{\kn,\ku})$. By \eqref{eq:C_fitt_add}, it restricts to an isomorphism
\[\alpha_{\kn,\ku}:\ \Upsilon(\kn)\to \Omega_{\kn,\ku}.\]
A similar argument leads to another isomorphism
\[\beta_{\kn,\ku}:\ \Upsilon(\kn\ku)\to \Omega_{\kn,\ku}.\]
We can thus define the isomorphism
\[\gamma_{\kn,\kn\ku}:=\beta_{\kn,\ku}^{-1}\circ\alpha_{\kn,\ku}:\ \Upsilon(\kn)\to \Upsilon(\kn\ku).\]

Assume now that the rank of $\bH^1_{(\FF^*)^\ku(\kn)}(K,\bT^*)^\vee$ is one. By the argument in Corollary \ref{cor:patched:cheb_rank_red}, we can find an ultraprime $\kq\in \PP$ such that the rank of $\bH^1_{(\FF^*)^\ku_\kq(\kn)}(K,\bT^*)^\vee$ is zero and $C_{\kn,\ku}=0$. Then the exact sequence
\[\xymatrix{0\ar[r] & \coker(\loc_\kq)\ar[r] & \bH^1_{(\FF^*)^\kq(\kn)}(K,\bT^*)^\vee \ar[r] & \bH^1_{(\FF^*)(\kn)}(K,\bT^*)^\vee\ar[r] & 0,}\]
 given by Proposition \ref{prop:patched:global_duality}, implies that $\bH^1_{(\FF^*)^\kq(\kn)}(K,\bT^*)$ has rank $0$, since $\kn\in \NN$, so we can define the map $\gamma_{\kn,\kn\kq}$ as above.

On the other hand, since the rank of $\bH^1_{(\FF^*)^\ku(\kn)}(K,\bT^*)^\vee$ is one, the fact that $\kn,\kn\ku\in \NN^{(0)}$ and Proposition \ref{prop:patched:global_duality} imply that the following maps are zero:
\[\begin{aligned}
&\loc_\ku:\ \bH^1_{\FF(\kn)}(K,\bT)\to \bH^1_\f(K_\ku,\bT),\\ &\loc_\ku:\ \bH^1_{\FF(\kn\ku)}(K,\bT)\to \bH^1_\tr(K_\ku,\bT).
\end{aligned}\]
Hence $\bH^1_{\FF(\kn)}(K,\bT)=\bH^1_{\FF_\ku(\kn)}(K,\bT)=\bH^1_{\FF(\kn\ku)}(K,\bT)$. Then the map
\[\loc_\kq:\ \bH^1_{\FF(\kn\ku)}(K,\bT)\to \bH^1_\f(K_\kq,\bT)\]
 is also surjective. Similarly, since $\kn\ku\in \NN^{(0)}$, $\bH^1_{(\FF^*)^\kq(\kn\ku)}(K,\bT^*)^\vee$ has rank $0$. Hence the map $\gamma_{\kn\ku,\kn\ku\kq}$ can be defined.

Finally, the map
\[\loc_\kq:\ \bH^1_{\FF_\ku(\kn)}(K,\bT)\to \bH^1_\f(K_\kq,\bT)\]
is surjective too. Thus Lemma \ref{lem:patched:tr_res} implies that
\[\bH^1_{(\FF^*)^\ku(\kn\kq)}(K,\bT^*)=\bH^1_{(\FF^*)^\ku_\kq(\kn)}(K,\bT^*).\]
By the assumption on the ultraprime $\kq\in \PP$, we have that the Selmer group $\bH^1_{(\FF^*)^\ku(\kn\kq)}(K,\bT^*)^\vee$ has rank zero, implying that $\kn\kq$ and $\kn\ku\kq$ are both in $\NN^{(0)}$ and that the map $\gamma_{\kn\kq,\kn\ku\kq}$ can be defined. We then define the map
\[\gamma_{\kn,\kn\ku}=\gamma_{\kn\ku,\kn\ku\kq}^{-1}\circ \gamma_{\kn\kq,\kn\ku\kq}\circ \gamma_{\kn,\kn\kq}:\ \Upsilon(\kn)\to \Upsilon(\kn\ku).\]
By construction, every Kolyvagin system $\kappa \in \bKS(\bT,\FF)$ in which $\kappa_{\kn}\in \Upsilon(\kn)$ satisfies that $\gamma_{\kn,\ku}(\kappa_{\kn})=\kappa_{\kn\ku}$.
\end{proof}

\begin{remark}
One can show that the map $\gamma_{\kn,\ku}$ is independent of the choice of the auxiliary prime $\kq$. However, we do not make use of this fact and only need the existence of such map for our purposes.
\end{remark}

\begin{lemma}
Suppose that $\bT$ satisfies Assumptions \ref{ass:core_rank_positve}, \ref{ass:patched:basic} and \ref{ass:H2}. Let $\kq_1,\kq_2\in \PP$ and let $\kn_1,\kn_2\in \NN^{(0)}$ be vertices such that $\kq_1\mid \kn_1$ and $\kq_2\mid \kn_2$. Then there is an ultraprime $\ku\in \PP$ such that both $\kn_i\ku/\kq_i\in \NN^{(0)}$ for $i=1,2$ and there are isomorphisms
\[\gamma_{\kn_i,\kn_i\ku/\kq_i}:\ \Upsilon(\kn_i)\to \Upsilon(\kn_i\ku/\kq_i)\]
such that, for all $\kappa\in \bKS(\bT,\FF)$, if $\kappa_{\kn_i}\in \Upsilon(\kn_i)$, the equality $\gamma_{\kn_i,\kn_i\ku/\kq_i}(\kappa_{\kn})=\kappa_{\kn_i\ku/\kq_i}$ holds.
\label{lemma:iwasawa:kol_gamma_tau}
\end{lemma}

\begin{proof}
We can apply Propositions \ref{prop:patched:cheb} and \ref{prop:patched:cheb_rank_red} to find $\ku\in \PP$ such that for $i=1,2$, 
\begin{itemize}
\item If $\rank_\Lambda\bH^1_{\FF_{\kq_i}(\kn_i/\kq_i)}(K,\bT)\geq 1$, the map
 \[\loc_{\ku}:\ \bH^1_{\FF_{\kq_i}(\kn_i/\kq_i)}(K,\bT)\to \bH^1_\f(K_\ku,\bT)\]
is surjective.
\item If $\rank_\Lambda\bH^1_{\FF_{\kq_i}(\kn_i/\kq_i)}(K,\bT) =0$, the map
\begin{equation}
    \loc_{\ku}:\ \bH^1_{\FF(\kn_i/\kq_i)}(K,\bT)\to \bH^1_\f(K_\ku,\bT)
    \label{eq:loc}
\end{equation}
is surjective.
\item $\rank_\Lambda\ \bH^1_{(\FF^*)_\ku(\kn_i/\kq_i)}(K,\bT^*)^\vee=0$.
\end{itemize}
Note that the last condition can be achieved since $\kn_i\in \NN^{(0)}$ implies that $\rank_\Lambda\ \bH^1_{(\FF^*)(\kn_i/\kq_i)}(K,\bT^*)^\vee \leq 1$

Independently of the rank of $\bH^1_{\FF^(\kn_i/\kq_i)_{\kq_i}}(K,\bT)$, the map in \eqref{eq:loc} is also surjective, so Lemma \ref{lem:patched:tr_res} implies that
\[\rank_\Lambda \bH^1_{(\FF^*)(\kn_i\ku/\kq_i)}(K,\bT^*)^\vee=\rank_\Lambda\bH^1_{(\FF^*)_\ku(\kn_i/\kq_i)}(K,\bT^*)^\vee=0,\]
so $\kn_i\ku/\kq_i\in \NN^{(0)}$.

When $\rank_\Lambda\ \bH^1_{\FF(\kn_i/\kq_i)_{\kq_i}}(K,\bT)\geq 1$, the map 
   \[\loc_{\ku}:\ \bH^1_{\FF(\kn_i)}(K,\bT)\to \bH^1_\f(K_\ku,\bT)\]
is surjective, so Lemma \ref{lem:patched:tr_res} implies that
\[\rank_\Lambda\bH^1_{(\FF^*)(\kn_i\ku)}(K,\bT^*)=\rank_\Lambda\bH^1_{(\FF^*)_\ku(\kn_i)}(K,\bT^*)=0,\]
where the last equality follows from the fact $\kn_i\in \NN^{(0)}$.

In this case, we can construct the map $\gamma_{\kn_i,\kn_i\ku/\kq_i}$ as the composition of maps constructed in Lemma \ref{lemma:iwasawa:kol_gamma_tau}, i.e., 
\[\gamma_{\kn_i,\kn_i\ku/\kq_i}= \gamma_{\kn_i/\kq_i\ku,\kn_i\ku}^{-1}\circ \gamma_{\kn_i,\kn_i\ku}.\]

When $\rank_\Lambda\ \bH^1_{\FF(\kn_i/\kq_i)_{\kq_i}}(K,\bT)=0$, then we necessarily have, by Proposition \ref{prop:patched:rank_modified}, that $\chi(\FF)=1$ and, by Proposition \ref{prop:patched:global_duality}, we also get that 
\[\rank_\Lambda\ \bH^1_{(\FF^*)(\kn_i/\kq_i)^{\kq_i}}(K,T^*)^\vee=\rank_\Lambda\ \bH^1_{(\FF^*)(\kn_i/\kq_i)}(K,T^*)^\vee=0.\]
We can now use the maps in Lemma \ref{lemma:iwasawa:kol_gamma_n_nu}, to construct
\[\gamma_{\kn_i,\kn_i\ku/\kq_i}= \gamma_{\kn_i/\kq_i,\kn_i\ku/\kq_i}\circ \gamma_{\kn_i/\kq_i,\kn_i}^{-1}.\qedhere\]
\end{proof}

\begin{remark}
Lemma \ref{lemma:iwasawa:kol_gamma_tau} is the only place in this proof where the condition $p\geq 5$ is required. For the other parts of this proof, it is enough to assume that $p\geq 3$.
\end{remark}

We can now use the previous two lemmas to construct the map to construct such maps for any two core vertices.

\begin{lemma}
Under Assumptions \ref{ass:core_rank_positve}, \ref{ass:patched:basic} and \ref{ass:H2}, there is, for every $\kn_1,\kn_2\in \NN^{(0)}$, an isomorphism
\[\gamma_{\kn_1,\kn_2}:\Upsilon(\kn_1)\to \Upsilon(\kn_2)\]
such that, for every $\kappa\in \bKS(\bT,\FF)$ satisfying that $\kappa_{\kn_1}\in \Upsilon(\kn_1)$, then $\gamma_{\kn_1,\kn_2}(\kappa_{\kn_1})=\kappa_{\kn_2}$.
\label{lem:core_graph_connected}
\end{lemma}

\begin{proof}
We can assume without loss of generality that the quantity $s=\nu(\kn_2)-\nu(\kn_1)\geq 0$. A recursive application of Corollary \ref{cor:patched:dual_res_vanish_417} shows the existence of primes $\ku_1,\ldots,\ku_s$ such that for all $i$, the product $\kn\ku_1\cdots \ku_i\in \NN^{(0)}$. Let $\kr:=\kn\ku_1\cdots \ku_s$. By Lemma \ref{lemma:iwasawa:kol_gamma_n_nu}, we can construct by composition a map
\[\gamma_{\kn_1,\kr}:\ \Upsilon(\kn_1)\to \Upsilon(\kr),\]
such that every Kolyvagin system $\kappa\in \bKS(\bT,\FF)$ such that $\kappa_{\kn_1}\in \Upsilon(\kn_1)$ satisfies that $\gamma_{\kn_1,\kr}(\kappa_{\kn_1})=\kappa_{\kr}$.

For every prime $\kq_1\mid \kr$ not dividing $\kn_2$ and every prime $\kq_2\mid \kn_2$ and not dividing $\kr$, we can apply Lemma \ref{lemma:iwasawa:kol_gamma_tau} recursively to find the map
\[\gamma_{\kr,\kn_2}:\Upsilon(\kr)\to \Upsilon(\kn_2),\]
such that every Kolyvagin system $\kappa\in \bKS(\bT,\FF)$ such that $\kappa_{\kr}\in \Upsilon(\kr)$ satisfies that $\gamma_{\kr,\kn_1}(\kappa_{\kr})=\kappa_{\kn_2}$.
The isomorphism $\gamma_{\kn_1,\kn_2}$ is constructed as the composition of these maps:
\[\gamma_{\kn_1,\kn_2}=\gamma_{\kr,\kn_2}\circ \gamma_{\kn_1,\kr}.\]
\end{proof}

Since, for every core vertex $\kc\in \NN$, we have that 
\[\Upsilon(\kc)=\bigwedge^{\chi(\FF)} \bH^1_{\FF}(K,\bT),\]
the following corollary holds.

\begin{corollary}
Let $\kappa\in \bKS(\bT,\FF)$ be a Kolyvagin system. If $\kappa_{\kc}=0$ for some core vertex $\kc\in \NN$, then $\kappa_{\kn}=0$ for any other vertices $\kn\in \NN^{(0)}$. In particular, $\kappa_{\kn}=0$ for all other core vertices $\kn$.
\label{cor:kol_projection_injective_core_all}
\end{corollary}

\begin{remark}
If we apply Lemma \ref{lem:core_graph_connected} with $\kn_1$ being a core vertex, we obtain that 
\[\kappa_{\kn_2}\in \Upsilon(\kn_2)\ \forall \kn_2\in \NN^{(0)},\ \forall \kappa\in \bKS(\bT,\FF).\]
\end{remark}

\begin{lemma}
Let $\kappa\in \bKS(\bT,\FF)$ be a Kolyvagin system such that $\kappa_{\kc}=0$ for some core vertex $\kc\in \NN$. Then $\kappa=0$.
\label{lem:kol_projection_injective}
\end{lemma}

\begin{proof}
We will show that $\kappa_{\kn}=0$ for all $\kn\in \NN$. We will proceed by induction on the quantity
\[\lambda(\kn)=\rank_\Lambda\ \bH^1_{\FF^*}(K,\bT^*)^\vee.\]
When $\lambda(\kn)=0$, then we can deduce by Lemma \ref{lem:core_graph_connected} that $\kappa_\kn=0$.

Assume now that $\kappa_{\kr}=0$ for all $\kr\in \NN$ such that $\lambda(\kr)\leq k$ for some natural number $k$ and that $\kn\in \NN$ is a vertex satisfying that $\lambda(\kn)=k$ and that $\kappa_{\kn}\neq 0$.

By Corollary \ref{cor:patched:cheb_rank_red} and Proposition \ref{prop:bidual_map_iwasawa_kernel}, we can construct a Kolyvagin ultraprime such that 
\[\kappa_\kn\notin\ker\left(\bigwedge^{\chi(\FF)} \bH^1_{\FF(\kn)}(K,\bT)\to \bigwedge^{\chi(\FF)-1} \bH^1_{\FF_\tau(\kn)}(K,\bT)\right)\]
and $\lambda(\kn\tau)<\lambda(\kn)$. By the induction hypothesis, $\kappa_{\kn\tau}=0$, which contradicts the Kolyvagin relation between $\kappa_\kn$ and $\kappa_{\kn\tau}$. The contradiction proves that $\kappa_\kn=0$ for all $\kappa_n\in \NN$.
\end{proof}

We can now conclude the proof of Theorem \ref{th:kol_core_projection}.

\begin{proof}[Proof of Theorem \ref{th:kol_core_projection}]
The projection 
\[\bKS(\bT,\FF)\to \bigwedge^{\chi(\FF)} \bH^1_{\FF(\kn)}(K,\bT)\]
is surjective by Lemma \ref{lem:kol_core_projection_sur} and injective by Lemma \ref{lem:kol_projection_injective}.
\end{proof}

Hence we know that the module of Kolyvagin systems is a free, cyclic $\Lambda$-module that is controlled by the module of Stark systems.

\begin{corollary}
Let $\FF$ be a Selmer structure of core rank $\chi(\FF)\geq 1$. If $\bT$ satisfies Assumptions \ref{ass:patched:basic} and \ref{ass:H2}, the set of Kolyvagin systems is a free, cyclic $\Lambda$-module.
\label{cor:iwasawa:kol_mod}
\end{corollary}

\begin{corollary}
Under Assumptions \ref{ass:core_rank_positve}, \ref{ass:patched:basic} and \ref{ass:H2}, the regulator map 
\[\Reg:\ \bSS(\bT,\FF)\to \bKS(\bT,\FF)\]
is an isomorphism.
\label{cor:iwasawa:reg_iso}
\end{corollary}

It now makes sense to consider generators of the module of Kolyvagin systems.

\begin{definition}
We say that a Kolyvagin system $\kappa\in \bKS(\bT,\FF)$ is \emph{primitive} if it generates $\bKS(\bT,\FF)$ as a $\Lambda$-module.
\end{definition}

Analogously to Theorem \ref{th:stark:fitt_sel_res}, primitive Kolyvagin systems control the Fitting ideal with transverse local conditions.
\begin{theorem}
    Suppose Assumptions \ref{ass:patched:basic} and \ref{ass:H2} hold and that $\FF$ is a cartesian Selmer structure of core rank $\chi(\FF)\geq 1$. If $p\geq 5$ and $\kappa\in \bKS(\bT,\FF)$ is a primitive Kolyvagin system, then
    \[\Fitt^0_\Lambda\Bigl(\bH^1_{(\FF^*)(\kn)}(K,\bT^*)^\vee\Bigr)\subset_{f.i.}\ind(\kappa_\kn) \ \forall\, \kn\in \NN.\]
    \label{th:kol_fitting0}
\end{theorem}

\begin{proof}
By Corollary \ref{cor:iwasawa:reg_iso}, there is a primitive Stark system $\varepsilon$ such that $\Reg(\varepsilon)=\kappa$. For every $\kn\in \NN$, Proposition \ref{prop:weak_core_vertex_existence} shows the existence of a weak core vertex $\kc\in \NN$ such that $\kn\mid \kc$. Then Proposition \ref{prop:patched:global_duality} gives an exact sequence

\begin{center}
\begin{tikzpicture}[descr/.style={fill=white,inner sep=1.5pt}]
        \matrix (m) [
            matrix of math nodes,
            row sep=3.5em,
            column sep=1.3em,
            text height=1.5ex, text depth=0.25ex
        ]
        {  0&\bH^1_{\FF(\kn)}(K,\bT) & \bH^1_{\FF^\kc}(K,\bT) \\
        &\displaystyle{\bigoplus_{\ku\mid \kn} \bH^1_{\f}(K_\ku,\bT)\oplus \bigoplus_{\ku\mid \frac{\kc}{\kn}} \bH^1_{\s}(K_\ku,\bT)}
            & \bH^1_{(\FF^*)(\kn)}(K,\bT^*)^\vee&0.  \\
        };

        \path[overlay,->, font=\scriptsize,>=latex]
        (m-1-1) edge (m-1-2)
        (m-1-2) edge (m-1-3)
         
        (m-1-3) edge [out=355,in=175] (m-2-2) 
        (m-2-2) edge (m-2-3)
        (m-2-3) edge (m-2-4)
       
        ;    
\end{tikzpicture}
\end{center}
Then $\kappa_\kn=(\Reg_\kn\circ \phi_{\kc,\kn})(\varepsilon_{\kc})$. Since $\varepsilon$ is a primitive Stark system, Theorem \ref{th:stark_core_projection} shows that $\varepsilon_\kc$ generates $\bigwedge^{\chi(\FF)
+\nu(\kc)}\bH^1_{\FF^c}(K,\bT)$. Hence Propositions \ref{prop:bidual_map_fitting} and \ref{prop:bidual_map_iwasawa_fitting_inclusion} imply that 
\[\Fitt^0_\Lambda\Bigl(\bH^1_{\FF^*(\kn)}(K,\bT^*)^\vee\Bigr)\subset_{f.i.}\ind(\kappa_n).\qedhere\]
\end{proof}

\subsection{Fitting ideals of the Selmer group}

In this section, we give a description of the Fitting ideals of the Selmer group, by relating them with the theta ideals of a primitive Kolyvagin system.

\begin{theorem}
Let $\kappa\in \bKS(\bT,\FF)$ be a primitive Kolyvagin system and suppose that Assumptions \ref{ass:core_rank_positve}, \ref{ass:patched:basic} and \ref{ass:H2} hold. Then 
\[\Fitt^i_\Lambda\Bigl(\bH^1_{\FF^*}(K,\bT^*)^\vee\Bigr)\subset_{f.i.}\bTheta(\kappa).\]
\label{th:iwasawa:kolyvagin_theta_fitting}
\end{theorem}

\begin{proof}
By Corollary \ref{cor:patched:dual_res_vanish_417}, we can choose a weak core vertex $\kc\in \NN$ such that, for every $\ku\mid \kc$, the following equality holds
\begin{equation}
\bH^1_{(\FF^*)(\ku)}(K,\bT^*)=\bH^1_{(\FF^*)_\ku}(K,\bT^*).
\label{eq:kol_fit_tr_res}
\end{equation}

The core vertex $\kc$ induces an exact sequence
\[\xymatrix{\bH^1_{\FF}(K,\bT) \ar@{>->}[r]& \bH^1_{\FF^{\kc}}(K,\bT)\ar[r] & \displaystyle{\bigoplus_{\ku\mid \kc}\bH^1_\s(K_\ku,\bT)} \ar@{->>}[r] & \bH^1_{\FF^*}(K,\bT^*)^\vee.}\]

For every $\kq\mid \kc$, there is another exact sequence 
\[\xymatrix{\bH^1_{\FF^\kq}(K,\bT)\ar@{>->}[r] & \bH^1_{\FF^{\kc}}(K,\bT)\ar[r] & \displaystyle{\bigoplus_{\ku\mid \frac{\kn}{\kq}}\bH^1_s(K_\ku,\bT)} \ar@{->>}[r] & \bH^1_{(\FF^*)_\kq}(k,\bT^*).}\]

By the definition of the Fitting ideals, we can deduce that 
\[\Fitt^i_\Lambda\Bigl(\bH^1_{\FF^*}(K,\bT^*)^\vee\Bigr)=\sum_{\kq\mid \kc} \Fitt^{i-1}_\Lambda\Bigl(\bH^1_{(\FF^*)_\kq}(K,\bT^*)^\vee\Bigr).\]
By \eqref{eq:kol_fit_tr_res}, we can deduce that
\[\Fitt^i_\Lambda\Bigl(\bH^1_{\FF^*}(K,\bT^*)^\vee\Bigr)=\sum_{\kq\mid \kc} \Fitt^{i-1}_\Lambda\Bigl(\bH^1_{(\FF^*)(\kq)}(K,\bT^*)^\vee\Bigr).\] 

By Proposition \ref{prop:core_vertex_existence}, we can extend the equality to all the ultraprimes in $\PP$:
\[\Fitt^i_\Lambda\Bigl(\bH^1_{\FF^*}(K,\bT^*)^\vee\Bigr)=\sum_{\kq\in \PP} \Fitt^{i-1}_\Lambda\Bigl(\bH^1_{(\FF^*)(\kq)}(K,\bT^*)^\vee\Bigr).\] 

We can now proceed to prove Theorem \ref{th:iwasawa:kolyvagin_theta_fitting} by induction. For $i=0$, it is Theorem \ref{th:kol_fitting0} for $n=1$. Assuming the result holds true for $i-1$, for all Selmer structures $\FF(\kq)$, we have that 
\[\Fitt^i_\Lambda\Bigl(\bH^1_{\FF^*}(K,\bT^*)^\vee\Bigr)=\sum_{\kq\in \PP} \sum_{\kn\in \NN_{i-1}, \kq\nmid \kn} \Fitt^0_\Lambda\Bigl(\bH^1_{(\FF^*)({\kq\kn})}(K,\bT^*)^\vee\Bigr).\]
Rearranging the sum,
\[\Fitt^i_\Lambda\Bigl(\bH^1_{\FF^*}(K,\bT^*)^\vee\Bigr)=\sum_{\kn\in \NN_{i}} \Fitt^0_\Lambda\Bigl(\bH^1_{(\FF^*)({\kn})}(K,\bT^*)^\vee\Bigr).\]
By Theorem \ref{th:kol_fitting0}, we obtain that
\[\Fitt^i_\Lambda\Bigl(\bH^1_{\FF^*}(K,\bT^*)^\vee\Bigr)\subset_{f.i.} \sum_{\kn\in \NN_{i}}\ind(\kappa_\kn)=\bTheta(\kappa).\qedhere\]
\end{proof}

\section{Relation with the classical theory}
\label{sec:classical}

This section compares the ultra Kolyvagin systems constructed in \textsection{\ref{sec:iwasawa:kol}} with the previous construction in \cite{BurnsSakamotoSano2}.

\subsection{Ultra Kolyvagin systems with finite coefficients}

It is possible to use patched cohomology and ultraprimes to build ultra Kolyvagin systems for a Selmer group with finite coefficient ring. However, no technical advantage appears a result of this construction, but it will be necessary to compare Definition \ref{def:kol_highrank} with \cite[Definition 5.1]{BurnsSakamotoSano2}

When working with finite coefficients, we are forced to work with exterior biduals instead of exterior powers. The reason is that Selmer groups are no longer necessarily free, so we need to use a rank reduction map analogue to Proposition \ref{prop:bidual_map}, instead of Corollary \ref{cor:exterior_rankred} (see \cite[Lemma 2.1]{Sakamoto18} for a construction of this map for modules over a self-injective ring).

\begin{namedass}{(SI)}
    Let $R$ be a self-injective, local, artinian ring with finite residue field and let $T$ be an $R[[G_K]]$-module which is free and finitely generated as an $R$-module, ramifies only at finitely many primes, and satisfies Assumption \ref{ass:patched:basic}. 
    \label{ass:selfinj}
\end{namedass}

The reader is referred to \cite[Definition 5.1]{BurnsSakamotoSano2} for the classical construction of Kolyvagin systems for Selmer structures defined over $T$. The module of such Kolyvagin systems will be denoted by $\KS(T,\FF)$. An identical construction will lead to the module of ultra Kolyvagin systems $\bKS(T,\FF)$, but considering patched Selmer groups and collection of classes indexed square-free product by Kolyvagin ultraprimes. 

The proofs in \cite{BurnsSakamotoSano2} also work for studying the module of ultra Kolyvagin systems. In particular, the proof of \cite[Theorem 5.2]{BurnsSakamotoSano2} implies that $\bKS(T,\FF)$ is a free $R$-module of rank one whenever $\FF$ is a cartesian Selmer structure of positive core rank.

The next theorem shows that both constructions lead to the same module.

\begin{theorem}
Let $R$ and $T$ be as in Assumption \ref{ass:selfinj} and let $\FF$ be a cartesian Selmer structure defined on $T$ in the sense of \cite[Definition 3.9]{Sakamoto18}. Let $\NN_{\cl}$ denote the set of square-free products of Kolyvagin primes, while $\NN_u$ denotes the set of square-free products of Kolyvagin ultraprimes. Then there is an isomorphism
\[\bKS(T,\FF)\to \KS(T,\FF):\ \{\kappa_\kn\}_{\kn\in \NN_u}\mapsto \{\kappa_\kn\}_{\kn\in \NN_\cl},\]
defined by restricting via the inclusion $\NN_{\cl}\hookrightarrow \NN_u$ as constant ultraprimes.

The inverse map is is defined by the ultra Kolyvagin system constructed as follows: every $\kn\in \NN_\ku$ represented by a sequence $(n_i)_{i\in \N}$, the class $\kappa_{\kn}$ is constructed by patching the classes $\kappa_{n_i}$ via the identification
\[\bigcap^{\chi(\FF)} \bH^1_{\FF(\kn)}(K,T)=\UU_i\left(\bigcap^{\chi(\FF)} \bH^1_{\FF(n_i)}(K,T)\right).\]
\label{th:comp_ultra_kol}
\end{theorem}

\begin{remark}
Note that, since $R$ is finite, there is a canonical inclusion\footnote{The inclusion is not an equality since sequences in which the number of prime divisors is not bounded might not represent a (finite) product of Kolyvagin ultraprimes.} $\NN_u\subset \UU(\NN_\cl)$. Under this identification, $\NN_u$ contains all finite square-free products of constant Kolyvagin ultraprimes, so we can identify $\NN_\cl\hookrightarrow \NN_u$.
\label{rem:kol_primes}
\end{remark}

\begin{proof}[Proof of Theorem \ref{th:comp_ultra_kol}]
It is straightforward that if we have an ultra Kolyvagin system $\kappa=\{\kappa_\kn\}_{\kn\in \NN_u}$ in $\bKS(T_k,\FF)$, the subset of $\kappa_\kn$ when $\kn\in \NN_\cl$ is a constant ultraprime forms a classical Kolyvagin system in $\KS(T_k,\FF)$.

We can now outline the construction of the inverse map. Assume that $\kappa=\{\kappa_n\}_{n\in \NN}\in \KS(T_k,\FF)$ is a classical Kolyvagin system and let $\kn\in \NN_\ku$ be a square-free product of Kolyvagin ultraprimes. By Remark \ref{rem:kol_primes}, we can view $\kn$ as an element in $\UU(\NN_\cl)$, so it can be represented by a sequence $(n_i)_{i\in \N}$, where $n_i\in \NN_\cl$. Since the Selmer groups are finite, we can use Corollary \ref{cor:ultrafilter_finite_modules} to conclude that
\[\bigcap^{\chi(\FF)} \bH^1_{\FF(\kn)}(K,T)=\UU_i\left(\bigcap^{\chi(\FF)} \bH^1_{\FF(n_i)}(K,T)\right).\]
Under this identification, the sequence $(\kappa_{n_i})_{i\in \N}$ represents an element in the exterior bidual $\bigcap^{\chi(\FF)} \bH^1_{\FF(\kn)}(K,T_k)$, that will be defined as $\kappa_{\kn}$. Note that the construction of $\kappa_{\kn}$ is independent of the choice of the representative sequence.

For every ultraprime $\ku=(\ell_i)_{i\in \N}$ dividing $\kn$, the sequences $(\phi_{\ell_i})_{i\in \N}$ and  $(\psi_{\ell_i})_{i\in \N}$ represent, by Proposition \ref{prop:ultraproduct_hom}, elements 
\[\begin{array}{cc}
\phi_\ku\in \bH^1_{\ku}(K_\ku,T)^+,\ &\psi_\ku\in\bH^1_{\tr}(K_\ku,T)^+,
\end{array}\]
respectively, satisfying that $\phi_\ku=\phi_\ku^\fs\circ\psi_\ku$. We will use those to compute the Kolyvagin system relations. 

Since the rank reduction maps commute with the ultraproduct, we have that 
\[\begin{array}{cc}
    \alpha_{\kn,\ku}(\kappa_\kn)=(\alpha_{n_i,\ell_i}(\kappa_{n_i}))_{i\in \N},\ & \beta_{\kn,\ku}(\kappa_{\kn\ku})=(\beta_{n_i,\ell_i}(\kappa_{n_i\ell_i}))_{i\in \N},
\end{array}\]
where, again, we are using the following identification:
\[\bigcap^{\chi(\FF)-1} \bH^1_{\FF_\ku(\kn)}(K,T)=\UU_i\left(\bigcap^{\chi(\FF)-1} \bH^1_{\FF_{\ell_i}(n_i)}(K,T)\right).\]
Since $\kappa_{n_i}$ and $\kappa_{n_i\ell_i}$ satisfy the Kolyvagin system relation $\alpha_{n_i,\ell_i}(\kappa_{n_i})=\beta_{n_i,\ell_i}(\kappa_{n_i\ell_i})$ for all $i\in \N$, then 
\[\alpha_{\kn,\ku}(\kappa_\ku)=(\beta_{\kn,\ku}(\kappa_{\kn\ku})).\]
Then this construction gives a well defined map
\[\KS(T,\FF)\to \bKS(T,\FF).\]
It is a direct computation to check that the composition
\[\KS(T,\FF)\to\bKS(T,\FF)\to \KS(T,\FF)\]
is the identity map. Indeed, when $\kn$ is equivalent to a constant sequence $(n)$, then, by construction, of $\kappa_{\kn}$ is the patching of a sequence $(\kappa_{n_i})_{i\in \N}$, where $n_i=n$ for $\UU$-many $i$. Then $\kappa_{\kn}=\kappa_n$.

Since both $\KS(T,\FF)$ and $\bKS(T,\FF)$ are free $R_k$-modules of rank one, the above maps are, in fact, isomorphisms.
\end{proof}

\subsection{Assumptions}
\label{sec:cl:ass}

Recall that we established in Notation \ref{not:iw_limit} that 
\[\begin{array}{cc}
\Lambda_k:=\Lambda/(p^k,X^k),\ & \bT_k:=\bT\otimes_\Lambda \Lambda_k.
\end{array}\]

The goal of this section is to establish conditions on a Selmer structure defined on $\bT$ that ensure that the propagated Selmer structures on $\bT_k$ are cartesian in the sense of \cite[Definition 3.9]{Sakamoto18}, so the module of ultra Kolyvagin systems can be studied using Theorem \ref{th:comp_ultra_kol}.

\begin{namedass}{(Prop)}
Assume that $p\geq 5$ and let $\bT$ be a Galois representation with coefficients in the Iwasawa algebra satisfying Assumptions \ref{ass:patched:basic} and \ref{ass:iwasawa:local_no_fin_sub} and let $\FF$ be a cartesian Selmer structure defined on $\bT$ such that its propagation to $\bT/(X)$ is also cartesian.
\label{ass:classic}
\end{namedass}

Assumption \ref{ass:classic} can be reinterpreted in terms of the local conditions.
\begin{proposition}
Assume that $\bT$ satisfies Assumptions \ref{ass:patched:basic} and \ref{ass:iwasawa:local_no_fin_sub}. Then Assumption \ref{ass:classic} is equivalent to $\bH^1_{/\FF}(K_\ku,\bT)$ being a free Iwasawa module for every $\ku\in \Sigma_{\FF}$.
\label{prop:iwasawa:free_cartesian_propagated}
\end{proposition}

\begin{proof}
Suppose first that Assumption \ref{ass:classic} holds. In particular, $\FF$ is cartesian, so $\bH^1_{/\FF}(K_\ku,\bT)$ is pseudo-isomorphic to a free module $F$ for every $\ku\in \Sigma_{\FF}$. Since it does not contain any finite submodules, there is a finite module $C$ fitting in the exact sequence 
\[\xymatrix{0\ar[r] & \bH^1_{/\FF}(K_\ku,\bT)\ar[r] & F\ar[r] & C\ar[r] & 0.}\]
For the sake of contradiction, assume that $C$ is not the trivial module. The snake lemma then induces another exact sequence
\[\xymatrix{0\ar[r] &C[X]\ar[r] &  \bH^1_{/\FF}(K_\ku,\bT)\otimes \Lambda/(X)\ar[r] & F/XF\ar[r] & C/XC\ar[r] & 0,}\]
which implies that $\bH^1_{/\FF}(K_\ku,\bT)\otimes \Lambda/(X)$ contains a finite submodule.

However, Proposition \ref{prop:patched:long_profinite} and diagram chasing induce another exact sequence
\begin{equation}
\xymatrix{0\ar[r] &\bH^1_{/\FF}(K_\ku,\bT)\otimes \Lambda/(X)\ar[r] & \bH^1_{/\FF}(K_\ku,\bT/X\bT)\ar[r] & \bH^2(K_\ku,\bT)[X].}
\label{eq:prop_ses}
\end{equation}
When the propagated Selmer structure is cartesian, $\bH^1_{/\FF}(K_\ku,\bT/X\bT)$ contains no finite Selmer modules, so neither does $\bH^1_{/\FF}(K_\ku,\bT)\otimes \Lambda/(X)$. Therefore, $C=0$ and $\bH^1_{/\FF}(K_\ku,\bT)$ is a free Iwasawa module.

Conversely, it is clear that the freeness $\bH^1_{/\FF}(K_\ku,\bT)$ for every $\ku\in \Sigma_{\FF}$ is stronger than the cartesian condition on $\FF$, so we are only left to show that the propagation of $\FF$ to $\bT/X\bT$ is also cartesian. Remark \ref{rem:iwasawa:local_duality_finsub} implies that, under Assumption \ref{ass:iwasawa:local_no_fin_sub}, $\bH^2(K_\ku,\bT)[X]$ contains no finite submodules. If $\bH^1_{/\FF}(K_\ku,\bT)$ is a free module, then the snake lemma applied to the short exact sequence in \eqref{eq:prop_ses} implies that $\bH^1_{/\FF}(K_\ku,\bT/X\bT)$ contains no finite submodules either, so the propagated Selmer structure is cartesian on $\bT/X\bT$.
\end{proof}

Assumption \ref{ass:classic} is enough to conclude that the propagated Selmer structures are cartesian on $\bT_k$.
\begin{proposition}
Let $\bT$ and $\FF$ be as in Assumption \ref{ass:classic}. Then, for all $k\in \N$, the propagation of $\FF$ to $\bT_k$ is cartesian, in the sense of \cite[Definition 3.9]{Sakamoto18}.
\label{prop:iwasawa:finite_cartesian}
\end{proposition}

\begin{proof}
We need to show, for every $\ku\in \Sigma_{\FF}$, that the composition

\begin{center}
\begin{tikzpicture}[descr/.style={fill=white,inner sep=1.5pt}]
        \matrix (m) [
            matrix of math nodes,
            row sep=3.5em,
            column sep=3.5em,
            text height=1.5ex, text depth=0.25ex
        ]
        {  \bH^1_{/\FF}(K_\ku,\bT/\m\bT) &\bH^1_{/\FF}(K_\ku,\bT/(p^k,X)\bT) &\bH^1_{/\FF}(K_\ku,\bT_k)\\
        };
        \path[overlay,->, font=\scriptsize,>=latex]
        (m-1-1) edge node[midway,above]{$[p^{k-1}]$}(m-1-2)
        (m-1-2) edge node[midway,above]{$[X^{k-1}]$} (m-1-3)

        ;    
\end{tikzpicture}
\end{center}
is injective. The first map $[p^{k-1}]$ is known to be injective by the argument in \cite[Lemma 3.7.1]{MazurRubin}, since the propagated Selmer structure on $\bT/X\bT$ is cartesian and $\Lambda/(X)$ is a discrete valuation ring. For the injectivity of $[X^{k-1}]$, consider the following commutative diagram with exact rows:
\begin{center}
\begin{tikzpicture}[descr/.style={fill=white,inner sep=1.5pt}]
        \matrix (m) [
            matrix of math nodes,
            row sep=3.5em,
            column sep=1.2em,
            text height=1.5ex, text depth=0.25ex
        ]
        { \bH^1\left(K_\ku,\frac{\bT}{p^k\bT}\right) & \bH^1\left(K_\ku,\frac{\bT}{p^k\bT}\right) & \bH^1\left(K_\ku,\frac{\bT}{(p^k,X)\bT}\right) & \bH^2\left(K_\ku,\frac{\bT}{p^k\bT}\right)[X]  \\ 
         \bH^1\left(K_\ku,\frac{\bT}{p^k\bT}\right) &\bH^1\left(K_\ku,\frac{\bT}{p^k\bT}\right) & \bH^1\left(K_\ku,\frac{\bT}{(p^k,X^k)\bT}\right) & \bH^2\left(K_\ku,\frac{\bT}{p^k\bT}\right)[X^k].  \\
        };
        \path[overlay,->, font=\scriptsize,>=latex]
        (m-1-2) edge node[midway,above]{$\Pi_1$}(m-1-3)
        (m-1-3) edge node[midway,above]{$\delta_1$} (m-1-4)
        (m-2-2) edge node[midway,above]{$\Pi_k$} (m-2-3) 
        (m-2-3) edge node[midway,above]{$\delta_k$} (m-2-4)
        (m-1-2) edge node[midway,right]{$X^{k-1}$} (m-2-2)
        (m-1-3) edge node[midway,right]{$[X^{k-1}]$} (m-2-3)
        (m-1-4) edge node[midway,right]{$\subset$} (m-2-4)
        (m-1-1) edge node[midway,right]{} (m-2-1)
        (m-1-1) edge node[midway,above]{$X$} (m-1-2)
        (m-2-1) edge node[midway,above]{$X^k$} (m-2-2)
        ;    
\end{tikzpicture}
\end{center}
Let 
\[x\in [X^{k-1}]^{-1}\biggl(\bH^1_{\FF}(K_\ku,\bT/(p^k,X^k)\bT)\biggr).\] 
By Definition \ref{def:propagation_quotients}, there exists $\gamma\in \bH^1_{\FF}(K_\ku,\bT/p^k\bT)$ such that $\Pi_k(\gamma)=[X^{k-1}](x)$. The exactness of the last row implies that $[X^{k-1}](x)\in \ker(\delta_k)$, so the injectivity of the rightmost vertical map implies that $x\in \ker(\delta_1)=\Im(\Pi_1)$.

Let $\alpha\in \bH^1(K_\ku,\bT/p^k\bT)$ be such that $\Pi_1(\alpha)=x$. The exactness of the bottom row implies that 
\[X^{k} \bH^1(K_\ku,\bT/p^k\bT)=\ker(\Pi_k),\] 
so we can find an element $\beta\in \bH^1(K_\ku,\bT/p^k\bT)$ such that 
\[\gamma=X^{k-1}\alpha+X^k\beta=X^{k-1} (\alpha+X \beta)\in \bH^1_{\FF}(K_\ku,\bT/p^k\bT).\]

We claim that $\bH^1_{/\FF}(K_\ku,\bT/p^k\bT)$ contains no $X$-torsion. Indeed, by diagram chasing, there is an exact sequence
\[\xymatrix{0\ar[r] & \bH^1_{/\FF}(K_\ku,\bT)\otimes \Lambda/(p^k) \ar[r]& \bH^1_{/\FF}(K_\ku,\bT/p^k\bT) \ar[r] & \bH^2(K_\ku,\bT)[p^k]\ar[r] &0.}\]
By Proposition \ref{prop:iwasawa:free_cartesian_propagated}, $\bH^1_{/\FF}(K_\ku,\bT)$ is a free module. Therefore, $\bH^1_{/\FF}(K_\ku,\bT)\otimes \Lambda/(p^k)$ contains no $X$-torsion. In addition $\bH^2(K_\ku,\bT)[p^k]$ cannot contain $X$-torsion either since that would contradict Assumption \ref{ass:iwasawa:local_no_fin_sub} (see also Remark \ref{rem:iwasawa:local_duality_finsub}). Hence $\bH^1_{/\FF}(K_\ku,\bT/p^k\bT)$ cannot contain $X$ torsion either.

Since $\bH^1_{/\FF}(K_\ku,\bT/p^k\bT)$ contains no $X$-torsion, then 
\[\alpha+X\beta \in \bH^1_{\FF}(K_\ku,\bT/p^k\bT).\]
By the exactness of the top row, $X \bH^1(K_\ku,\bT/p^k\bT)=\ker(\Pi_1)$, so we have that 
\[x=\Pi_1(\alpha)=\Pi_1(\alpha+X\beta)\in \Pi_1\Bigl(\bH^1_{\FF}(K_\ku,\bT/p^k\bT)\Bigr)=\bH^1_{\FF}(K_\ku,\bT/(p^k,X)\bT),\]
which concludes the proof of the injectivity of $[X^{k-1}]$.
\end{proof}

Next we show that the core rank of $\FF$ remains invariant after propagating to $\bT_k$. Following \cite{MazurRubin}, that is equivalent to the following proposition.

\begin{proposition}
Let $\bT$ and $\FF$ be as in Assumption \ref{ass:classic}. Then
\[\dim_{\F_p} \bH^1_{\FF}(K,\bT/\m\bT)-\dim_{\F_p} \bH^1_{\FF^*}(K,\bT^*[\m])=\chi(\FF).\]
\label{prop:core_inv}
\end{proposition}

\begin{proof}
    By the argument in \cite[Corollary 5.2.6]{MazurRubin}, since $\FF$ is cartesian in $\bT/X\bT$ and $\Lambda/(X)$ is a discrete valuation ring, there is an equality
    \[\begin{aligned}
    &\dim_{\F_p} \bH^1_{\FF}(K,\bT/\m\bT)-\dim_{\F_p} \bH^1_{\FF^*}(K,\bT^*[\m])=\\&\rank_{\Z_p} \bH^1_{\FF}(K,\bT/X\bT)-\rank_{\Z_p} \bH^1_{\FF^*}(K,\bT^*[X])^\vee.
    \end{aligned}\]

    Applying Corollary \ref{cor:patched:global_long_exact} to the short exact sequence
    \[\xymatrix{0\ar[r] & \bT \ar[r]^{X} & \bT \ar[r] & \bT/X\bT \ar[r] & 0,}\]
    we obtain an injection
    \[\bH^1_{\FF}(K,\bT)\otimes \Lambda/(X)\hookrightarrow\bH^1_{\FF}(K,\bT/X\bT).\]
    Then we can define $C$ as the cokernel of this map. Since $\bH^1_{\FF}(K,\bT)$ is pseudo-isomorphic to a free module, we have that
    \[\rank_\Lambda\bH^1_{\FF}(K,\bT)=\rank_{\Z_p}\biggl(\bH^1_{\FF}(K,\bT)\otimes \Lambda/(X)\biggr),\]
    so the additivity of the rank implies that
    \begin{equation}
    \rank_{\Z_p} C=\rank_{\Z_p}\bH^1_{\FF}(K,\bT/X\bT)-\rank_\Lambda \bH^1_{\FF}(K,\bT)
    \label{eq:rk_c}
    \end{equation}
    This proof will be completed with the computation of the rank of $C$. By Corollary \ref{cor:patched:dual_res_vanish}, we can find a square-free product of ultraprimes $\kn$, divisible by all the primes in $\Sigma_{\FF}$, such that $\bH^1_{\FF^*_\kn}(K,\bT^*)=0$. Define $I$ as the image
    \[I:=\Im\left(\bH^1(K^\kn/K,\bT)\to \bigoplus_{\ku\mid \kn} \bH^1_{/\FF}(K_\ku,\bT)\right).\]
    Then there is a short exact sequence
    \[\xymatrix{0\ar[r] & \bH^1_{\FF}(K,\bT)\ar[r] & \bH^1(K^{\kn}/K,\bT) \ar[r] & I\ar[r] &0.}\]
    Since $\FF$ is cartesian, $I$ is a torsion-free module, so the snake lemma induces another exact sequence
    \[\xymatrix{0\ar[r] & \bH^1_{\FF}(K,\bT)\otimes \Lambda/(X) \ar[r] &\bH^1(K^{\kn}/K,\bT)\otimes \Lambda/(X) \ar[r] & I/XI \ar[r] &0.}\]
    Similarly, define $J$ as the image
    \[J:=\Im\left(\bH^1(K^\kn/K,\bT/X\bT)\to \bigoplus_{\ku\mid \kn} \bH^1_{/\FF}(K_\ku,\bT/X\bT)\right).\]
    Then there is a commutative diagram with exact rows:
    \begin{center}
\begin{tikzpicture}[descr/.style={fill=white,inner sep=1.5pt}]
        \matrix (m) [
            matrix of math nodes,
            row sep=3.5em,
            column sep=1.5em,
            text height=1.5ex, text depth=0.25ex
        ]
        { 0 &  \bH^1_{\FF}(K,\bT)\otimes \Lambda/(X) & \bH^1(K^{\kn}/K,\bT)\otimes\Lambda/(X) & I/XI &0 \\
        0 & \bH^1_\FF(K,\bT/X\bT) & \bH^1(K^\kn/K,\bT/X\bT) &  J & 0. \\
        };
        \path[overlay,->, font=\scriptsize,>=latex]
        (m-1-1) edge (m-1-2)
        (m-1-2) edge (m-1-3)
        (m-1-3) edge (m-1-4)
        (m-1-4) edge (m-1-5)
        (m-2-1) edge (m-2-2) 
        (m-2-2) edge (m-2-3)
        (m-2-3) edge (m-2-4)
        (m-2-4) edge (m-2-5)
        (m-1-2) edge (m-2-2)
        (m-1-3) edge (m-2-3)
        (m-1-4) edge (m-2-4)
        
        ;    
\end{tikzpicture}
\end{center}
Note that Corollary \ref{cor:patched:global_long_exact}, applied to the short exact sequence
\begin{equation}
\xymatrix{0\ar[r] & \bT\ar[r]^{X} & \bT\ar[r] &\bT/X\bT\ar[r] & 0,}
\label{eq:ses_quot_x}
\end{equation}
 proves that the middle vertical map is injective with cokernel being isomorphic to $\bH^2(K^{\kn}/K,\bT)[X]$. If we define $A$ and $B$ as the kernel and cokernel fitting in the exact sequence
\[\xymatrix{0\ar[r] & A\ar[r] & I/XI\ar[r] & J\ar[r] & B\ar[r] &0.}\]
The snake lemma induces another exact sequence:
\[\xymatrix{0\ar[r] &A \ar[r] & C\ar[r] & \bH^2(K^{\kn}/K,\bT)[X] \ar[r] & B \ar[r] & 0.}\]
The additivity of the rank implies that
\[\rank_{\Z_p} C=\rank_{\Z_p}\bH^2(K^{\kn}/K,\bT)[X]+\rank_{\Z_p} A-\rank_{\Z_p} B.\]
Our choice of $\kn$ and the combination of Propositions \ref{prop:patched:global_duality} and \ref{prop:patched:selmer_torsion} produce a commutative diagram with exact rows and columns
\    \begin{center}
\begin{tikzpicture}[descr/.style={fill=white,inner sep=1.5pt}]
        \matrix (m) [
            matrix of math nodes,
            row sep=3em,
            column sep=0.9em,
            text height=1.5ex, text depth=0.25ex
        ]
        {  & A & 0 &0 &\\
            &  I/XI & \displaystyle{\bigoplus_{\ku\mid \kn} \bH^1_{/\FF}(K_\ku,\bT)\otimes \Lambda/(X)} & \bH^1_{\FF^*}(K,\bT^*)^\vee\otimes \Lambda/(X) &0 \\
        0 & J & \displaystyle{\bigoplus_{\ku\mid \kn} \bH^1_{/\FF}(K_\ku,\bT/X\bT)} & \bH^1_{\FF^*}(K,\bT^*[X])^\vee  & 0 \\
         & B & \displaystyle{\bigoplus_{\ku\mid \kn} \bH^2(K_\ku,\bT)[X]} & 0. &\\
        };
        \path[overlay,->, font=\scriptsize,>=latex]
        (m-2-2) edge (m-2-3)
        (m-2-3) edge (m-2-4)
        (m-2-4) edge (m-2-5)
        (m-3-1) edge (m-3-2)
        (m-3-2) edge (m-3-3)
        (m-3-3) edge (m-3-4)
        (m-3-4) edge (m-3-5)

        (m-1-2) edge (m-2-2)
        (m-1-3) edge (m-2-3)
        (m-1-4) edge (m-2-4)
        (m-2-2) edge (m-3-2)
        (m-2-3) edge (m-3-3)
        (m-2-4) edge (m-3-4)
         (m-3-2) edge (m-4-2)
        (m-3-3) edge (m-4-3)
        (m-3-4) edge (m-4-4)
        
        ;    
\end{tikzpicture}
\end{center}
Note that the exactness of the middle column follows from Corollary \ref{cor:patched:global_long_exact} applied to the short exact sequence in \eqref{eq:ses_quot_x} and the exactness of the rightmost column follows from Proposition \ref{prop:patched:selmer_torsion}. Combining the snake lemma with Proposition \ref{prop:patched:poitou-tate}, we can construct isomorphisms
\[B\cong \bigoplus_{\ku\mid \kn} \bH^2(K_\ku,\bT)[X] \cong \bH^2(K^{\kn}/K,\bT)[X],\]
where the last isomorphism holds because $\bH^1_{\FF^*_\kn}(K,\bT^*)=0$. Hence the ranks of $A$ and $C$ coincide. A double application of the snake lemma shows that 
\[A=\ker\left(I/XI\to \bigoplus_{\ku\mid \kn} \bH^1_{/\FF}(K_\ku,\bT)\otimes \Lambda/(X)\right)=\bH^1_{\FF^*}(K,\bT^*)^\vee[X].\]
By the structure theorem of finitely generated Iwasawa modules, we have that 
\[\rank_{\Z_p}\bH^1_{\FF^*}(K,\bT^*)^\vee[X]=\rank_{\Z_p}\biggl(\bH^1_{\FF^*}(K,\bT^*)^\vee\otimes \Lambda/(X)\biggr)-\rank_\Lambda\bH^1_{\FF^*}(K,\bT^*)^\vee.\]
By Proposition \ref{prop:patched:selmer_torsion} and \eqref{eq:rk_c}, we can then conclude that 
\[\begin{aligned}
    &\chi(\FF):=\rank_\Lambda \bH^1_{\FF}(K,\bT)-\rank_\Lambda\bH^1_{\FF^*}(K,\bT^*)^\vee\\
&\rank_{\Z_p}\bH^1_{\FF}(K,\bT/X\bT)-\rank_{\Z_p}\bH^1_{\FF^*}(K,\bT^*[X])^\vee=\\
&\dim_{\F_p} \bH^1_{\FF}(K,\bT/\m\bT)-\dim_{\F_p} \bH^1_{\FF^*}(K,\bT^*[\m]).
\end{aligned}\qedhere\]
\end{proof}

\subsection{Comparison theorem}

The goal of this section is to generalise the idea in Theorem \ref{th:comp_ultra_kol} to express the ultra Kolyvagin systems from Definition \ref{def:kol_highrank} as a limit of classical Kolyvagin systems.

Since, by Proposition \ref{prop:selmer_profinite}, an Iwasawa Selmer group can be recovered as the limit
\[\bH^1_{\FF}(K,\bT)=\varprojlim_{k\in \N} \bH^1_{\FF}(K,\bT_k),\]
where, in the finite Selmer groups, $\FF$ represents the propagated structure. The first part of the argument will be to generalise this argument to the exterior biduals of Selmer groups.

First, we need to say what are the transition maps used to define the inverse limit. The functoriality of Selmer groups and exterior biduals produces, for every $k\in \N$, an homomorphism
\begin{equation}
\bigcap_{\Lambda_{k+1}}^{\chi(\FF)} \bH^1_{\FF}(K,\bT_{k+1})\to \bigcap_{\Lambda_{k+1}}^{\chi(\FF)} \bH^1_{\FF}(K,\bT_k).
\label{eq:comp_trans_pre}
\end{equation}

The annihilator $\Lambda_{k+1}[(p^k,X^k)]$ is principal, since it is generated by $pX$. Then there is an isomorphism $\Lambda_k\cong \Lambda_{k+1}[(p^k,X^k)]$. Then, for every $\Lambda_k$-module $M_k$,
\[\Hom(M_k,\Lambda_{k+1})=\Hom\Bigl(M_k,\Lambda_{k+1}[(p^k,X^k)]\Bigr)=\Hom(M_k,\Lambda_k).\]
We can then identify
\[\bigcap_{\Lambda_{k+1}}^{\chi(\FF)} \bH^1_{\FF}(K,\bT_k)=\bigcap_{\Lambda_{k}}^{\chi(\FF)} \bH^1_{\FF}(K,\bT_k).\]
Then the map in \eqref{eq:comp_trans_pre} can be expressed as
\[\bigcap_{\Lambda_{k+1}}^{\chi(\FF)} \bH^1_{\FF}(K,\bT_{k+1})\to \bigcap_{\Lambda_{k}}^{\chi(\FF)} \bH^1_{\FF}(K,\bT_k).\]

These maps can be used to construct the inverse limit. In order to compare it with the exterior bidual of the Iwasawa Selmer group, we need to construct the projection maps from the Iwasawa Selmer group.

When Assumption \ref{ass:patched:basic} holds, $\bH^1_{\FF}(K,\bT)$ is a free Iwasawa module by Proposition \ref{prop:iwasawa:selmer_free}. It is then clear that 
\[\left(\bigcap_\Lambda^{\chi(\FF)}\bH^1_{\FF}(K,\bT)\right)\otimes_\Lambda \Lambda_k=\bigcap_{\Lambda_k}^{\chi(\FF)} \Bigl(\bH^1_\FF(K,\bT)\otimes_\Lambda \Lambda_k\Bigr).\]
Using the functoriality of Selmer groups and exterior biduals, we can construct a map
\begin{equation}
\bigcap_\Lambda^{\chi(\FF)}\bH^1_{\FF}(K,\bT)\to \bigcap_{\Lambda_k}^{\chi(\FF)} \bH^1_\FF(K,\bT_k).
\label{eq:bidual_projection}
\end{equation}
These maps are compatible when varying $k$, so they induce a map to the inverse limit
\begin{equation}
\bigcap_\Lambda^{\chi(\FF)}\bH^1_{\FF}(K,\bT)\to \varprojlim_{k\in \N}\bigcap_{\Lambda_k}^{\chi(\FF)} \bH^1_\FF(K,\bT_k).
\label{eq:proj_from_infty}
\end{equation}
We can show that this map is, in fact, an isomorphism.

\begin{lemma}
 Let $\FF$ be a cartesian Selmer structure on $T$ such that $\bH^1_\FF(K,T)$ is a free $R$-module. Then the following map is an isomorphism:
 \[\bigcap_R^{\chi(\FF)}\bH^1_{\FF}(K,T)\cong \varprojlim_{k\in \N}\left(\bigcap_{R_k}^{\chi(\FF)} \bH^1_\FF(K,T_k)\right).\]
\label{lem:proj_bidual_iso_limit}
\end{lemma}

\begin{proof}
Since $\bH^1_{\FF}(K,T)$ is a free, finitely generated $R$-module, there is an isomorphism
\[\bigcap_R^{\chi(\FF)}\bH^1_{\FF}(K,T)\cong\varprojlim_{k\in \N}\left(\bigcap_{R_k}^{\chi(\FF)} \bH^1_\FF(K,T)\otimes_R R_k\right).\]
By Proposition \ref{prop:patched:selmer_quotient_higherk} and the left-exactness of the exterior bidual and the inverse limit, the following map is injective:
\[\varprojlim_{k\in \N}\left(\bigcap_{R_k}^{\chi(\FF)} \bH^1_\FF(K,T)\otimes_R R_k\right)\hookrightarrow\varprojlim_{k\in \N}\left(\bigcap_{R_k}^{\chi(\FF)} \bH^1_\FF(K,T_k)\right).\]

For the surjectivity, let
\[(c_k)_{k\in \N}\in \varprojlim_{k\in \N}\left(\bigcap_{R_k}^{\chi(\FF)} \bH^1_\FF(K,T_k)\right).\]
Then, for every $k\in \N$ and every $j\geq k$,
\[c_k\in \Im\left( \bigcap_{R_j}^{\chi(\FF)} \bH^1_\FF(K,T_j)\to \bigcap_{R_k}^{\chi(\FF)} \bH^1_\FF(K,T_k)\right).\]
Since the exterior bidual preserves surjections,
\[c_k\in \bigcap_{R_k}^{\chi(\FF)}\biggl(\Im\Bigl(\bH^1_\FF(K,T_j)\to\bH^1_\FF(K,T_k)\Bigr)\biggr)\subset \bigcap_{R_k}^{\chi(\FF)}\bH^1_\FF(K,T_k).\]
Since $\bH^1_{\FF}(K,T)=\varprojlim_{k\in \N}\bH^1_{\FF}(K,T_k)$ by Proposition \ref{prop:selmer_profinite}, then
\[c_k\in \bigcap_{R_k}^{\chi(\FF)} \biggl(\Im\Bigl(\bH^1_\FF(K,T)\to\bH^1_\FF(K,T_k)\Bigr)\biggr),\]
which, using the compactness of $\bH^1_\FF(K,T)$, completes the proof of the surjectivity.
\end{proof}

Let $\NN$ be the set of square-free products of Kolyvagin ultraprimes for $\bT$. Since $\bH^1_{\FF(\kn)}(K,\bT)$ is a free, finitely generated Iwasawa module by Proposition \ref{prop:iwasawa:selmer_free}, there is a canonical identification
\[\bigwedge^{\chi(\FF)} \bH^1_{\FF(\kn)}(K,\bT)\cong \bigcap^{\chi(\FF)} \bH^1_{\FF(\kn)}(K,\bT).\]

Note that every $\kn\in \NN$ is a square-free product of Kolyvagin primes for every $\bT_k$. The combination of Theorem \ref{th:comp_ultra_kol} and Lemma \ref{lem:proj_bidual_iso_limit} allows the construction of a map 
\begin{equation}
\bKS(\bT,\FF)\to \overline{\KS}(\bT,\FF)=\varprojlim_{k\in \N} \KS(\bT_k,\FF),
\label{eq:kol_limit}
\end{equation}
where the inverse limit is constructed \cite[Definition 5.24]{BurnsSakamotoSano2}. The above map is, indeed, an isomorphism.

\begin{theorem}
Let $\bT$ and $\FF$ be as in Assumption \ref{ass:classic}. Then there is a canonical identification
\[\bKS(\bT,\FF)\cong \overline{\KS}(\bT,\FF).\]
\label{th:ultrakol_limit}
\end{theorem}

\begin{proof}
By construction and Lemma \ref{lem:proj_bidual_iso_limit}, the map in \eqref{eq:kol_limit} is injective.

Since both modules are free Iwasawa modules of rank one, in order to show the surjectivity, it is enough to show that the map
\[\bKS(\bT,\FF)\to \bKS(\bT_1,\FF)\] 
is surjective. By Proposition \ref{prop:core_vertex_existence}, there exists a core vertex $\kc\in \NN$ and, by Proposition \ref{prop:patched:selmer_torsion}, 
\[\bH^1_{\FF^*(\kc)}(K,\bT_1^*)=0.\]
By Theorem \ref{th:kol_core_projection} and \cite[Theorem 5.20]{BurnsSakamotoSano2}, adapted to consider ultraprimes, the following diagram is commutative with the vertical maps being isomorphisms:
\[\xymatrix{
    \bKS(\bT,\FF)\ar[r] \ar[d]^{\cong} & \bKS(\bT_1,\FF)\ar[d]^{\cong}\\
    \bH^1_{\FF^*(\kc)}(K,\bT)\ar[r] & \bH^1_{\FF^*(\kc)}(K,\bT_1).
}\]
Since the Selmer groups $\bH^1_{\FF^*(\kc)}(K,\bT)$ and $\bH^1_{\FF^*(\kc)}(K,\bT_1)$ are free $\Lambda$ and $\Lambda/\m$ modules of rank $\chi(\FF)$, respectively, by Proposition \ref{prop:core_inv}, then Proposition \ref{prop:iwasawa:ass_equiv} implies that the bottom map is surjective. Hence the top horizontal map 
\[\bKS(\bT,\FF)\twoheadrightarrow \bKS(\bT_1,\FF)\]
is surjective as well.
\end{proof}

Let $\kappa^{(k)}\in \KS(\bT_k,\FF)$. In \cite[Definition 5.1]{BurnsSakamotoSano2}, the following ideals were constructed:
\[\Theta_i\Bigl(\kappa^{(k)}\Bigr)=\sum_{n\in (\NN_\cl)_i} \ind(\kappa^{(k)}_n),\]
where $(\NN_\cl)_{i}$ is the set of square-free products of exactly $i$ Kolyvagin primes. 

By Theorem \ref{th:comp_ultra_kol}, we can view $\kappa^{(k)}$ as an element in $\bKS(\bT_k,\FF)$ and consider the ideals
\[\bTheta_i\Bigl(\kappa^{(k)}\Bigr)=\sum_{\kn\in (\NN_u)_i} \ind(\kappa^{(k)}_\kn)\]
where, similarly, $(\NN_u)_i$ denotes the square-free product of exactly $i$ Kolyvagin ultraprimes. It will be shown in Corollary \ref{cor:theta_comparison} below that both ideals $\Theta_i\Bigl(\kappa^{(k)}\Bigr)$ and $\bTheta_i\Bigl(\kappa^{(k)}\Bigr)$ coincide.

For every $\overline{\kappa}=\Bigl(\kappa^{(k)}\Bigr)\in \overline{\KS}(\bT,\FF)$, the argument in \cite[Definition 5.24]{BurnsSakamotoSano2} can be used to construct ideals
\[\Theta_i\Bigl(\overline{\kappa}\Bigr)=\varprojlim_{k\in \N} \Theta_i\Bigl(\kappa^{(k)}\Bigr) \subset \Lambda.\]

We will show that, under the identification in Theorem \ref{th:ultrakol_limit}, the constructions of theta ideals from Kolyvagin systems in $\bKS(\bT,\FF)$ and $\overline{\KS}(\bT,\FF)$ are equivalent.
\begin{theorem}
Let $\FF$ be a Selmer structure satisfying Assumption \ref{ass:classic}, let $\kappa\in \bKS(\bT,\FF)$ and let $\overline{\kappa}=\Bigl(\kappa^{(k)}\Bigr)_{k\in \N}\in \overline{\KS}(\bT,\FF)$ be its identification under the isomorphism in Theorem \ref{th:ultrakol_limit}. Then 
\[\bTheta_i(\kappa)=\Theta_i\Bigl(\overline{\kappa}\Bigr).\]
\label{th:ultrakol_theta_limit}
\end{theorem}

The rest of this section is dedicated to the proof of Theorem \ref{th:ultrakol_theta_limit}. We start showing the equality between the notion of theta ideals for classical and ultra Kolyvagin systems with finite coefficients.

\begin{lemma}
Assume that $\bT_k$ and $\FF$ satisfy Assumption \ref{ass:selfinj} and let $\kappa^{(k)}\in \bKS(\bT_k,\FF)$ and let $\kn=(n_i)_{i\in \N}\in \NN(T_k)$. For $\UU$-many $i$, the following equality holds: 
\[\ind\left(\kappa^{(k)}_\kn, \bigcap^{\chi(\FF)}_{R_k}H^1_{\FF(\kn)}(K,T_k)\right)=\ind\left(\kappa^{(k)}_{n_i}, \bigcap^{\chi(\FF)}_{R_k}\bH^1_{\FF(n_i)}(K,T_k)\right).\]
\label{lem:index_kappa_umany}
\end{lemma}

\begin{proof}
Recall from the proof of Theorem \ref{th:comp_ultra_kol} that 
\[\bigcap^{\chi(\FF)}\bH^1_{\FF(\kn)}(K,T_k)=\UU\left(\bigcap^{\chi(\FF)}H^1_{\FF(n_i)}(K,T_k)\right).\]
Then $\kappa_{\kn}$ can be represented by the sequence $\Bigl(\kappa_{n_i}\Bigr)$ in this ultraproduct. By Corollary \ref{cor:ultrafilter_finite_modules}, for $\UU$-many $i$, there is an isomorphism
\[\bigcap^{\chi(\FF)}\bH^1_{\FF(\kn)}(K,T_k)\cong \bigcap^{\chi(\FF)}H^1_{\FF(n_i)}(K,T_k)\]
identifying $\kappa_{\kn}$ with $\kappa_{n_i}$, which completes the proof of the lemma.
\end{proof}

\begin{corollary}
Assume that $\bT_k$ and $\FF$ satisfy Assumption \ref{ass:selfinj} and let $\kappa^{(k)}\in \KS(\bT_k,\FF)=\bKS(\bT_k,\FF)$. Then 
\[\Theta_i\Bigl(\kappa^{(k)}\Bigr)=\bTheta_i\Bigl(\kappa^{(k)}\Bigr)\]
\label{cor:theta_comparison}
\end{corollary}

\begin{proof}
It is clear that 
\[\Theta_i(\kappa^{(k)})\subset \bTheta_i(\kappa^{(k)})\]
since the constant Kolyvagin ultraprimes are a subset of the Kolyvagin ultraprimes. Conversely, we need to show that 
\[\ind\biggl(\kappa_\kn, \bH^1_{\FF(n)}(K,T)\biggr)\subset \Theta_i(\kappa^{(k)}).\]
However, it follows from Lemma \ref{lem:index_kappa_umany}, since there exists a square free product of $i$ classical Kolyvagin primes $n$ such that 
\[\ind(\kappa_\kn)=\ind(\kappa_n).\qedhere\]
\end{proof}

The second part of the proof consists on the study of the preservation of indices of the Kolyvagin classes under the isomorphism in \eqref{eq:proj_from_infty}.

\begin{lemma}
Let $\FF$ be a Selmer structure such that the Selmer group $\bH^1_{\FF}(K,T)$ is a free $R$-module. Let $\alpha\in \bigcap^{\chi(\FF)}_R \bH^1_{\FF}(K,T)$ and let $\alpha^{(k)}\in \bigcap^{\chi(\FF)}_{R_k} \bH^1_{\FF}(K,T_k)$ be its image under the projection map in \eqref{eq:bidual_projection}. Then 
\[\ind\left(\alpha,\bigcap^{\chi(\FF)}_\Lambda \bH^1_{\FF}(K,\bT)\right)+(p^k,X^k)=\ind\left(\alpha^{(k)}, \bigcap^{\chi(\FF)}_{\Lambda_k} \bH^1_{\FF}(K,\bT_k)\right),\]
where we are identifying the ideals of $\Lambda_k$ with the ideals of $\Lambda$ containing $(p^k,X^k)$.
\label{lem:bidual_proj_index}
\end{lemma}

\begin{proof}
Since $\bH^1_{\FF}(K,T)$ is a free module, we can consider the class 
\[\alpha+(p^k,X^k)\in \left(\bigcap^{\chi(\FF)}_\Lambda \bH^1_{\FF}(K,\bT)\right)\otimes_\Lambda \Lambda_k=\bigcap^{\chi(\FF)}_{\Lambda_k} \Bigl(\bH^1_{\FF}(K,\bT)\otimes_\Lambda \Lambda_k\Bigr).\]
By construction, it is clear that
\[\ind(\alpha)+(p^k,X^k)=\ind\Bigl(\alpha+(p^k,X^k)\Bigr).\]
Then $\alpha^{(k)}$ is the image of $\alpha+(p^k,X^k)$ under the following map, which is injective by Proposition \ref{prop:patched:selmer_quotient_higherk} and the left-exactness of the exterior bidual:
\[\bigcap^{\chi(\FF)}_{\Lambda_k} \Bigl(\bH^1_{\FF}(K,\bT)\otimes_\Lambda \Lambda_k\Bigr)\hookrightarrow  \bigcap^{\chi(\FF)}_{\Lambda_k} \bH^1_{\FF}(K,\bT_k).\]
Since the domain is a free $\Lambda_k$-module, we have that
\[\ind\Bigl(\alpha^{(k)}\Bigr)=\ind\Bigl(\alpha+(p^k,X^k)\Bigr)=\ind(\alpha)+(p^k,X^k).\qedhere\]
\end{proof}

This result proves an inclusion between theta ideals. Since $\NN_i(\bT)\subset \NN_i(\bT_k)$, the following corollary holds.
\begin{corollary}
Let $\FF$ be a Selmer structure defined on $\bT$ satisfying Assumption \ref{ass:classic} and let $\kappa\in \bKS(T,\FF)$ be an ultra Kolyvagin system and let $\kappa^{(k)}\in \bKS(T_k,\FF)$ be its projection under the identification in Theorem \ref{th:ultrakol_limit}. Then 
\[\bTheta_i(\kappa)\subset \bTheta_i(\kappa^{(k)}).\]
\label{cor:dvr:theta_proj_finite}
\end{corollary}

We can now conclude the proof of Theorem \ref{th:ultrakol_theta_limit}.

\begin{proof}[Proof of Theorem \ref{th:ultrakol_theta_limit}]
By Corollary \ref{cor:dvr:theta_proj_finite} and taking the limit, we can deduce that 
\[\bTheta_i(\kappa)\subset \bTheta_i(\overline\kappa).\]

Conversely, let 
\[a=(a_k)\in \varprojlim_{k\in \N} \bTheta_i(\kappa^{(k)})\subset \varprojlim_{k\in \N}\Lambda_k=\Lambda.\]
By Corollary \ref{cor:theta_comparison}, and the construction of theta ideals, there are square-free products $n_{k,1},\ldots, n_{k,s_k}\in \NN_i(\bT_k)$ of exactly $i$ Kolyvagin primes for $\bT_k$ such that 
\[a_k\in \sum_{i=1}^{s_k}\ind\Bigl(\kappa^{(k)}_{n_{k,a}}\Bigr).\]

In an attempt to simplify the notation, we denote $n_{k,i}:=n_{k,s_k}$ for every $i\geq s_k$. 

For every $a\in \N$, let $\kn_a$ be the square-free product of ultraprimes represented by the sequence $(n_{k,a})_{k\in \N}$. By Lemmas \ref{lem:index_kappa_umany} and \ref{lem:bidual_proj_index}, for every fixed $a$ and $\UU$-many $k$, we have that 
\[\ind(\kappa_{\kn_a})+(p^k,X^k)=\ind(\kappa^{(k)}_{n_{k,a}}),\]
where we are again using the equivalence of ideals of $\Lambda_k$ and ideals of $\Lambda$ containing $(p^k,X^k)$. By the compatibility relation of the different Kolyvagin systems in the tower, we have that, for every $j\geq k$,
\[a_k+(p^k,X^k)\subset \sum_{a=1}^{\infty}\ind(\kappa^{(j)}_{n_{j,a}})+(p^k,X^k).\]
 That implies that 
\[a_k+(p^k,X^k)\subset \sum_{a=1}^{\infty}\ind(\kappa_{\kn_a})+(p^k,X^k)\subset \bTheta_i(\kappa)+(p^k,X^k).\]
Since $\Lambda$ is compact in the profinite topology, every (finitely generated) ideal is also compact, so closed. Since $(p^k,X^k)$ is a basis of neighbourhoods of the identity and 
\[\Bigl(a+(p^k,X^k)\Bigr)\cap \bTheta_i(\kappa)\neq \emptyset\ \forall i\in \N.\]
Then the closure of $\bTheta_i(\kappa)$ implies that it contains $a$.
\end{proof}

We conclude this section with a proof for the question in \cite[Remark 5.3]{BurnsSakamotoSano2} for the Galois representation $\bT_k$.

\begin{corollary}
Suppose that $\bT$ and $\FF$ satisfy Assumption \ref{ass:classic} and consider a primitive Kolyvagin system $\kappa^{(k)}\in \KS(\bT_k,\FF)$. Then 
\[\Theta_i(\kappa^{(k)})=\Fitt^i_{\Lambda_k}\Bigl(\bH^1_{\FF^*}(K,(\bT_k)^*)^\vee\Bigr).\]
\label{cor:theta_fitting_selfinj}
\end{corollary}

\begin{proof}
By Theorems \ref{th:comp_ultra_kol} and \ref{th:ultrakol_limit}, there is a primitive Kolyvagin system $\kappa\in \bKS(\bT,\FF)$ that projects to $\kappa^{(k)}$ under the map
\[\bKS(\bT,\FF)\cong\overline{\KS}(\bT,\FF)=\biggl(\varprojlim_{j\in \N}\KS(\bT_j,\FF)\biggr)\to\KS(\bT_k,\FF).\]
By Proposition \ref{prop:patched:selmer_torsion}, Theorem \ref{th:iwasawa:kolyvagin_theta_fitting} and Corollary \ref{cor:dvr:theta_proj_finite}, we have that
\[\Fitt^i_{\Lambda_k}\Bigl(\bH^1_{\FF^*}(K,(\bT_k)^*)^\vee\Bigr)=\Fitt^i_{\Lambda}\Bigl(\bH^1_{\FF^*}(K,\bT^*)^\vee\Bigr)\otimes \Lambda_k\subset \bTheta_i(\kappa)\otimes \Lambda_k\subset \Theta_i\Bigl(\kappa^{(k)}\Bigr).\]
The other inclusion is proven in \cite[Theorem 5.2]{BurnsSakamotoSano2}.
\end{proof}

\begin{remark}
The argument in Corollary \ref{cor:theta_fitting_selfinj} is generalisable to other self-injective quotients of $\Lambda$.
\end{remark}

\subsection{Iwasawa Selmer groups}

Theorem \ref{th:iwasawa:kolyvagin_theta_fitting} proved that, when $\FF$ is a cartesian Selmer structure, the Fitting ideal $\Fitt^i_\Lambda\Bigl(\bH^1_{\FF^*}(K,\bT^*)^\vee\Bigr)$ is contained, with finite index, in the ideal $\Theta_i(\kappa)$ associated with a primitive Kolyvagin system $\kappa\in \KS(\bT,\FF)$. The aim of this section is to show that this inclusion is indeed an equality when $\FF$ satisfies Assumption \ref{ass:classic}.

\begin{theorem}
Let $\bT$ be an Iwasawa Galois representation and let $\FF$ be a Selmer structure defined on $\bT$ such that Assumption \ref{ass:classic} holds and $\chi(\FF)\geq 1$. If $\kappa\in \bKS(\bT,\FF)$ is a primitive ultra Kolyvagin system, then
\[\bTheta_i(\kappa)=\Fitt^i_\Lambda\Bigl(\bH^1_{\FF^*}(K,\bT^*)^\vee\Bigr)\ \forall i\in \Z_{\geq 0}.\]
\label{th:iwasawa:ultrakol:theta_fitting_equal}
\end{theorem}

\begin{proof}
Fix $i\in \Z_{\geq 0}$. Theorem \ref{th:iwasawa:kolyvagin_theta_fitting} implies that
\[\Fitt^i_\Lambda\Bigl(\bH^1_{\FF^*}(K,\bT^*)^\vee\Bigr)\subset \bTheta_i(\kappa).\]
Conversely, for every $k\in \N$, let $\kappa^{(k)}\in \bKS(\bT_k,\FF)$ the image of $\kappa$ under the projection map
\[\bKS(\bT,\FF)\cong\overline{\KS}(\bT,\FF)=\biggl(\varprojlim_{j\in \N}\KS(\bT_j,\FF)\biggr)\to\KS(\bT_k,\FF).\]
By Corollary \ref{cor:theta_comparison} and \cite[Theorem 5.2]{BurnsSakamotoSano2},
\[\bTheta_i\Bigl(\kappa^{(k)}\Bigr)=\Theta_i\Bigl(\kappa^{(k)}\Bigr)\subset \Fitt^i_{\Lambda_k} \Bigl(\bH^1_{\FF^*}(K,(\bT_k)^*)^\vee\Bigr).\]
By Propositions \ref{prop:patched:selmer_torsion} and \ref{prop:fitting_tensor},
\[\Fitt^i_{\Lambda_k} \Bigl(\bH^1_{\FF^*}(K,(\bT_k)^*)^\vee\Bigr)=\Fitt^i_\Lambda\Bigl(\bH^1_{\FF^*}(K,\bT^*)^\vee\Bigr)+(p^k,X^k),\]
where we are identifying ideals of $\Lambda_k$ with ideals of $\Lambda$ that contain $(p^k,X^k)$.
Corollary \ref{cor:dvr:theta_proj_finite} implies that 
\[\bTheta_i(\kappa)+(p^k,X^k)\subset \bTheta_i(\kappa^{(k)}).\]
Taking all of the above into account, we can conclude that 
\[\bTheta_i(\kappa)+(p^k,X^k)\subset \Fitt^i_\Lambda\Bigl(\bH^1_{\FF^*}(K,\bT^*)^\vee\Bigr)+(p^k,X^k).\]
Since that holds for all $k\in \N$, we obtain that 
\[\bTheta_i(\kappa)\subset \Fitt^i_\Lambda\Bigl(\bH^1_{\FF^*}(K,\bT^*)^\vee\Bigr).\qedhere\]
\end{proof}

The following is a generalisation of \cite[Proposition 2.9]{Kur2014}, which assumed the vanishing of the Iwasawa $\mu$-invariant associated with the Selmer group.

\begin{theorem}
Let $\FF$ be a Selmer structure satisfying Assumption \ref{ass:classic}. Then $\bH^1_{\FF^*}(K,\bT^*)^\vee$ contains no finite Iwasawa submodules.
\label{th:iwasawa:nofinsub}
\end{theorem}

\begin{proof}
We start assuming that $\bH^1_{\FF^*}(K,\bT^*)^\vee$ is torsion, what implies that $\chi(\FF)\geq 0$. 

If $\chi(\FF)\geq1$, pick a primitive Kolyvagin system $\kappa\in \bKS(\bT,\FF)$. By Theorem \ref{th:iwasawa:ultrakol:theta_fitting_equal},
\[\Fitt^0_\Lambda\Bigl(\bH^1_{\FF^*}(K,T^*)\Bigr)^\vee=\bTheta_0(\kappa)=\ind(\delta_1).\]
By Proposition \ref{prop:iwasawa:selmer_free}, $\bH^1_{\FF}(K,\bT)$ is a free module of rank $\chi(\FF)$. Then
\[\bigwedge^{\chi(\FF)}\bH^1_{\FF}(K,\bT)\cong \Lambda.\]
Therefore,
\[\ind\left(\delta_1,\bigwedge^{\chi(\FF)}\bH^1_{\FF}(K,\bT)\right)\]
is a principal ideal, which, by Proposition \ref{prop:fitting_ppal_finsub}, implies that $\bH^1_{\FF^*}(K,\bT^*)^\vee$ contains no finite submodules.

When $\chi(\FF)=0$, we can use Proposition \ref{prop:weak_core_vertex_existence} to ensure the existence of a weak core vertex $\kc\in \NN$. Since $\bH^1_{\FF^*}(K,\bT^*)^\vee$ is torsion, the assumption $\chi(\FF)=0$ and Proposition \ref{prop:iwasawa:selmer_free} implies that $\bH^1_{\FF}(K,\bT)=0$.
Then Proposition \ref{prop:patched:global_duality} induces an exact sequence 
\[\xymatrix{0\ar[r] & \bH^1_{\FF^\kc}(K,\bT)\ar[r] & \displaystyle\bigoplus_{\ku\mid \kc}\bH^1_{\s}(K_\ku,T)\ar[r] & \bH^1_{\FF^*}(K,\bT^*)^\vee\ar[r] &0.}\]
Hence $\bH^1_{\FF^*}(K,\bT^*)^\vee$ has projective dimension one, so it contains no finite submodules due to \cite[Proposition 5.3.19]{NSW}.

Now assume that $\bH^1_{\FF^*}(K,\bT^*)^\vee$ has positive rank $r$. By Proposition \ref{prop:patched:cheb_rank_red}, we can find some $\kn\in \NN$, divisible by exactly $r$ Kolyvagin ultraprimes, satisfying that the Selmer group $\bH^1_{(\FF^*)_\kn}(K,\bT^*)^\vee$ is torsion. By Proposition \ref{prop:patched:global_duality}, there is an exact sequence
\[\xymatrix{\displaystyle \bigoplus_{\ku\mid \kn} \bH^1_{\s}(K,\bT)\ar^{\varphi}[r] & \bH^1_{\FF^*}(K,\bT^*)^\vee\ar[r] & \bH^1_{(\FF^*)_\kn}(K,\bT^*)^\vee \ar[r] &0.}\]
Since $\kn\in\NN$, the first term of the exact sequence is isomorphic to $\Lambda^r$. Then the first part of this proof implies that $\bH^1_{(\FF^*)_\kn}(K,\bT^*)^\vee$. The image of the map $\varphi$ has rank $r$, so it contains no torsion elements, which implies that $\bH^1_{\FF^*}(K,\bT^*)^\vee$ contains no finite sumbodules.
\end{proof}

\section{Applications}
\label{sec:app}

\subsection{Elliptic curves}
\label{sec:app:ec}

Let $E/\Q$ be an elliptic curve defined over the rationals. Assume that $\bT$ can be obtained as the tensor product $T_p E\otimes \Z_p[[\Gal(\Q_\infty/\Q)]]$, where $\Q_\infty/\Q$ is the cyclotomic $\Z_p$-extension.

For this representation, Kato constructed in \cite{Kato} an Euler system that, in combination with \cite[Theorem 5.3.3]{MazurRubin}, produces a Kolyvagin system $\overline \kappa_{\Kato}\in \overline{\KS}(\bT,\FF^p)$, where $\FF^p$ denotes the Selmer structure relaxed at $p$, in which the local conditions are the finite cohomology groups for finite primes not above $p$ and the full local cohomology groups for $p$-adic and archimedean primes (see \cite[Definition 5.3.2]{MazurRubin}). Note that Proposition \ref{prop:core_rank} implies that $\chi(\FF^p)=1$.

By the isomorphism in Theorem \ref{th:ultrakol_limit}, we can see it as an ultra Kolyvagin system $\kappa_{\Kato}\in \bKS(\bT,\FF^p)$. Then Theorem \ref{th:iwasawa:kolyvagin_theta_fitting} can be used to compute the structure of the fine Selmer group.
\begin{theorem}
Let $E/\Q$ be an elliptic curve defined over the rationals, let $\bT=T_pE\otimes \Z_p[[\Gal(\Q_\infty/\Q)]]$ and let $\kappa_{\Kato}\in\bKS(\bT,\FF^p)$ be the ultra Kolyvagin system obtained from Kato's Euler system. Assume that:
\begin{itemize}
\item \namedlabel{Eord}{(E1)} E has good ordinary reduction at $p\geq 5$.
\item \namedlabel{Esur}{(E2)} $G_K$ acts surjectively on $T_pE$. 
\item \namedlabel{Etam}{(E3)} The Tamagawa numbers of $E$ are prime to $p$.
\item \namedlabel{Eanom}{(E4)} $E(\Q_{\infty,\p})[p^\infty]$ is a $p$-divisible group, where $\p$ is the unique prime of $\Q_\infty$ above $p$.
\item \namedlabel{EIMC}{(E5)} The Iwasawa main conjecture holds for $E/\Q$.
\end{itemize} 
Then 
\[\bTheta_i(\kappa_{\Kato})=\Fitt^i_\Lambda\Bigl(\bH^1_{\FF^*}(K,\bT^*)^\vee\Bigr)\ \forall i\in \Z_{\geq 0}.\]
\label{th:kato}
\end{theorem}

\begin{proof}
The surjectivity assumption implies that $\bT$ satisfies Assumption \ref{ass:patched:basic}. In addition, the second and third assumption implies Assumption \ref{ass:iwasawa:local_no_fin_sub} (see Remark \ref{rem:iwasawa:local_nofinsub_divisible}). 

Note that $\Sigma_{\FF^p}$ contains only $p$, $\infty$, and the primes dividing the conductor of $E$. When $\ell\in \Sigma_{\FF^*}$, \cite[Lemma 5.3.1]{MazurRubin} implies that
\[\bH^1_{\FF^p}(\Q_\ell,\bT)=\bH^1(\Q_\ell,\bT),\]
so $\FF^p$ is cartesian. The propagation of $\FF$ to $\bT/(X)$ is also cartesian by Proposition \ref{prop:iwasawa:free_cartesian_propagated}.

Then the result follows from Theorem \ref{th:iwasawa:ultrakol:theta_fitting_equal}, since the primitivity of Kato's Kolyvagin system is equivalent to the Iwasawa main conjecture.
\end{proof}

\begin{remark}
The case in Theorem \ref{th:kato} was not covered in \cite[Theorem 10.14]{Ohshita} since the Tate module of an elliptic curve is its own Cartier dual.
\end{remark}

\subsection{Rankin-Selberg convolutions}
\label{sec:app:rs}

Let $f$ and $g$ modular forms of weights $k$ and $\ell$, respectively, and levels $N_f$ and $N_g$, respectively. Denote by $\varepsilon_{f}$ and $\varepsilon_g$ the Nebentypus characters of $f$ and $g$, respectively. Consider the Galois representation associated with its Rankin-Selberg convolution,
\[T:=T_f\otimes T_g,\]
where $T_f$ and $T_g$ are the Galois representations associated with $f$ and $g$. As usual, denote
\[\bT:=T\otimes \Z_p[[\Gal(\Q_\infty/\Q)]].\]

In \cite{LeiLoefflerZerbes}, Kato's construction was generalised by Lei, Loeffler and Zerbes to build an Euler system for the representation $T$, that will be denoted by $\bz_{\LLZ}$.

\begin{remark}
In this case, we cannot obtain a Kolyvagin system in $\bKS(\bT, \FF^p)$ using \cite[Theorem 5.3.3]{MazurRubin} and Theorem \ref{th:ultrakol_limit} since $\chi(\FF^p)=2$. Note that the fact $\chi(\FF^p)=2$ can be seen from Proposition \ref{prop:core_rank}, since $d_-(T)=2$ because both $T_f$ and $T_g$ are odd representations and all the other terms in the formula vanish. It is important to remark that, since
\[H^0(\Q_\ell,\bT^*)=H^0((\Q_\ell)_\infty,T^*),\]
is a torsion Iwasawa module for every finite prime $\ell$, where $(\Q_\ell)_\infty$ is the cyclotomic $\Z_p$-extension of $\Q_\ell$, only $p$ and $\infty$ contribute to the formula in \eqref{eq:core_rank}.

\end{remark}

However, it is possible, under the following assumption, to use $\bz_{\LLZ}$ to obtain a Kolyvagin system for a different Selmer structure. In order to do that, we need to make the following assumption on $f$ and $g$.
\begin{namedass}{(RS)}
Assume that
\begin{itemize}
\item \namedlabel{RSord}{(RS1)} Both $f$ and $g$ are ordinary at $p\geq 5$. In addition, assume the weights of $f$ and $g$ are at least $2$.
\item \namedlabel{RSBI}{(RS2)} The action of $G_\Q$ on $T_f\otimes T_g$ satisfies Assumption \ref{ass:patched:basic}.
\item \namedlabel{RSTam}{(RS3)} The cohomology groups $H^0((K_\infty)_\ell,\bT^*)$ are $p$-divisible when $\ell$ is either the prime $p$ or a prime divisor of $N_fN_g$.
\end{itemize}
    \label{ass:RS}
\end{namedass}

\begin{remark}
By \cite[Proposition 4.3.1]{Loeffler_BI}, Assumption \ref{RSBI} will never hold if $\epsilon_f\epsilon_g=1$. However, it is shown in \cite{KingsLoefflerZerbes} that Assumption \ref{RSBI} holds for all but finitely many primes in the case where $k,\ell\geq 2$, at least one of them is odd, $(N_f,N_g)=1$ and neither $f$ nor $g$ is of CM-type. The reader is referred to \cite{Loeffler_BI} and \cite{Studnia} for more involved situations in which Assumption \ref{RSBI} is guaranteed to hold.
\end{remark}

Since $f$ (resp. $g$) is ordinary, $T_f$ (resp. $T_g$) contains a subrepresentation $T_f^+$ (resp. $T_g^+$) of weight $k-1$ (resp. $\ell-1$). Then the quotients $T_f/T_f^+$ and $T_g/T_g^+$ have Hodge-Tate weight $0$.

Under this assumption, we can consider the Panchiskin subrepresentation
\[T^+:=T_f^+\otimes T_g+T_f\otimes T_g^+.\]
Note that $T^+$ has Hodge-Tate weights $\{k-1,\ell-1,k+\ell-2\}$ while $T/T^+$ has Hodge-Tate $0$. We can consider the Iwasawa analogue
\[\bT^+:=T^+\otimes \Z_p[[\Gal(\Q_\infty/\Q)]]\subset \bT.\]
Moreover, denote the quotients
\[\begin{array}{cc}
    \bT^{-}:=\bT/\bT^+,\ \ & T^-:=T/T^+,
\end{array}\]
which are free Iwasawa and p-adic modules, respectively, of rank one.

We can define the following Selmer structure of Greenberg-type, which was first defined in \cite{Greenberg_reps}.
\begin{definition}
The Greenberg Selmer structure $\FGR$ is defined by the local conditions\footnote{There is no need to define a local condition at the archimedean prime, since we are assuming that $p$ is odd.}
\[\bH^1_{\FGR}(\Q_\ku,\bT):=\left\{
    \begin{aligned}
    &\bH^1_{\FGR}(\Q_\ku,\bT)=\bH^1_\f(\Q_\ku,\bT),\ \textrm{if $\ku\neq p,$}\\
    &\bH^1_{\FGR}(\Q_p,\bT)=\ker\biggl(\bH^1(\Q_p,\bT)\to \bH^1(\II_p,\bT^-)\biggr).
    \end{aligned}\right.\]
\end{definition}

In order to compute the local condition at $p$, we make the following further assumptions.
\begin{namedass}{(RS)}
    Assume that
\begin{itemize}
\item \namedlabel{RSanom}{(RS4)} $H^0(\Q_p,\bT^-\otimes \Lambda/\m)=0$,
\item \namedlabel{RSquot}{(RS5)} $T^+$ does not admit any non-zero quotient in which $G_{\Q_p}$ acts by the cyclotomic character.
\end{itemize}
\end{namedass}

\begin{remark}
Under Assumption \ref{RSord}, the Weil bound implies that Assumption \ref{RSquot} will hold provided that at least one of the weights is at least $3$.
\end{remark}

\begin{proposition}
    Under Assumption \ref{RSanom},
    \[\bH^1_\FGR(\Q_p,\bT)=\Im\biggl(\bH^1(\Q_p,\bT^+)\to \bH^1(\Q_p,\bT)\biggr).\]
    \label{prop:greenberg_image}
\end{proposition}

\begin{proof}
Since 
\[\Im\biggl(\bH^1(\Q_p,\bT^+)\to \bH^1(\Q_p,\bT)\biggr)=\ker\biggl(\bH^1(\Q_p,\bT)\to \bH^1(\Q_p,\bT^-)\biggr),\]
we are only left to show that the map
\[\bH^1(\Q_p,\bT^-)\to \bH^1(\II_p,\bT^-)\]
is injective. By the inflation-restriction sequence, that is equivalent to show that
\[\bH^1\Bigl(\Q_p^{\ur}/\Q_p,\bH^0(\II_p,\bT^-)\Bigr)=0,\]
where $\Q_p^\ur$ denotes the maximal unramified extension of $\Q_p$. If $\bT^-$ is ramified, then it necessarily holds that $\bH^0(\II_p,\bT^-)=0$. Otherwise, we have that $\bH^0(\II_p,\bT^-)=\bT^-$ and, by Assumption \ref{RSanom},
\[\bH^0(\Q_p^{\ur}/\Q_p,\bT^-)=0.\]
In the second case, note that there is an exact sequence
\[\xymatrix{0\ar[r] & \bH^0(\Q_p^\ur/\Q_p,\bT^-)\ar[r] & \bT^-\ar[r]^{\gamma-1} & \bT^-\ar[r] &\bH^1(\Q_p^\ur/\Q_p,\bT^-)\ar[r] & 0,}\]
where $\gamma$ denotes the Frobenius element in $\Gal(\Q_p^{\ur}/\Q_p)$. By the additivity of the rank, $\bH^1(\Q_p,\bT^-)$ is a torsion $\Lambda$-module. However, Assumption \ref{RSanom} is stronger than $\bH^0(\Q_p^\ur,\Q_p,\bT^-/f\bT^-)=0$, where $f$ is any element in $\Lambda$. Then we can conclude that
\[\bH^1(\Q_p^\ur/\Q_p,\bT^-)_\tors=0,\]
so this cohomology group necessarily vanishes.
\end{proof}

With the identification in Proposition \ref{prop:greenberg_image}, it is shown within the proof of \cite[Corollary B.1.5]{LeiLoefflerZerbes2} that the Greenberg local condition coincides with the Bloch-Kato condition.

\begin{proposition}(\cite[Corollary B.1.5]{LeiLoefflerZerbes2}) 
    If we assume \ref{RSord}, \ref{RSanom} and \ref{RSquot}, then 
    \[\bH^1_\FBK(\Q_p,\bT)=\bH^1_{\FGR}(\Q_p,\bT).\]
    \label{prop:BK_GR}
\end{proposition}

By \cite[Proposition 6.5.4]{LeiLoefflerZerbes}, the Euler system $\bz_{\LLZ}$ of Lei, Loeffler and Zerbes satisfies the following local conditions: for every number abelian extension $F/\Q$ and every prime $w$ of $F$,
\[\loc_w((\bz_{\LLZ})_F)\in H^1_\f(F_w,T).\]

By \cite[Corollary B.1.5]{LeiLoefflerZerbes2} and \cite[Proposition 12.2.3]{KingsLoefflerZerbes}, the Kolyvagin derivative produces a Kolyvagin system
\[\kappa_{\LLZ}\in \overline{\KS}(\bT,\FBK)=\bKS(\bT,\FBK),\]
where $\FBK$ represents the Bloch-Kato Selmer structure\footnote{This descent process is finer than the one described in \cite{MazurRubin}, since the Kolyvagin system obtained is associated with the Bloch-Kato Selmer structure instead of the canonical one.} and we are using the identification in Theorem \ref{th:ultrakol_limit}.

The following corollary is obtained as an application of Theorem \ref{th:iwasawa:kolyvagin_theta_fitting}.
\begin{theorem}
Under Assumptions \ref{RSord}-\ref{RSquot}, there exists $\delta\in\Lambda$, measuring the divisibility $\kappa_{\LLZ}$, defined independently of $i$ such that
\[\bTheta_i(\kappa_{\LLZ})=\delta\cdot\Fitt^i\Bigl(\bH^1_{(\FBK)^*}(\Q_\infty,\bT^*)^\vee\Bigr)\ \forall i\in \Z_{\geq 0}.\]
Moreover, $\delta$ can be taken equal to $1$ when the Iwasawa main conjecture holds.
    \label{th:llz}
\end{theorem}

\begin{remark}
The Iwasawa main conjecture for the Rankin-Selberg convolution has been proven in some cases in \cite{Wan}.
\end{remark}

\begin{proof}[Proof of Theorem \ref{th:llz}]
Note that Assumption \ref{ass:patched:basic} is guaranteed by \ref{RSBI} and Assumption \ref{ass:iwasawa:local_no_fin_sub}, by \ref{RSTam}. \cite[Lemma 5.3.1 (ii)]{MazurRubin} implies, when $\ell\in \Sigma_{\FF^*}$, that 
\[\bH^1_{\FF^p}(\Q_\ell,\bT)=\bH^1(\Q_\ell,\bT),\]
so these local conditions are cartesian. By Proposition \ref{prop:BK_GR}, 
\begin{equation}
\bH^1_{/\FBK}(\Q_p,\bT)\hookrightarrow \bH^1(\Q_p,\bT^-).
\label{eq:BK_quot_inj}
\end{equation}
As shown in the proof of Proposition \ref{prop:greenberg_image}, Assumption \ref{RSanom} implies that the latter is torsion-free, which implies that the Bloch-Kato condition is cartesian.

Moreover, we can see that the map in \eqref{eq:BK_quot_inj} is an isomorphism. Note that its cokernel is isomorphic to $\bH^2(\Q_p,\bT^+)[X]$. By local duality, there is an isomorphism
\[\bH^2(\Q_p,\bT^+)\cong \bH^0\Bigl(\Q_p,(\bT^+)^*\Bigr)=\bH^0\Bigl((\Q_\infty)_\p,(T^+)^*\Bigr)=0,\]
where $\p$ is the unique prime of $\Q_\infty$ above $p$ and the last equality follows from assumption \ref{RSquot} and the fact that $\Gal((\Q_\infty)_\p/\Q_p)$ is a pro-$p$ group. Hence,
\[\bH^1_{/\FBK}(\Q_p,\bT)\cong\bH^1(\Q_p,\bT^-).\]
The same argument used to show the torsion-freeness of $\bH^1(\Q_p,\bT^-)$ leads to
\[\bH^1(\Q_p,\bT^-/X\bT^-)[p]=0.\]
By \ref{RSanom}, there is an injection
\[\bH^1(\Q_p,\bT^-)\otimes \Lambda/(X)\hookrightarrow\bH^1(\Q_p,\bT^-/X\bT^-),\]
so we can conclude that the first is a free $\Z_p$-module. This, together with the absence torsion, is enough to conclude that $\bH^1(\Q_p,\bT^-)$ is a free Iwasawa module (see the argument in Proposition \ref{prop:iwasawa:selmer_free}). Then Proposition \ref{prop:iwasawa:free_cartesian_propagated} shows that the propagation of $\FBK$ to $\bT/X\bT$ is cartesian.

The Kolyvagin system $\kappa_{\LLZ}$ might not be primitive, but there exist a primitive one $\kappa\in \KS(\bT,\FBK)$ and $\delta\in \Lambda$ such that $\kappa_{\LLZ}=\delta\cdot\kappa$. Then Theorem \ref{th:iwasawa:ultrakol:theta_fitting_equal} implies that
\[\bTheta_i(\kappa_{\LLZ})=\delta\Theta_i(\kappa)=\Fitt^i\Bigl(\bH^1_{(\FBK)^*}(\Q,\bT^*)^\vee\Bigr).\]
The Iwasawa main conjecture is equivalent to the primitivity of $\kappa_{\LLZ}$, meaning that $\delta\in \Lambda^\times$.
\end{proof}

\begingroup

\sloppy
\hbadness=10000
\printbibliography[heading=bibintoc]
\endgroup

\end{document}